\pdfoutput=1

\documentclass[a4paper, 11pt, english]{amsart}
\usepackage{amsmath, amsfonts,amssymb, amsthm}
\usepackage{graphicx}
\usepackage{color}
\usepackage{epsf}
\usepackage{tikz}
\usetikzlibrary{arrows,decorations.pathreplacing, bending,hobby}
\usetikzlibrary{shapes.geometric}
\usetikzlibrary{decorations.markings}
\usepackage[a4paper]{geometry}
\usepackage{bbm}
\usepackage[normalem]{ulem}
\usepackage{cancel}
\usepackage{comment}
\usepackage{changepage}
\usepackage{caption}
\usepackage{subcaption}

\usepackage{mathrsfs}

\usepackage{mathabx}\changenotsign

\newcommand\hcancel[2][black]{\setbox0=\hbox{$#2$}%
\rlap{\raisebox{.45\ht0}{\textcolor{#1}{\rule{\wd0}{0.5pt}}}}#2} 

\usepackage{accents}
\newcommand{\vect}[1]{\accentset{\rightharpoonup}{#1}}

\usepackage{mathdots}
\usepackage[
backend=biber,
firstinits=true, 
natbib=true,
style=numeric,
isbn=false,
maxbibnames=9
]{biblatex}
\usepackage{csquotes}
\DeclareFieldFormat{title}{\mkbibquote{#1}}
\DeclareFieldFormat[misc]{date}{Preprint (#1)}

\renewbibmacro*{doi+eprint+url}{
	\iftoggle{bbx:url}
	{\iffieldundef{doi}{\usebibmacro{url+urldate}}{}}{}
	\newunit\newblock
	\iftoggle{bbx:eprint}
	{\usebibmacro{eprint}}
	{}
	\newunit\newblock
	\iftoggle{bbx:doi}
	{\printfield{doi}}
	{}
}

\renewbibmacro{in:}{}
\DeclareFieldFormat{pages}{#1}

\usepackage[colorlinks=true, linkcolor=blue, filecolor=magenta, urlcolor=cyan]{hyperref}

\usepackage{enumitem}
\newcommand{\customlabel}[2]{#2\def\@currentlabel{#2}\label{#1}}

\def\tmindeg{\delta^{(3)}}

\def\div{\text{div}}

\usepackage{mathtools}
\usepackage{dsfont}
\setlist{itemsep=0pt, topsep=0pt}

\newtheorem{theorem}{Theorem}[section]
\newtheorem{lemma}[theorem]{Lemma}
\newtheorem{claim}[theorem]{Claim}

\newtheorem{conjecture}[theorem]{Conjecture}
\newtheorem{corollary}[theorem]{Corollary}
\newtheorem{question}[theorem]{Question}
\newtheorem{proposition}[theorem]{Proposition}

\numberwithin{subcase}{case}

\newtheorem{definition}[theorem]{Definition}

\newtheorem*{lemma*}{\textbf{Lemma~\ref{lem:coverdownlemma}}}

\numberwithin{equation}{section}

\newcommand{\eps}{\varepsilon}

\def\al#1{}
	\renewcommand{\al}[1]{\footnote{\textbf{AL: }#1}}         

\newenvironment{proofclaim}[1][Proof of claim]{\begin{proof}[#1]}{\end{proof}}

\usepackage{babel}

\allowdisplaybreaks

\title{Cycle decompositions in~$k$-uniform hypergraphs}

\author[A. Lo]{Allan Lo}
\address[A. Lo, S. Piga]{School of Mathematics, University of Birmingham, Edgbaston, Birmingham, B15 2TT, UK}
\email{s.a.lo@bham.ac.uk, s.piga@bham.ac.uk
}
\author[S. Piga]{Sim\'on Piga}
\author[N. Sanhueza-Matamala]{Nicol\'as Sanhueza-Matamala}
\address[N. Sanhueza-Matamala]{Departamento de Ingeniería Matemática, Facultad de Ciencias Físicas
y Matemáticas, Universidad de Concepción, Chile}
\email{nicolas@sanhueza.net}
\thanks{
The research leading to these results was supported by ANID-Chile through the FONDECYT Iniciación Nº11220269 grant (N.~Sanhueza-Matamala) and 
EPSRC, grant no. EP/V002279/1 (A.~Lo and S.~Piga) and EP/V048287/1 (A.~Lo).
There are no additional data beyond that contained within the main manuscript.}

\begin{document}

\begin{abstract}
    We show that $k$-uniform hypergraphs on $n$ vertices whose codegree is at least $(2/3 + o(1))n$ can be decomposed into tight cycles, subject to the trivial divisibility conditions.
    As a corollary, we show those graphs contain tight Euler tours as well.
    In passing, we also investigate decompositions into tight paths.
    
    In addition, we also prove an alternative condition for building absorbers for edge-decompositions of arbitrary $k$-uniform hypergraphs, which should be of independent interest.
\end{abstract}

\maketitle

\section{Introduction}
\label{sec:intro}

Given a $k$-uniform hypergraph~$H$, a \emph{decomposition of $H$} is collection of subgraphs of~$H$ such that every edge is covered exactly once.
If all these subgraphs are isomorphic copies of the same $k$-uniform graph $F$, we say $H$ has an \emph{$F$-decomposition}, and that $H$ is \emph{$F$-decomposable}.
We refer the reader to the survey of Glock, Kühn, and Osthus~\cite{GlockKuhnOsthus2021} for a recent account on extremal aspects of hypergraph decomposition problems.
Here we investigate hypergraph decompositions into tight cycles.

Given $\ell > k \geq 2$, the \emph{$k$-uniform tight cycle of length $\ell$}, denoted by $C^{\smash{(k)}}_\ell$, is the~$k$-graph whose vertices are $\{v_1, \dotsc, v_{\ell} \}$ and its edges are $\{ v_i, v_{i+1}, \dotsc, v_{i + k - 1} \}$ for all~$i \in \{1, \dotsc, \ell\}$, with the subindices understood modulo~$\ell$.
Given a vertex set $S \subseteq V(H)$, we define the \emph{degree $\deg_H(S)$ of~$S$} as the number of edges of $H$ which contain~$S$.
Given a vertex $v \in V(H)$, we define the \emph{degree of $v$} as the degree of $\{v\}$.
Given some~$0 \leq i < k$, we let~$\delta_{i}(H)$ (and $\Delta_i(H)$) be the minimum (and maximum, respectively,) value of~$\deg_H(S)$ taken over all $i$-sets of vertices~$S$.
We call $\delta_{k-1}(H)$ the \emph{minimum codegree of~$H$} and sometimes we will write just $\delta(H)$ if $k$ is clear from context.

We say that a $k$-graph $H$ is \emph{$C^{\smash{(k)}}_\ell$-divisible} if $|E(H)|$ is divisible by $\ell$ and the degree of every vertex of~$H$ is divisible by~$k$.
Clearly, being $C^{\smash{(k)}}_\ell$-divisible is a necessary condition to admit a $C^{\smash{(k)}}_\ell$-decomposition, but in general it is not a sufficient condition.
We are interested in extremal questions of the sort: which conditions on the minimum degree of large $C^{\smash{(k)}}_\ell$-divisible graphs ensure the existence of $C^{\smash{(k)}}_\ell$-decompositions?
Given~$\ell\geq k$, we define the \emph{$C_{\ell}^{\smash{(k)}}$-decomposition threshold} $\delta_{\smash{C^{(k)}_{\ell}}}$ as the least $d > 0$ such that for every $\eps > 0$, there exists $n_0$ such that any $C_{\ell}^{\smash{(k)}}$-divisible $k$-graph $H$ on $n \geq n_0$ vertices with $\delta_{k-1}(H) \geq (d + \eps)n$ admits an $C_{\ell}^{\smash{(k)}}$-decomposition.

In this paper, we are interested in~$\delta_{\smash{C^{(k)}_{\ell}}}$.
For $k = 2$, $k$-graphs are just graphs, tight cycles are just graph cycles, and minimum codegree is just minimum degree, and here much more is known about the values of~$\delta_{\smash{C^{(2)}_{\ell}}}$.
Barber, K\"uhn, Lo, and Osthus~\cite{BKLO2016} show that $\delta_{\smash{C^{(2)}_4}} = 2/3$ and for each even $\ell \geq 6$, $\delta_{\smash{C^{(2)}_{\ell}}} = 1/2$.
Taylor~\cite{Taylor2019} proved exact minimum degree conditions which yield decompositions into cycles of length $\ell$ in large graphs, for $\ell = 4$ and every even $\ell \geq 8$.
For odd values of~$\ell$, the situation is different.
Joos and Kühn~\cite{JoosKuhn2021} showed that $\delta_{\smash{C^{(2)}_\ell}} = 1/2 + c_{\ell}$, where $c_\ell$ is a sequence of non-zero numbers depending on $\ell$ only, which satisfy $c_{\ell} \rightarrow 0$ when $\ell \rightarrow \infty$.

For $k = 3$, the last two authors~\cite{PigaSanhuezaMatamala2021} showed that $\delta_{\smash{C^{(3)}_\ell}} =2/3$ for all sufficiently large~$\ell$.
In fact, they show that the constant `$2/3$' is also sharp for the more general problem of decomposing hypergraphs into tight cycles of possibly different lengths, which we describe now.

A \emph{(tight) cycle-decomposition} of a $k$-graph $H$ is an edge partition of~$H$ into tight cycles (of possibly different lengths). 
A condition which is easily seen to be necessary to admit a cycle-decomposition is that the degree of every vertex of~$H$ is divisible by~$k$.
We define the \emph{cycle-decomposition threshold}~$\delta^{(k)}_{\smash{\textrm{cycle}}}$ be the least $d > 0$ such that for every $\eps > 0$, there exists $n_0$ such that any $k$-graph~$H$ on $n \geq n_0$ vertices with~$\delta_{k-1}(H) \geq (d + \eps)n$ such that every vertex of~$H$ has degree divisible by~$k$ admits a cycle-decomposition.
Note that $\delta^{(2)}_{\smash{\textrm{cycle}}} = 0$ as every graph with even degrees admits a cycle decomposition. 
For $k = 3$, the last two authors~\cite{PigaSanhuezaMatamala2021} showed that $\delta^{(3)}_{\smash{\textrm{cycle}}} = 2/3$.
Glock, Kühn and Osthus~\cite[Conjecture 5.5]{GlockKuhnOsthus2021} posed the following conjecture for $k \geq 3$.

\begin{conjecture}[{Glock, Kühn and Osthus}] \label{conj:GKO}
For $k \ge 3$, $\delta^{(k)}_{\smash{\mathrm{cycle}}} \le (k-1)/k$. 
\end{conjecture}

Cycle decompositions are also related with generalisations of Euler tours to hypergraphs.
An \emph{Euler tour} in a $k$-graph $H$ is a sequence of (possibly repeating) vertices~$v_1 \dotsb v_{m}$ such that each $k$ cyclically consecutive vertices forms an edge of~$H$, and all edges of $H$ appear uniquely in this way. 
Similarly, we define the \emph{Euler tour threshold}~$\delta^{(k)}_{\smash{\mathrm{Euler}}}$ be the least $d > 0$ such that for every $\eps > 0$, there exists $n_0$ such that any~$k$-graph~$H$ on $n \geq n_0$ vertices with $\delta_{k-1}(H) \geq (d + \eps)n$ such that every vertex of~$H$ is divisible by~$k$ admits an Euler tour.
Chung, Diaconis, and Graham~\cite{CDG} conjectured that every large $K_n^{\smash{(k)}}$ such that every vertex has degree divisible by~$k$ admits an Euler tour.
Glock, Joos, Kühn, and Osthus~\cite{GJKO2020} confirmed this conjecture and showed the existence of Euler tours in suitable hypergraphs by using results on cycle decompositions, which in particular show $\delta^{(k)}_{\smash{\mathrm{Euler}}} < 1$ for all $k$.
For $k = 2$, it is easy to see that~$\delta^{(2)}_{\smash{\mathrm{Euler}}} = 1/2$ (as $\delta(H) \ge |V(H)|/2$ is needed to ensure that the graph~$H$ is connected), and examples show that $\delta^{(k)}_{\smash{\mathrm{Euler}}} \geq 1/2$ holds for all $k \geq 3$~\cite[Section 1.3]{GJKO2020}.
The following conjecture for all $k \geq 3$ was posed.

\begin{conjecture}[{Glock, Kühn, and Osthus~\cite{GlockKuhnOsthus2021}}] \label{conj:GKO2}
For $k \ge 3$, $\delta^{(k)}_{\smash{\mathrm{Euler}}} \le (k-1)/k$. 
\end{conjecture}

It was first conjectured that $\delta^{(k)}_{\smash{\mathrm{Euler}}} =1/2$ for all $k \ge 3$ in~\cite{GJKO2020}, but this was disproven by the last two authors~\cite{PigaSanhuezaMatamala2021} by showing that $\delta^{(3)}_{\smash{\mathrm{Euler}}} = 2/3$.

Our main result bounds $\delta_{\smash{C^{(k)}_{\ell}}}$ for every $k \geq 2$ and each sufficiently large $\ell$.

\begin{theorem}\label{theorem:main}
    For every~$k\geq 3$ there exists an~$\ell_0\in\mathbb{N}$ such that for every~$\ell\geq \ell_0$ it holds that~$\delta_{\smash{C_\ell^{(k)}}}\leq 2/3$.
\end{theorem}

The case when $k=3$ already appears  in~\cite{PigaSanhuezaMatamala2021}.
For $k \geq 4$, we do not know if the constant `$2/3$' appearing in Theorem~\ref{theorem:main} is best-possible.
We discuss lower bounds in Section~\ref{section:lowerbounds}.

In order to prove Theorem~\ref{theorem:main}, we also find the decomposition threshold for tight paths. 
Given $\ell > k \geq 2$, the \emph{$k$-uniform tight path on $\ell$ vertices}, denoted by $P^{\smash{(k)}}_\ell$, is the~$k$-graph whose vertices are $\{v_1, \dotsc, v_{\ell} \}$ and its edges are $\{ v_i, v_{i+1}, \dotsc, v_{i + k - 1} \}$ for all $i \in \{1, \dotsc, \ell-k+1\}$.
A $k$-graph $H$ is \emph{$P_\ell^{\smash{(k)}}$-divisible} if $|E(H)|$ is divisible by $|E(P_\ell^{\smash{(k)}})| = \ell - k+1$.
For $\ell > k \geq 2$, we naturally define the \emph{$P_{\ell}^{\smash{(k)}}$-decomposition threshold} $\delta_{\smash{P_\ell^{(k)}}}$ as the least $d > 0$ such that for every $\eps > 0$, there exists $n_0$ such that any $P_{\ell}^{\smash{(k)}}$-divisible $k$-graph $H$ on $n \geq n_0$ vertices with $\delta_{k-1}(H) \geq (d + \eps)n$ admits an $P_{\ell}^{\smash{(k)}}$-decomposition.
We prove that $\delta_{\smash{P_\ell^{(k)}}} = 1/2$.

\begin{theorem}\label{theorem:path}
    For every~$k\geq 3$ and $ \ell \ge k+1$, $\delta_{\smash{P_\ell^{(k)}}} = 1/2$.
\end{theorem}

Using the techniques we apply to prove our main result, we can give bounds on $\delta^{(k)}_{\smash{\mathrm{cycle}}}$ and $\delta^{(k)}_{\smash{\mathrm{Euler}}}$, which in particular prove Conjectures~\ref{conj:GKO} and \ref{conj:GKO2} in a strong sense.
In fact, we also prove that both thresholds are always equal.

\begin{theorem}\label{cor:main}
    For all $k \ge 3$, $1/2 \le \delta^{(k)}_{\smash{\mathrm{Euler}}} = \delta^{(k)}_{\smash{\mathrm{cycle}}} \le \inf_{\ell > k} \{\delta_{\smash{C_\ell^{(k)}}}\} \le 2/3$.
\end{theorem}

\subsection{Proof ideas}\label{subsec:ideas}
Our proof uses the `iterative absorption' framework to tackle decomposition problems in hypergraphs; see~\cite{BGKLMO2020} for a introduction.
The proof of the main result (Theorem~\ref{theorem:main}) has three ingredients: an Absorber lemma, a Vortex lemma, and a Cover-down lemma.
The Vortex lemma gives a sequence of subsets $V(H) = U_0 \supseteq U_1 \supseteq \dotsb \supseteq U_t$, and $U_t$ has size independent of $n = |V(H)|$.
The Absorber lemma gives a small subgraph~$A \subseteq H$ such that for any $C^{\smash{(k)}}_\ell$-divisible leftover~$L\subseteq U_t$, the $k$-graph~$A \cup L$ has a cycle decomposition.
This reduces the problem to the search of a cycle packing in $H' = H - A$ which only has uncovered edges in $U_t$ (those can be later `absorbed' by~$A$).
This is found using the Cover-down lemma: in the $i$th step we find a collection of edge-disjoint cycles which covers all edges in $H'[U_i] - H'[U_{i+1}]$ but only uses few edges in $H'[U_{i+1}]$, this allows the process to be iterated.

Our proof of the Cover-down lemma requires a result of Joos and Kühn on `fractional decompositions'~\cite{JoosKuhn2021}, and a detour which finds and uses (tight) path decompositions.

Most of the work is required to prove the Absorber lemma.
We follow the approach of~\cite{PigaSanhuezaMatamala2021}, where absorbers are built by first finding `tour-trail decompositions' of the leftover graphs.
These decompositions consist of edge-disjoint subgraphs, each of which forms a tour or a trail.
It turns out that it is simple to build absorbers if the leftover can be decomposed into tours.
The goal is then to modify the leftover via the addition of \emph{gadgets}, these will suitably modify a given tour-trail decomposition in steps, so that at the end no trails remain.
We also prove an alternative condition for the existence of absorbers, see Lemma~\ref{lem:absorbing}, which should be of independent interest.

In this high-level description, this is the same outline used to find cycle decompositions when $k=3$ in~\cite{PigaSanhuezaMatamala2021}, but the proof for $k > 3$ requires several non-trivial modifications.
This is specially true in the construction of the absorbers, which is way more involved than in the $k=3$ case, and can be considered the main new contribution of the paper.

\subsection{Organisation}
In Section~\ref{section:lowerbounds} we give new lower bounds for the $C^{\smash{(k)}}_\ell$-decomposition threshold, for certain values of $k$ and~$\ell$.

In Section~\ref{section:transformers}, we establish a connection between the notion of \emph{transformers} and \emph{absorbers}. 
In Section~\ref{section:mainproof} we explain the iterative absorption method, including the statements of their key lemmata. 
At the end of this section we prove Theorem~\ref{theorem:main}.

Sections~\ref{section:vortex} to~\ref{section:coverdown} are devoted to the proofs of the lemmata used in the iterative absorption.
Section~\ref{section:vortex} contains the proof of the Vortex lemma.
The proof of the Absorber lemma is the main technical part of our paper, and its proof spans Sections~\ref{section:gadgets}, \ref{section:tourtrail}, and~\ref{section:transformer}.
The proof of the Cover-down lemma appears in Section~\ref{section:coverdown}.
We prove Theorem~\ref{theorem:path} in Section~\ref{subsec:pathdecom}.

In Section~\ref{section:euleriantours} we provide the necessary lemmata for the proof of Theorem~\ref{cor:main}. 
We finish in Section~\ref{section:remarks} with remarks and questions.

\subsection{Notation}\label{section:notation}
Let $[n] = \{ 1, \dots, n\}$. 
Since isolated vertices make no difference in our context, we usually do not distinguish from a hypergraph~$H=(V(H),E(H))$ and its set of edges~$E(H)$.
For a subset $U \subseteq V(H)$, we write $H \setminus U$ to mean the subgraph of~$H$ obtained by deleting vertices in~$U$.
We write $H[U] = H \setminus (V(H) \setminus U)$. 
For a $k$-graph~$G$ (not necessarily a subgraph of~$H$), we write $H-G = (V(H), E(H) \setminus E(G) )$.
We will suppress brackets and commas to refer to $k$-tuples of vertices when they are considered as edges of a hypergraph. 
For instance, for $v_1,\dots, v_k \in V(H)$, $v_1\dotsb v_k \in H$ means that the edge~$\{v_1,\dots, v_k\}$ is in~$E(H)$.
Whenever we have a set of vertices~$\{v_1,\dots, v_m\}$ indexed by an interval~$[m]$ any operation apply to the indices is considered to be modulo~$m$. 
For a vertex set $S \subseteq V(H)$, the \emph{neighbourhood~$N_H(S)$ of~$S$}  is 
the set of vertex sets $T \subseteq V(H) \setminus S$ such that $S \cup T \in H$.
Given $U \subseteq V(H)$, define~$N_H(S, U) = N_{H[S \cup U]}(S) $.
The degrees~$\deg_H(S)$ and~$\deg_H(S, U)$ correspond to $|N_H(S)|$ and $|N_H(S, U)|$, respectively.
We suppress $H$ if it can be deduced from context.

We also use the following notation.
Given $k \geq 2$ and $r \geq 1$, and a $k$-graph $H$, define~$\delta^{(r)}(H)$ to be the minimum of~$\vert N(e_1)\cap N(e_2)\cap \dotsb \cap N(e_r)\vert$ among all possible choices of $r$ different $(k-1)$-sets of vertices~$e_1, \dotsc, e_r$.
More generally, given a set of vertices~$U \subseteq V(H)$, we also define $\delta^{(r)}(H, U)$ as the minimum of~$\vert U \cap N(e_1)\cap N(e_2)\cap \dotsb \cap N(e_r)\vert$ among all possible choices of $r$ different $(k-1)$-sets of vertices~$e_1, \dotsc, e_r$.

We will use hierarchies in our statements.
The phrase ``$a \ll b$'' means ``for every~$b > 0$, there exists $a_0 > 0$, such that for all $0 < a \leq a_0$ the following statements hold''.
We implicitly assume all constants in such hierarchies are positive, and if $1/a$ appears we assume $a$ is an integer.

Suppose that Lemma~A states that a $k$-graph~$H$ contains a subgraph~$J$. 
We write `apply Lemma~A and obtain edge-disjoint subgraphs~$J_1, \dots, J_{\ell}$' to mean that `for each~$i \in [\ell]$, we apply Lemma~A to $H- \bigcup_{j \in[i-1]} {J_j}$ to obtain~$J_{i}$'.
Note that $H- \bigcup_{j \in[i-1]} {J_j}$ will also satisfy the condition of Lemma~A, but we will not check them explicitly. 
Furthermore, suppose that we have already found a subgraph~$H'$ of~$H$ and we say that `apply Lemma~A and obtain subgraph~$J$ such that $V(J) \setminus U$ are new vertices' to mean the `we apply Lemma~A to $H- (V(H') \setminus U)$ to obtain~$J$'.

\section{Lower bounds} \label{section:lowerbounds}

Given a $k$-graph $H$, let $\mathcal{C}_{\ell}(H)$ be the family of all~$C^{\smash{(k)}}_{\ell}$ in~$H$ and~$\mathcal C_\ell(H,e)$ the family of~$\ell$-cycles containing a fixed edge~$e\in H$. 
A \emph{fractional $C^{\smash{(k)}}_{\ell}$-decomposition} of~$H$ is a function $\omega: \mathcal{C}_{\ell}(H) \rightarrow [0,1]$ such that, for every edge $e \in H$, $\sum_{\substack{C \in \mathcal{C}_{\ell}(H,e)}} \omega(C) = 1$.
We define the \emph{fractional $C^{\smash{(k)}}_{\ell}$-decomposition threshold} $\delta^*_{\smash{C^{(k)}_{\ell}}}$ be the least $d > 0$ such that, for every $\eps > 0$, there exists $n_0$ such that any $k$-graph $H$ on $n \geq n_0$ vertices with~$\delta_{k-1}(H) \geq (d + \eps)n$ admits a fractional $C^{\smash{(k)}}_{\ell}$-decomposition.

Here, we give lower bounds on the parameter~$\delta^\ast_{C^{\smash{(k)}}_\ell}$.
Joos and Kühn~\cite{JoosKuhn2021} showed that~$\delta^*_{C^{\smash{(k)}}_\ell} \geq \frac{1}{2} + \frac{1}{(k-1 + 2/k)(\ell - 1)}$ holds for each $k \geq 2$ and $\ell$ not divisible by~$k$.
We give new bounds, which remove the dependency on~$k$.

\begin{proposition} \label{proposition:newlowerbound}
Let $1 \le i < k < \ell$ with $\ell$ not divisible by~$k$ and  $r  = k / \gcd (k,\ell)$.
Let $I_{\mathrm{free}} = \{ 0 \leq i \leq k : i \not\equiv 0 \bmod r \}$, $I_{\mathrm{odd}} = \{ 0 \leq i \leq k : i \not\equiv 0 \bmod 2 \}$ and~$I_{\mathrm{even}} = \{ 0 \leq i \leq k : i \equiv 0 \bmod 2 \}$.
Then
\begin{equation}
    \delta^\ast_{C^{(k)}_\ell} \geq 
    \frac{1}{2} + \frac{1}{2^{k}(\ell - 1)} \max 
    \left\lbrace \sum_{i \in I_{\mathrm{free}} \cap I_{\mathrm{odd}} } \binom{k}{i}, \sum_{i \in I_{\mathrm{free}} \cap I_{\mathrm{even}}} \binom{k}{i} \right\rbrace 
    \ge \frac12 + \frac1{4 (\ell -1)}.
    \label{equation:masterequation}
\end{equation}
\end{proposition}

Our constructions are based on~\cite[Proposition~3.1]{HanLoSanhuezaMatamala2021}.
Given vertex-disjoint sets $A, B$ and $0 \leq i \leq k$, we let  $H^{\smash{(k)}}_{i}(A,B)$ be the $k$-graph on $A \cup B$ such that $e \in H_i^{\smash{(k)}}(A,B)$ if and only if $|e \cap B| = i$.
We need the following observation. 
 
\begin{proposition} \label{prop:admissiblescycle}
Let $1 \le i < k < \ell$ and $d  = \gcd (k,\ell)$. 
Let $A$ and $B$ be disjoint vertex sets, and let $H_i = H^{(k)}_{\smash{i}}(A,B)$.
Then $H_i$ is $C^{\smash{(k)}}_{\ell}$-free for all $k-i \not\equiv 0 \bmod{k / d}$.
\end{proposition}

\begin{proof}
Suppose $\ell > k$ is such that $v_1 \dotsb v_{\ell}$ are the vertices of a copy of $C^{\smash{(k)}}_\ell$ in $H_i$.
We shall show that $k/d$ divides $k-i$.
For all $j \in [\ell]$, let $\phi_j \in \{A, B\}$ be such that $v_j \in \phi_j$ and let $\phi_{\ell+j} = \phi_j$.
Moreover, if two edges~$e_1,e_2\in E(H_i)$ satisfy~$|e_1\cap e_2|=k-1$, then the two vertices~$u\in e_1\setminus e_2$ and~$v\in e_2\setminus e_1$ belong to the same vertex-class~$A$ or~$B$.
In particular, $\phi_j = \phi_{j +k}$ for all $j \in [\ell]$.
Hence,~$\phi_{j+d} = \phi_j$ for all $ j \in [\ell]$.
Thus
\begin{align*}
	k-i = | \{v_1,\dots,v_k\} \cap A  | = | \{ j \in [k] : \phi_j = A \}| \in \{ k/d, 2k/d, \dots, k\}
\end{align*}
as required.
\end{proof}

We say a $k$-graph~$H$ on $n$ vertices admits an \emph{$\eta$-approximate $F$-decomposition} if it has a collection of edge-disjoint copies of $F$ covering all but $\eta n^k$ edges.
By a result of R\"odl, Schacht, Siggers, and Tokushige~\cite{RSST2007}, any bound on the codegree of $k$-graphs not containing $\eta$-approximate decompositions, for arbitrary small $\eta$, is essentially equivalent to bounding the corresponding numbers for fractional $C^{\smash{(k)}}_\ell$-decomposition.
Thus, we will focus on the former.

\begin{proof}[Proof of Proposition~\ref{proposition:newlowerbound}]
Let $n$ be sufficiently large.
Let $A$ and $B$ be disjoint vertex sets each of size~$\lfloor n/2 \rfloor$ and $\lceil n/2 \rceil$ respectively.
For each $0 \leq i \leq k$, let $H_i = H_i^{\smash{(k)}}(A,B)$.
Note that
\begin{align}
    \frac{|H_i|}{\binom{n}{k}} = \frac{1}{2^k} \binom{k}{i} + o(1)\,.
    \label{equation:binomial}
\end{align}

Let $d = \gcd(k, \ell)$.
Since $\ell$ is not divisible by $k$, then $r = k/d > 1$.
It holds that~$k - i \not\equiv 0 \bmod r$ if and only if $i \not\equiv 0 \bmod r$.
Let $H_{\text{odd}}\!=\!\bigcup_{i \in I_{\text{odd}}}\!H_i$ and $H_{\text{even}}\!=\! \bigcup_{i \in I_{\text{even}}}\!H_i$.
Note that~$\delta(H_{\text{odd}}), \delta(H_{\text{even}})\!\ge\! n/2 - k$.
From \eqref{equation:binomial} it follows that both $|H_{\text{odd}}|$ and $ |H_{\text{even}}|$ have size $(\frac{1}{2}+o(1)) \binom{n}{k}$.

Let $H'_{\text{odd}} = \bigcup_{i \in I_{\text{odd}} \cap I_{\text{free}} } H_i$ and $H'_{\text{even}} = \bigcup_{i \in I_{\text{even}} \cap I_{\text{free}} } H_i$.
Note that by \eqref{equation:binomial}, we have
\begin{equation}
 \frac{|H'_{\text{odd}}|}{\binom{n}{k}} 
 =
 \frac{1}{2^k} \sum_{i \in I_{\text{odd}} \cap I_{\text{free}}} \binom{k}{i} + o(1)
 \quad \text{and} \quad 
 \frac{|H'_{\text{even}}|}{\binom{n}{k}} 
 =
 \frac{1}{2^k} \sum_{i \in I_{\text{even}} \cap I_{\text{free}}} \binom{k}{i} + o(1).
 \label{equation:lowerbound-sizeshodd}
\end{equation}

Observe that given odd numbers~$i\neq j$ there is no tight path in $H_{\text{odd}}$ connecting an edge from~$H_i$ with an edge of~$H_j$. 
Therefore, Proposition~\ref{prop:admissiblescycle} yields that no edge in~$H'_{\text{odd}}$ is contained in a copy of~$C_\ell^{\smash{(k)}}$ in $H_{\text{odd}}$.
Let 
\begin{align*}
p =  \frac{2|H'_{\text{odd}}|}{(\ell-1) \binom{n}k},
\end{align*}
and suppose $\eta > 0$ is given.
Consider $H^*_{\text{even}}$ to be a sub-$k$-graph of~$H_{\text{even}}$ such that~$\delta (H^*_{\text{even}}) \ge (p-4 \eta) n/2 $ and 
\begin{align}
 |H^*_{\text{even}}| \le (p - 2.1 \eta ) |H_{\text{even}}|  < |H'_{\text{odd}}| / (\ell-1) - \eta n^k. \label{eqn:Heven}
\end{align} 
Such a sub-$k$-graph can be obtained by taking random edges from~$H_{\text{even}}$ independently with probability~$p - 3\eta $.

We claim $H=H_{\text{odd}} \cup H^*_{\text{even}}$ does not admit an~$\eta$-approximate $C_\ell^{\smash{(k)}}$-decomposition.
Since no edge of $H'_{\text{odd}}$ is contained in a copy of~$C_\ell^{\smash{(k)}}$ in~$H_{\text{odd}}$, each $C_\ell^{\smash{(k)}}$ containing an edge in~$H'_{\text{odd}}$ must contain at least one edge in~$H^*_{\text{even}}$.
Therefore, if $H$ contains an~$\eta$-approximate $C_\ell^{\smash{(k)}}$-decomposition, then we have 
\begin{align*}
(\ell-1) |H^*_{\text{even}}| \ge |H'_{\text{odd}}|-\eta |H| \ge |H'_{\text{odd}}| - \eta n^k, 
\end{align*}
contradicting~\eqref{eqn:Heven}.
Note that $\delta(H) \ge \delta(H_{\text{odd}}) + \delta(H^*_{\text{even}}) 
\ge (1+p - 4.5 \eta)n/2$.
Therefore, from \eqref{equation:lowerbound-sizeshodd}, letting $n$ tend to infinity, and $\eta$ tend to zero, we deduce that
\[ \delta^\ast_{C^{(k)}_\ell} 
\geq
\lim_{n \rightarrow \infty} \frac{1}{2}(1 + p) = \frac{1}{2} + \frac{1}{2^{k}(\ell - 1)} \sum_{i \in I_{\text{odd}} \cap I_{\text{free}}} \binom{k}{i}. \]
An analogous construction, selecting $H^\ast_{\text{odd}} \subseteq H_{\text{odd}}$ as a random set of the appropriate size with respect to $H'_{\text{even}}$, gives that
\[ \delta^\ast_{C^{(k)}_\ell} \geq \frac{1}{2} + \frac{1}{2^{k}(\ell - 1)} \sum_{i \in I_{\text{even}} \cap I_{\text{free}}} \binom{k}{i}, \]
and therefore we have
\begin{align*}
    \delta^\ast_{C^{(k)}_\ell} \geq \frac{1}{2} + \frac{1}{2^{k}(\ell - 1)} \max \left\lbrace \sum_{i \in I_{\text{odd}} \cap I_{\text{free}}} \binom{k}{i}, \sum_{i \in I_{\text{even}} \cap I_{\text{free}}} \binom{k}{i} \right\rbrace\,,
\end{align*}
which gives the first inequality of \eqref{equation:masterequation}.

To bound this last term, note that
\begin{align*}
    \sum_{i \in I_{\text{odd}} \cap I_{\text{free}}} \binom{k}{i} + \sum_{i \in I_{\text{even}} \cap I_{\text{free}}} \binom{k}{i}
    & = \sum_{1 \leq i \leq k, i \not\equiv 0 \bmod r} \binom{k}{i} \geq 2^{k-1}.
\end{align*} 
The last inequality follows since $\sum_{i \in [k] \colon i \equiv 0 \bmod r} \binom{k}{i}$ counts the number of sets of $[k]$ of size divisible by $r$,
and we recall that $r > 1$.
If $\mathcal{P}_r \subseteq \mathcal{P}([k])$ is that family, then $X \mapsto X \triangle \{1\}$ is an injection from $\mathcal{P}_r$ to $\mathcal{P}([k]) \setminus \mathcal{P}_r$, and thus $|\mathcal{P}_r| \leq |\mathcal{P}([k])|/2 = 2^{k-1}$.
We deduce that $\max \left\lbrace \sum_{i \in I_{\text{odd}} \cap I_{\text{free}}} \binom{k}{i}, \sum_{i \in I_{\text{even}} \cap I_{\text{free}}} \binom{k}{i} \right\rbrace \geq 2^{k-2}$,
which then yields
\[ \delta^\ast_{C^{(k)}_\ell} \geq \frac{1}{2} + \frac{1}{4(\ell - 1)}, \]
as desired.
\end{proof}

We can get better bounds for some choices of $k$ and $\ell$ by looking at \eqref{equation:masterequation} in detail.

\begin{corollary} \label{cor:newlowerbound}
Let $3 \leq k < \ell$ with $\ell \not\equiv 0 \bmod k$.
Then
\begin{align*}
 \delta^\ast_{C^{(k)}_{\ell}} \geq \begin{cases}
    \frac{1}{2} + \frac{1}{2(\ell - 1)} & \text{if $k/\gcd(\ell, k)$ is even}, \\
    \frac{1}{2} + \frac{1 - 2^{-k}}{2(\ell - 1)} & \text{if $\gcd(\ell, k) = 1$ and $k$ is odd}.
\end{cases} 
\end{align*}
\end{corollary}

\begin{proof}
    Let $d = \gcd(k, \ell)$ and $k = dr$.
    If $r$ is even, then $I_{\text{odd}} \cap I_{\text{free}} = I_{\text{odd}}$.
    Therefore
    $\sum_{i \in I_{\text{odd}} \cap I_{\text{free}}} \binom{k}{i} = \sum_{i \in I_{\text{odd}}} \binom{k}{i} = 2^{k-1}$,
    so $\delta^\ast_{\smash{C^{(k)}_{\ell}}} \geq \frac{1}{2} + \frac{1}{2(\ell - 1)}$ follows from Proposition~\ref{proposition:newlowerbound}.
    
    If $d = 1$, then $r = k$,
    and therefore $I_{\text{free}} = [k-1]$.
    This implies $\sum_{i \in  I_{\text{free}}} \binom{k}{i} = 2^k - 2$ and therefore
    $\max \left\lbrace \sum_{i \in I_{\text{odd}} \cap I_{\text{free}}} \binom{k}{i}, \sum_{i \in I_{\text{even}} \cap I_{\text{free}}} \binom{k}{i} \right\rbrace \geq 2^{k-1} - 1$,
    then $\delta^\ast_{C^{\smash{(k)}}_\ell} \geq \frac{1}{2} + \frac{1 - 2^{-k}}{2(\ell - 1)}$ again follows from Proposition~\ref{proposition:newlowerbound}.
\end{proof}

Finally, we can get bounds for the non-fractional thresholds~$\delta_{C^{\smash{(k)}}_\ell}$ by modifying the~$k$-graphs we construct in the proof of Proposition~\ref{proposition:newlowerbound} in such a way that they also are $C^{\smash{(k)}}_\ell$-divisible.
By removing at most $\ell - 1$ edges it is easy to make the total number of edges divisible by~$\ell$, so the only real challenge is to make every degree divisible by~$k$.
We prove later (Corollary~\ref{corollary:degreeadjuster}) that, for each $\eps > 0$ (assuming $n$ sufficiently large), we can find $F \subseteq H$ whose number of edges is divisible by $\ell$, $\delta_{k-1}(H - F) \geq \delta_{k-1}(H) - \eps n$,
and for each $v \in V(H)$, $\deg_{H - F}(v) \equiv 0 \bmod k$.
Thus the graphs we construct in Proposition~\ref{proposition:newlowerbound} can be modified to be $C^{\smash{(k)}}_\ell$-divisible, which implies the following bounds:

\begin{corollary} \label{cor:lowerbound}
    For all $3 \leq k < \ell$ and $\ell$ not divisible by $k$,
    \begin{enumerate}[label=\upshape{(\roman*)}]
        \item $\delta_{C^{\smash{(k)}}_\ell} \geq \frac{1}{2} + \frac{1}{4(\ell - 1)}$,
        \item if $k / \gcd(\ell, k)$ is even, then $\delta_{C^{\smash{(k)}}_\ell} \geq \frac{1}{2} + \frac{1}{2(\ell - 1)}$, and
        \item if $k / \gcd(\ell, k) = 1$ and $\ell$ is odd, then $\delta_{C^{\smash{(k)}}_\ell} \geq \frac{1}{2} + \frac{1 - 2^{-k}}{2(\ell - 1)}$.
    \end{enumerate}
\end{corollary}

\section{Absorbers versus transformers} \label{section:transformers}

In this section, we introduce absorbers and transformers, which are essential tools in the iterative absorption technique.
We prove that the existence of absorbers is essentially equivalent to the existence of transformers, and we work with the latter concept in the rest of the paper.
We state our results in a general fashion, that is, for~$F$-decompositions into general hypergraphs, not just cycles.

Given a~$k$-graph $F$ and~$0\leq i < k$ let~$\div_i(F)=\gcd\{\deg_F(S)\colon S\in\binom{V(F)}{i}\}$.
We say that a~$k$-graph $H$ is~\emph{$F$-divisible} if $\deg_H(S)$ is divisible by~$\div_{|S|}(S)$ for every subset~$S$ of~$V(H)$ on at most~$k-1$ vertices.
It is not hard to check that this definition in fact generalises the notions of~$C_\ell^{\smash{(k)}}$-divisible and $P_\ell^{\smash{(k)}}$-divisible introduced in Section~\ref{sec:intro} and that being~$F$-divisible is a necessary condition for the existence of an~$F$-decomposition.
As before, this condition is not sufficient in general and hence we define the \emph{$F$-decomposition threshold}~$\delta_F$ to be the least $d > 0$ such that for every $\eps > 0$, there exists $n_0$ such that any $F$-divisible $k$-graph $H$ on $n \geq n_0$ vertices with $\delta_{k-1}(H) \geq (d + \eps)n$ admits an $F$-decomposition.

Let $F$ and $G$ be $k$-graphs.
We say that a $k$-graph~$A$ is an \emph{$F$-absorber for~$G$} if both~$A$ and $A\cup G$ have $F$-decompositions and $A [V(G) ] = \emptyset$. 
Note that if there is an $F$-absorber for~$G$, then $G$ is $F$-divisible. 
The following definition describes $k$-graphs containing absorbers in a robust way. 

\begin{definition}\label{def:abs}\rm
Let $\eta: \mathbb{N} \rightarrow \mathbb{N}$ be a function.
We say that a $k$-graph~$H$ on $n$ vertices is \emph{$(F, m_G, m_W,\eta)$-absorbing} if, for all $F$-divisible subgraphs~$G$ of~$H$ with~$|V(G)| \le m_G$ and~$W \subseteq V(H)\setminus V(G)$ with~$|W| \le m_W - \eta(|V(G)|)$, $H \setminus W$ contains an $F$-absorber~$A$ for~$G$ with $|V(A)| \le \eta(|V(G)|)$.
\end{definition}

We will use so-called transformers to construct absorbers. 
The rôle of transformers is to replace~$G$ with a `homomorphic copy'~$G'$ of~$G$. 
Given $k$-graphs $G$ and~$G'$, a function $\phi: V (G) \rightarrow V (G')$ is an \emph{edge-bijective homomorphism from $G$ to~$G'$} if we have~$G' = \{ \phi(v_1) \dotsb \phi (v_k) : v_1 \dotsb v_k \in G\}$.
If such function exists, we say~$G$ and~$G'$ are~\emph{homomorphic}. 
A \emph{$(G,G';F)$-transformer} is a $k$-graph~$T$ such that~$T \cup G$ and $T \cup G'$ are $F$-decomposable and $T[V(G)] \cup T[V(G')]$ is empty.
The following definition is analogous to Definition~\ref{def:abs} but for transformers. 

\begin{definition}\label{def:trans}\rm
Let $\eta: \mathbb{N} \rightarrow \mathbb{N}$ be an increasing function with $\eta(x) \ge x$. 
We say that a $k$-graph~$H$ on $n$ vertices is \emph{$(F, m_G, m_W,\eta)$-transformable} if, for all vertex-disjoint homomorphic $F$-divisible subgraphs~$G, G'$ of~$H$ and $W \subseteq V(H) \setminus V(G \cup G')$ and $|V(G)|,|V(G')| \le m_G$ and $|W| \le m_W - \eta(|V(G)|)$, $H \setminus W$ contains a $(G,G';F)$-transformer~$T$ with $|V(T)| \le \eta(\max \{ |V(G)|,|V(G')|\})$.
\end{definition}

It is not difficult to see that if a $k$-graph $H$ is $(F, m_G, m_W,\eta)$-absorbing, then $H$ is also~$(F, m_G, m_W,2  \eta)$-transformable (see proof of Lemma~\ref{lem:absorbing}). 
In fact, the converse is true with different constants as long as there are enough copies of~$F$ in~$H$.

\begin{lemma} \label{lem:absorbing}
Let $\eta: \mathbb{N} \longrightarrow \mathbb{N}$ be an increasing function with $\eta(x) \ge x$ and~$m_W, m_G \geq0$.
Let $F$ be a $k$-graph and $u_1\dotsb u_k \in F$. 
Let $H$ be a $k$-graph such that, for any distinct~$v_1, \dots, v_k \in V(H)$ and $W \subseteq V(H) \setminus \{v_i: i \in[k]\}$ with $|W| \le m_W$, $(H \cup \{ v_1, \dotsc, v_k \}) \setminus W$ contains a copy of~$F$ with $u_i$ mapped to $v_i$ for all $i \in [k]$. 
Then if~$H$ is $(F, m_G, m_W,\eta)$-absorbing then $H$ is $(F, m_G, m_W, 2 \eta)$-transformable.
Moreover, if $H$ is $(F, m_G, m_W, \eta)$-transformable, then $H$ is $(F, m_G, m_W,\eta')$-absorbing for some increasing function~$\eta': \mathbb{N} \rightarrow \mathbb{N}$  with $\eta'(x) \ge x$.
\end{lemma}

For $k$-graphs~$G$ and~$H$ and $q \in \mathbb{N}$, we write $G + q H$ to be the vertex-disjoint union of~$G$ and $q$ copies of~$H$. 
We now show that, by adding $q$ vertex-disjoint copies of~$F$ to~$G$, the $k$-graph $G+qF$ has an edge-bijective homomorphism to $K_m^{\smash{(k)}}$.
We will require the following theorem regarding the existence of $F$-decompositions in high codegree~$k$-graphs. 

\begin{theorem}[Glock, K\"{u}hn, Lo and Osthus~\cite{GKLO2016}] \label{thm:Fdesign}
For all $k$-graphs~$F$, there exists a constant~$c_F > 0$ such that $\delta_F \le 1- c_F$. 
\end{theorem}

A subsequent alternative proof of Theorem~\ref{thm:Fdesign} was given by Keevash~\cite{Keevash2018}.

\begin{lemma} \label{lem:strongcolouring}
Let $F$ be a $k$-graph. 
Then, for all $t \in \mathbb{N}$, there exist integers $q=q(t)$ and $m= m(t)$ such that, for any $F$-divisible $k$-graph~$G$ with $|G| = t$, there exists an edge-bijective homomorphism from~$G + q F $ to~$K_m^{\smash{(k)}}$.
\end{lemma}

\begin{proof}
Let $c_F>0$ be the constant given by Theorem~\ref{thm:Fdesign}.
Let $1/m \ll 1/t, 1/k,c_F$ be such that $K_m^{\smash{(k)}}$ is $F$-divisible and any $F$-divisible subgraph~$H$ of $K_m^{\smash{(k)}}$ with $\delta(H) \ge (1 -  c_F/2) m$ has an $F$-decomposition.

Let $G'$ be an isomorphic copy of~$G$ with $V(G') \subseteq V (K_m^{\smash{(k)}} )$. 
Clearly there is an edge-bijective homomorphism~$\phi$ from~$G$ to~$G'$.
Since both $G'$ and~$K_m^{\smash{(k)}}$ are $F$-divisible, so is~$H=K_m^{\smash{(k)}}-G'$.
Note that since $|G'|= t$ then $\delta(H) \ge (1 - t/m) m \ge (1 -  c_F/2) m$.
Hence $H$ has an $F$-decomposition and we fix one. 
That is, $H$ can be edge-partitioned into $q=|H| /|F|$ copies of~$F$.
We extend $\phi$ to an edge-bijective homomorphism from~$G+qF$ to~$K_m^{\smash{(k)}}$, where we map each $F$ in $G+qF$ to a distinct copy of~$F$ in the $F$-decomposition of $H = K_m^{\smash{(k)}} - G'$.
\end{proof}

We now sketch how to construct absorbers from transformers, that is, the backwards direction of the proof of Lemma~\ref{lem:absorbing}.
Let $G \subseteq H$ be an $F$-divisible $k$-graph with $t$ edges, and suppose we can find a $qF$, which is vertex-disjoint from $G$, inside $H$.
By Lemma~\ref{lem:strongcolouring}, $G+qF$ has an edge-bijective homomorphism to $K_m^{\smash{(k)}}$.
First, suppose that $H$ contains a~$K_m^{\smash{(k)}}$ vertex-disjoint from $G+qF$.
Then, by our assumption, $H$ contains a $(G+qF,K_m^{\smash{(k)}};F)$-transformer~$T$.
Note that $T_1 = T \cup qF$ is a $(G,K_m^{\smash{(k)}};F)$-transformer.
Let $s = t/|F|$.
Note that $sF$ has $t$ edges, precisely the same number of edges as $G$.
Therefore, since $m, q$ in Lemma \ref{lem:strongcolouring} depend on $t$ only, if there exists a copy of $sF + qF \subseteq H$ which is vertex-disjoint from $K_m^{\smash{(k)}}$ we can repeat the same construction as above.
In the end, we would obtain an $(sF,K_m^{\smash{(k)}};F)$-transformer~$T_2$. 
Then $ T_1 \cup K_m^{\smash{(k)}} \cup T_2$ is a~$(G +qF, s F;F)$-transformer and so $qF \cup T_1 \cup K_m^{\smash{(k)}} \cup T_2 \cup s F$ is an $F$-absorber for~$G$.

However, an obvious obstacle with this approach is that $H$ may not contain any such large clique~$K_m^{\smash{(k)}}$. 
To overcome this problem, we consider the \emph{extension operator}~$\nabla$ (which was introduced in~\cite[Definition 8.13]{GKLO2016}).
Fix an edge $u_1\dotsb u_k \in F$. 
Consider any distinct vertices $v_1, \dots, v_k \in V(H)$. 
Define $\nabla_{F, u_1 \dotsb u_k}(v_1 \dotsb v_k) $ to be a copy of~$F-u_1\dotsb u_k$ with $v_i$ playing the rôles of~$u_i$.
For a $k$-graph~$G$ on $V(G) \subseteq V(H)$, define $\nabla_{F, u_1 \dotsb u_k}(G)$ to be the union of $\bigcup_{e \in G} \nabla_{F, u_1 \dotsb u_k}(e)$, where the ordering of each edge~$e \in G$  will be clear from the context and $V( \nabla_{F, u_1 \dotsb u_k}(e) ) \setminus e$ are new vertices. 
Note that $G \cup \nabla_{F, u_1 \dotsb u_k}(G)$ is $F$-decomposable. 
The hypothesis of Lemma~\ref{lem:absorbing} implies that~$\nabla_{F, u_1 \dotsb u_k}(G)$, 
$\nabla_{F, u_1 \dotsb u_k}(K_m^{\smash{(k)}})$ and $\nabla_{F, u_1 \dotsb u_k}(s F)$
 exist. 
Furthermore, there are edge-bijective homomorphisms between $\nabla_{F, u_1 \dotsb u_k}(G)$ and 
$\nabla_{F, u_1 \dotsb u_k}(K_m^{\smash{(k)}})$, and between $\nabla_{F, u_1 \dotsb u_k}(K_m^{\smash{(k)}})$ and $\nabla_{F, u_1 \dotsb u_k}(s F)$.
We then construct transformers between them to obtain an $F$-absorber for~$G$. 

\begin{proof}[Proof of Lemma~\ref{lem:absorbing}]
First suppose that $H$ is $(F, m_G, m_W,\eta)$-absorbing.
Let $G$ and~$G'$ be vertex-disjoint $F$-divisible subgraphs of~$H$ with $|V(G)|,|V(G')| \le m_G$.
Let $W \subseteq V(H) \setminus V(G \cup G')$ with $|W| \le m_W - 2  \eta(\max\{ |V(G)|, |V(G')|\})$.
By the property of being~$(F, m_G, m_W,\eta)$-absorbing, $H \setminus (W \cup V(G'))$ contains an $F$-absorber~$A_1$ for~$G$ with $A_1 [V(G) ] = \emptyset$ and $|V(A_1)| \le \eta(|V(G')|)$.
Also, $H \setminus (W \cup V(A_1))$ contains an~$F$-absorber~$A_2$ for~$G'$ with $A_2 [ V(G') ] = \emptyset$ and $|V(A_2)| \le \eta(|V(G')|)$.
Let $T = A_1 \cup A_2$.
Note that $T \cup G$ and $T \cup G'$ have $F$-decompositions and $ T[V(G \cup G')]= \emptyset$.
Hence~$T$ is a $(G, G'; F)$-transformer. 
Moreover, 
\begin{align*}
|V(T)| = |V(A_1)| + |V(A_2)| \le \eta(|V(G)|)+\eta(|V(G')|) \le 2 \eta ( \max \{|V(G)|,|V(G')|\}).
\end{align*}
So $H$ is $(F, m_G, m_W,2\eta)$-transformable.

Now suppose that $H$ is $(F, m_G, m_W,\eta)$-transformable.
Let $q(t)$ and $m(t)$ be the functions given by Lemma~\ref{lem:strongcolouring}.
Let $\eta'$ be the function given by
\begin{align*}
\eta'(x) = 2 \eta \left(  \max_{j \in [\binom{x}{k}] \cup \{ 0\} } \left\{ |V(F)| \binom{  m ( j ) }{k} \right\} \right).
\end{align*}
We now show that $H$ is $(F, m_G, m_W,\eta')$-absorbing.

Let $G$ be an $F$-divisible subgraph of~$H$ with $|V(G)| \le m_G$ and $W \subseteq V(H)\setminus V(G)$ with  $|W| \le m_W- \eta'(m_G)$.
Let 
\begin{align*}
q_1 & =q( |G|), &
m_1 & = m(|G|), &
q_2 & = \binom{m_1}{k}/ |F|.
\end{align*}
Let $(q_1+q_2)F$ be in $H \setminus (V(G) \cup W)$, which exists by our assumption on~$H$. 
Let $G_1 = G + q_1 F$ and $G_3 =q_2F$. 
Hence, $G_1$ and~$G_3$ are vertex-disjoint and are in~$H \setminus W$. 
Let~$V' = \{v'_1, \dots, v'_{m_1}\} \subseteq V(H) \setminus (V(G_1 \cup G_3) \cup W)$. 
Consider a $G_2 = K_{m_1}^{\smash{(k)}}$ on~$V'$, which may not exist in~$H$. 
By Lemma~\ref{lem:strongcolouring}, there exists an edge-bijective homomorphism~$\phi_j$ from~$G_j$ to~$K_{m_1}^{\smash{(k)}}$ for $j \in \{1,3\}$. 
Order edges in~$K_{m_1}^{\smash{(k)}}$ into~$v'_{i_1} \dots v'_{i_k}$ such that $i_1 < \dots < i_k$. 
By~$\phi_1$ and $\phi_3$, this implies an ordering on all edges of $G_1 \cup G_3$.
Fix an edge~$u_1\dotsb u_k \in F$. 
Let 
\begin{align*}
G'_1 &= \nabla_{F,u_1\dots u_k} (G_1), &
G'_2 &= \nabla_{F,u_1\dots u_k} (K_{m_1}^{\smash{(k)}}), &
G'_3 &= \nabla_{F,u_1\dots u_k} (G_3).
\end{align*}
Let $\ell = |V(F)| \binom{m_1}{k}$.
Note that
\begin{align*}
	|V(G'_j)| \le |V(F)| |G_j| = |V(F)| \binom{m_1}{k} = \ell.
\end{align*}
By the property of~$H$, $H$ contains vertex-disjoint $G'_1, G'_2, G'_3$ such that $G'_i[ V(G_j) ] = \emptyset$ for~$i\in [3]$ and~$j\in\{1,3\}$. 
Since $H$ is $(F, m_G, m_W,\eta)$-transformable, $H \setminus (W \cup V(G_3'))$ contains a $(G'_1, G'_2; F)$-transformer~$T_1$ with $|V(T_1)| \le \eta(\ell)$.
Similarly, $H \setminus (W \cup V(G_1') \cup (T_1 \setminus V(G'_2)))$ contains a $(G'_2, G'_3; F)$-transformer~$T_2$ with $|V(T_2)| \le \eta(\ell)$.

Let $A = (G_1 - G) \cup G'_1 \cup T_1 \cup G'_2 \cup T_2 \cup G'_3 \cup G_3$. 
Recall that $(G_1 - G)$, $G_1 \cup G'_1$, $G'_3 \cup G_3$ and $G_3$ have $F$-decompositions. 
Hence 
\begin{align*}
	A \cup G &= (G_1  \cup G'_1) \cup (T_1 \cup G'_2) \cup (T_2 \cup G'_3) \cup G_3 \text{ and}\\
	A &= (G_1-G) \cup (G'_1 \cup T_1) \cup (G'_2 \cup T_2) \cup (G'_3 \cup G_3) 
\end{align*}
are $F$-decomposable.
Therefore $A$ is an $F$-absorber for~$G$. 
Note that $A[V(G)] = T_1 [V(G)] \subseteq T_1[V(G'_1)] = \emptyset$ and 
\begin{align*}
|V(A)| \le |V(T_1)| +|V(T_2)|  \le 2 \eta (\ell ) = \eta'( |V(G)|).
\end{align*}
Hence $H$ is $(F, m_G, m_W  , \eta')$-absorbing.
\end{proof}

\section{Iterative absorption and proof of the main result} \label{section:mainproof}

The method of iterative absorption is based on three main lemmata: the Vortex lemma, the Absorber lemma, and the Cover-down lemma.
We state these lemmata while explaining the general strategy, then we will use them to prove Theorem \ref{theorem:main}. 
The proof of these lemmata are in Sections~\ref{section:vortex}-\ref{section:coverdown} (Sections~\ref{section:gadgets}-\ref{section:transformer} are dedicated to the Absorber lemma).

A sequence of nested subsets $U_0 \supseteq \dotsb \supseteq U_t$ of vertices of a $k$-graph $H$ is a \emph{$(\delta, \xi, m)$-vortex for $H$} if
\begin{enumerate}[label={(V\arabic*)}]
    \item $U_0 = V(H)$,
    \item for each $i \in [t]$, $|U_{i}| = \lfloor \xi |U_{i-1}| \rfloor$,
    \item $|U_t| = m$,
    \item $\delta^{(2)}(H[U_i]) \geq \delta |U_{i}|$, for each $0 \leq i \leq t$ and
    \item $\delta^{(2)}(H[U_i], U_{i+1}) \geq \delta |U_{i+1}|$, for each $0 \leq i < t$.
\end{enumerate}
\medskip
The Vortex lemma gives us the existence of vortices with the right parameters. 

\begin{lemma}[Vortex lemma] \label{lemma:vortex}
    Let $\delta > 0$ and $1/m' \ll \xi, 1/k$.
    Let $H$ be a $k$-graph on $n \geq m'$ vertices with $\delta^{(2)}(H) \geq \delta$.
    Then $H$ has a $(\delta - \xi, \xi, m)$-vortex, for some $\lfloor \xi m' \rfloor \leq m \leq m'$.
\end{lemma}

Using the properties of such a vortex, we will iteratively find~$C_\ell^{\smash{(k)}}$-packings covering the edges from~$H[U_i]$ in every step, without taking too many edges from the following sets~$U_{i+1}, \dots, U_t$ in the vortex.
The Cover-down lemma will provide the existence of those packings in every step. 

 \begin{lemma}[Cover-down lemma] \label{lem:coverdownlemma}
    For every~$k\geq 3$ and every~$\alpha>0$, 
    there is an~$\ell_0 \in \mathbb{N}$ such that for every~$\mu>0$ and every~$n,\ell\in\mathbb{N}$ with $\ell \ge \ell_0$ and $1/n \ll \mu, \alpha$ the following holds.
	Let $H$ be a $k$-graph on $n$ vertices, and $U \subseteq V(H)$ with $|U| = \lfloor \alpha n \rfloor$, and they satisfy
    \smallskip
    \begin{adjustwidth}{20pt}{0pt}
    	\begin{enumerate}[label={\rm(CD$_{\arabic*}$)}]
            \item $\delta^{(2)}(H) \geq 2\alpha n$,
    	    \item $\delta^{(2)}(H,U) \geq \alpha \vert U\vert$, and \label{it:CDdegree}
    	    \item $\deg_H(x)$ is divisible by $k$ for each $x \in V(H) \setminus U$.\label{it:CDdiv}
    	\end{enumerate}
    \end{adjustwidth}
    \smallskip
	Then $H$ contains  a $C_\ell^{\smash{(k)}}$-decomposable subgraph~$F\subseteq H$ such that $H - H[U] \subseteq F$
	and~$\Delta_{k-1}(F[U]) \leq \mu n$.
\end{lemma}

Finally, after repeated applications of the Cover-down lemma, we only need to consider the edges remaining in $H[U_t]$.
For these last edges, we apply the Absorber lemma.
This lemma says that the~$k$-graph $H$ is~$(C_\ell^{\smash{(k)}}, m, m ,\eta')$-absorbing, and therefore, it contains an absorber for \textit{any} possible~$C_\ell^{\smash{(k)}}$-divisible $k$-graph left as a remainder in~$U_t$ (which is of size~$m$). 

\begin{lemma}[Absorber lemma] \label{lem:transformer}
Let $1/n \ll \eps \ll 1/\ell, 1/k, 1/m $ with $k \ge 3$ and $\ell \ge 2(k^2 - k) + 1$.
Let $H$ be a $k$-graph on $n$ vertices with $\tmindeg(H) \geq 2 \eps n$.
Then $H$ is $(C_\ell^{\smash{(k)}}, m, m ,\eta')$-absorbing for some increasing function $\eta': \mathbb{N} \rightarrow \mathbb{N}$ satisfying $\eta'(x) \geq x$ and independent of $\eps$ and~$n$.
\end{lemma}

Now we are ready to prove Theorem~\ref{theorem:main}. 

\begin{proof}[Proof of Theorem~\ref{theorem:main}]
    We divide the proof into three steps: setting the vortex and absorber, covering down, and using the absorber to conclude.
    We suppose $\eps$, $\ell$, $m'$, $n_0$ are chosen according to the following hierarchy: $1/n_0 \ll 1/m' \ll \eps, 1/\ell \ll 1/k$.
    
    Let $H$ be a $C^{\smash{(k)}}_\ell$-divisible $k$-graph on $n \geq n_0$ with~$\delta_{k-1}(H)\geq (2/3+8\eps)n$ and observe that this immediately implies that~$\delta^{(3)}(H) \geq 8 \eps n$ and that~$\delta^{(2)}(H) \geq (1/3+8\eps)n$.
    To prove the lemma, it is enough to show that $H$ has a $C^{\smash{(k)}}_\ell$-decomposition.
    
    \medskip
    \noindent \emph{Step 1: Setting the vortex and absorber.}
    By Lemma~\ref{lemma:vortex}, we obtain a $(1/3+7 \eps, \eps, m)$-vortex $U_0 \supseteq \dotsb \supseteq U_t$ in $H$, for some $m$ satisfying $\lfloor \eps m' \rfloor \leq m \leq m'$.
    
    Let $\mathcal{L}$ be the set of all $C^{\smash{(k)}}_\ell$-divisible subgraphs of $H[U_t]$.
    Clearly, $|\mathcal{L}| \leq 2^{\binom{|U_t|}{k}} \leq 2^{m^k}$.
    Let $L \in \mathcal{L}$ be arbitrary.
    Clearly, $\delta^{(3)}(H - H[U_1]) \geq 7 \eps n$.
    By Lemma~\ref{lem:transformer} and the choice of constants, we deduce $H - H[U_1]$ is $(C^{\smash{(k)}}_\ell, m, m, \eta')$-absorbing for some increasing function $\eta': \mathbb{N} \rightarrow \mathbb{N}$ which satisfies $\eta'(x) \geq x$.
    Thus, $H - H[U_1]$ contains an $C^{\smash{(k)}}_\ell$-absorber~$A_L$ for~$L$, with $A_L[U_t] = \emptyset$ and $|A_L| \leq \eta'(m)$.
    We iterate this argument, finding edge-disjoint absorbers $A_{L'} \subseteq H - H[U_1]$, one for each $L' \in \mathcal{L}$.
    This indeed can be done, since all the absorbers found so far only spans at most $ k |\mathcal{L}| \eta'(m) \leq \eps n$ edges overall.
    Thus $H - H[U_1]$, after removing all edges of the already found absorbers, still satisfies $\delta^{(3)}(H - H[U_1]) \geq 6 \eps n$ and thus Lemma~\ref{lem:transformer} can still be invoked.
    
    Let $A = \bigcup_{L \in \mathcal{L}} A_L \subseteq H - H[U_1]$ be the edge-disjoint union of all absorbers.
    As argued before, $A$ contains at most $\eps n/3$ edges in total.
    By construction, $A$ is $C^{\smash{(k)}}_\ell$-decomposable, and for each $L \in \mathcal{L}$, $A \cup L$ is $C^{\smash{(k)}}_\ell$-decomposable.
    Let $H':= H - A$.
    Note that $\delta_{k-1}(H') \geq (2/3+6 \eps) n$ and $U_0 \supseteq \dotsb \supseteq U_t$ is a $(1/3+5 \eps, \eps, m)$-vortex for $H'$ (this is because we have ensured $A \subseteq H - H[U_1])$.
    Since $H$ and $A$ are $C^{\smash{(k)}}_\ell$-divisible,~$H'$ is also $C^{\smash{(k)}}_\ell$-divisible.
    
    \medskip
    \noindent \emph{Step 2: Covering down.}
    Now we want to find a $C^{\smash{(k)}}_\ell$-packing in $H'$ which covers all edges in $H' - H[U_t]$.
    For this, we proceed as follows.
    Let $U_{t+1} = \emptyset$.
    For each $0 \leq i \leq \ell$ we will find $H_i \subseteq H'[U_i]$ such that
    
    \begin{enumerate}[label={(\alph*$_{i}$)}]
        \item \label{item:mainproof-hidecompos} $H' - H_i$ has a $C^{\smash{(k)}}_\ell$-decomposition,
        \item \label{item:mainproof-hidegreei} $\delta^{(2)}(H_i) \geq (1/3+3 \eps) |U_i|$,
        \item \label{item:mainproof-hidegreei+1} $\delta^{(2)}(H_i, U_{i+1}) \geq (1/3+3 \eps) |U_{i+1}|$, and
        \item \label{item:mainproof-hi+1untouched} $H_i[U_{i+1}] = H'[U_{i+1}]$.
    \end{enumerate}
    
    For $i = 0$ this is done by setting $H_0 = H'$.
    Now, suppose that for $0 \leq i < \ell$ we have found $H_i \subseteq H'[U_i]$ satisfying \ref{item:mainproof-hidecompos}--\ref{item:mainproof-hi+1untouched},
    we construct $H_{i+1} \subseteq H'[U_{i+1}]$ satisfying \hyperref[item:mainproof-hidecompos]{\upshape{(a$_{i+1}$)}}--\hyperref[item:mainproof-hi+1untouched]{\upshape{(d$_{i+1}$)}}.
    By \ref{item:mainproof-hidecompos}, $H_i$ is $C^{\smash{(k)}}_\ell$-divisible.
    Let $H'_i = H_i - H_i[U_{i+2}]$.
    By \ref{item:mainproof-hidegreei}--\ref{item:mainproof-hidegreei+1} and~$|U_{i+2}| \leq \eps |U_{i+1}| \leq \eps^2 |U_i|$, we have
    \begin{enumerate}[label={(C\arabic*)}]
        \item $\delta^{(2)}(H'_i) \geq (1/3+ 2\eps) |U_i|$,
        \item $\delta^{(2)}(H'_i, U_{i+1}) \geq (1/3+ 2\eps) |U_{i+1}|$, and
        \item $\deg_{H'_i}(x)$ is divisible by $k$ for each $x \in U_i \setminus U_{i+1}$.
    \end{enumerate}
    Now we apply Lemma~\ref{lem:coverdownlemma} with~$1/3$,~$\ell$,~$\eps^6$,~$|U_i|$,~$H_i'$ and~$U_{i+1}$ playing the r\^oles of the parameters $\alpha$, $\ell$, $\mu$, $n$, $H$ and $U$.
    By doing so, we obtain a $C^{\smash{(k)}}_\ell$-decomposable subgraph~$F_i \subseteq H'_i$ such that $H'_i - H'_i[U_{i+1}] \subseteq F_i$ and $\Delta_{k-1}(F_i[U_{i+1}]) \leq \eps^6 |U_{i}|$.
    
    We let $H_{i+1} = H[U_{i+1}] - F_i$, and we now show that it satisfies the required properties.
    Since $H' - H_{i+1}$ is the edge-disjoint union of $H' - H_i$ and $F_i$ and both are~$C^{\smash{(k)}}_\ell$-decomposable, we deduce that \hyperref[item:mainproof-hidecompos]{\upshape{(a$_{i+1}$)}} holds.
    Note that we have $\Delta_{k-1}(F_i[U_{i+1}]) \leq \eps^6 |U_i| \leq \eps^4 |U_{i+1}| \leq \eps |U_{i+1}|$.
    From the definition of $(1/3+5 \eps, \eps, m)$-vortex for $H'$, we deduce that
    $\delta^{(2)}(H'[U_{i+1}]) \geq (1/3+5 \eps) |U_{i+1}|$ and $\delta^{(2)}(H'[U_{i+1}], U_{i+2}) \geq (1/3+5 \eps)|U_{i+2}|$.
    Using this, we are able to deduce that
    $\delta^{(2)}(H'_{i+1}) \geq \delta^{(2)}(H'[U_{i+1}]) - 2 \Delta_{k-1}(F_i[U_{i+1}]) \geq (1/3+5 \eps - 2\eps)|U_{i+1}| \geq (1/3+3\eps)|U_{i+1}|$, 
    and similarly we have  $\delta^{(2)}(H'_{i+1}, U_{i+2}) \geq (1/3+3 \eps) |U_{i+2}|$, 
    This shows that \hyperref[item:mainproof-hidegreei]{(b$_{i+1}$)} and \hyperref[item:mainproof-hidegreei+1]{(c$_{i+1}$)} hold.
    Finally, since $F_i \subseteq H'_i = H_i - H_i[U_{i+2}]$ we have that $H_{i+1}[U_{i+2}] = H_i[U_{i+2}] = H'[U_{i+2}]$, and therefore \hyperref[item:mainproof-hi+1untouched]{\upshape{(d$_{i+1}$)}} holds.
    
    At the end of this process, we have obtained $H_t \subseteq H'[U_t]$ such that $H' - H_t$ has a~$C^{\smash{(k)}}_\ell$-decomposition.
    
    \medskip
    \noindent \emph{Step 3: Finish.}
    Since $H'$ and $H' - H_t$ are~$C^{\smash{(k)}}_\ell$-divisible, we deduce $H_t \subseteq H'[U_t]$ is~$C^{\smash{(k)}}_\ell$-divisible.
    Therefore, $H_t \in \mathcal{L}$, and by construction we know that $H_t \cup A$ has a~$C^{\smash{(k)}}_\ell$-decomposition.
    Thus $H$ is the edge-disjoint union of $H_t \cup A$ and $H' - H_t$ and both of them have $C^{\smash{(k)}}_\ell$-decompositions, so we deduce $H$ has a $C^{\smash{(k)}}_\ell$-decomposition as well.
\end{proof}

\section{Vortex lemma} \label{section:vortex}

We prove Lemma~\ref{lemma:vortex} by selecting subsets at random (cf.~\cite[Lemma~3.7]{BGKLMO2020}).

\begin{proof}
    Let $n_0 = n$ and $n_i = \lfloor \xi n_{i-1} \rfloor$ for all $i \geq 1$.
    In particular, note $n_i \leq \xi^i n$.
    Let $t$ be the largest $i$ such that $n_i \geq m'$ and let $m = n_{t + 1}$.
    Note that $\lfloor \xi m' \rfloor \leq m \leq m'$.
    
    Let $\xi_0 = 0$ and, for all $i \geq 1$, define $\xi_i = \xi_{i-1} + 2 (\xi^i n)^{-1/3}$.
    Thus we have
    \[ \xi_{t + 1} = 2 n^{-1/3} \sum_{i \in [t]} (\xi^{-1/3})^i \leq 2 n^{-1/3} \sum_{i \in \mathbb{N}} (\xi^{-1/3})^i \leq \frac{2(n \xi)^{-1/3}}{1 - \xi^{-1/3}} \leq \xi, \]
    where in the last inequality we used $1/n \le 1/m' \ll \xi$.
    
    Clearly, taking $U_0 = V(H)$ is a $(\delta - \xi_0, \xi, n_0)$-vortex in $H$.
    Suppose now we have already found a $(\delta - \xi_{i-1}, \xi, n_{i-1})$-vortex $U_0 \supseteq \dotsb \supseteq U_{i-1}$ in $H$ for some $i \leq t+1$.
    In particular, $\delta^{(2)}(H[U_{i-1}]) \geq (\delta - \xi_{i-1}) |U_{i-1}|$.
    Let $U_i \subseteq U_{i-1}$ be a random subset of size $n_i$.
    By standard concentration inequalities, with positive probability we have $\delta^{(2)}(H[U_i]) \geq (\delta - \xi_{i-1} - n_i^{-1/3}) |U_{i}|$ and $\delta^{(2)}(H[U_{i-1}], U_i) \geq (\delta - \xi_{i-1} - n_i^{-1/3}) |U_{i}|$.
    Since $\xi_{i-1} + n_i^{-1/3} \leq \xi_i$, we have found a $(\delta - \xi_i, \xi, n_i)$-vortex for $H$.
    In the end, we have found a $(\delta - \xi_{t+1}, \xi, n_{t+1})$-vortex for $H$.
    Since $m = n_{t+1}$ and $\xi_{t+1} \leq \xi$, we are done.
\end{proof}

\section{Transformers I: gadgets} \label{section:gadgets}

In this and the next two sections we prove Lemma~\ref{lem:transformer}, the Absorber lemma.
Following Lemma~\ref{lem:absorbing}, it is enough to find transformers instead of absorbers. 
In this part, we introduce \emph{gadgets}, which will be building blocks of our transformers.

A \emph{$k$-uniform trail} is a sequence of (possibly repeated) vertices such that any $k$ consecutive vertices form an edge, and no edge appears more than once.
A \emph{$k$-uniform tour} is a $k$-uniform trail $v_1 \dotsb v_t$ such that $v_i = v_{t-k+1+i}$ for $i \in [k-1]$. 
Let $H$ be a $k$-graph. 
A \emph{tour-trail decomposition~$\mathcal{T}$ of~$H$} is an edge-decomposition of~$H$ into tours and trails. 
Note that every $k$-graph has a tour-trail decomposition, namely, considering each edge of~$H$ as a trail (by giving to it an arbitrary ordering).
A \emph{tour decomposition} is  a tour-trail decomposition consisting only of tours.
When it comes to the construction of absorbers, it is of great help to work with remainder subgraphs which admit tour decompositions.
Indeed, it is straightforward to find edge-bijective homomorphisms between tours and cycles.

To construct absorbers, we will prove that actually any $C^{\smash{(k)}}_\ell$-divisible $k$-graph can be augmented to a new, not-so-large, subgraph which does have such a tour decomposition.
This will be done in Section~\ref{section:tourtrail}, see Lemma~\ref{lem:tourdecom}.
In this section, we will describe certain small subgraphs which we will call \emph{gadgets}.
The augmented subgraph which we mentioned will be built as an edge-disjoint union of gadgets.

\subsection{Residual graphs} \label{sec:residualgraph}

Consider~$k\in \mathbb{N}$ to be fixed. 
Now we introduce the terminology we need to describe the gadgets.
Let $P = v_1 v_2 \dotsb v_{t}$ be a trail. 
We define the \emph{ends of~$P$} to be the ordered $(k-1)$-tuples $v_{k-1} v_{k-2} \dotsb v_1$ and $v_{t-k+2} v_{t-k+3} \dotsb v_t$. 
We denote by~$D(P)$ the multiset of ends of~$P$, (possibly counted with repetitions if both ends are the same).
Let $\mathcal{T}$ be a tour-trail decomposition on vertex set~$V$. 
We define the \emph{residual di-$(k-1)$-graph~$D(\mathcal{T})$ of~$\mathcal{T}$} to be the multiset $\bigcup_P D(P)$, where the union is taken over all trails $P$ in $\mathcal{T}$.
Thus $D(\mathcal{T})$ consists of ordered $(k-1)$-tuples of vertices in $V$, possibly counted with repetitions.

For $i \in [k-1]$ and a vertex~$v \in V$, let $p_{\mathcal{T} ,i}(v)$ be the number of ordered $(k-1)$-tuples in~$D(\mathcal{T})$ with $v$ being the $i$th vertex. 
We say that $\mathcal{T}$ is \emph{balanced} if, for all $v \in V$ and $i \in [k-1]$,
\begin{align}\label{dfn:balanced}
  p_{\mathcal{T},i}(v) = p_{\mathcal{T},k-i}(v).
\end{align}
We omit $\mathcal{T}$ from the subscript if it is known from the context. 

Observe that if $v_1 \dotsb v_{k-1}, v_{k-1} \dotsb v_{1} \in D ( \mathcal{T} )$, then there are trails $P_i, P_j \in \mathcal{T}$ that can be merged into a trail (if $i \ne j$) or tour (if $i =j$) with edge set $E(P_i \cup P_j)$. 
Thus there is another tour-trail decomposition with fewer trails than~$\mathcal{T}$, which is obtained from $\mathcal{T}$ by removing $P_i$, $P_j$ and adding the tour or trail born from joining $P_i$ and~$P_j$.
We will abuse the notation by calling the resulting tour-trail decomposition by~$\mathcal{T}$.
This merging procedure will be indicated by 
\begin{align*}
	D(\mathcal{T}) = D(\mathcal{T}) \setminus \{ v_1 \dots v_{k-1}, v_{k-1} \dots v_{1}\}.
\end{align*}
Tours in~$\mathcal{T}$ do not contribute to~$D(\mathcal{T})$.
Hence $\mathcal{T}$ is a tour decomposition if and only if $D(\mathcal{T}) = \emptyset$ (after merging procedures). 
Recall that our goal in Lemma~\ref{lem:tourdecom} is to augment a $C_{\ell}^{\smash{(k)}}$-divisible $k$-graph to a $k$-graph which has a tour decomposition.
Hence, conceptually, it may be helpful to assume that all tour-trail decompositions consist of only trails, as tours will not appear in~$D(\mathcal{T})$.

Given a vertex~$x\in V$ and a~$k$-tuple~$\mathbf{y}\!=\!y_1 \dotsb y_{k}$, for every~$i\!\in\! [k]$ we define 
\begin{align*}
    r_i(\mathbf{y}, x) &= y_1\dotsb y_{i-1}xy_{i+1}\dotsb y_{k}\,, \text{ and }
    s_i(\mathbf y) = \{y_1\dotsb y_{i-1}y_{i+1}\dotsb y_{k}\}\,.
\end{align*}
In words, $r_i(\mathbf{y},x)$ replaces the~$i$th vertex in~$\mathbf{y}$ with the vertex~$x$, whilst~$s_i$ simply skips the $i$th vertex of $\mathbf{y}$.
Moreover, we define the \emph{reverse} of~$\mathbf{y}$ as~$\mathbf{y}^{-1} = y_{k}\dotsb y_1$.
Finally, given a permutation~$\sigma$ of~$[k]$ we write~$\sigma(\mathbf{y})$ for the tuple $y_{\sigma(1)}\dotsb y_{\sigma(k)}$.

Given vertices~$x, x'\in V$ and~$(k-1)$-tuples~$\mathbf{z}=z_1\dotsb z_{k-1}$ and~$\mathbf{z'}=z_1'\dotsb z_{k-1}'$, we define the following sets of $(k-1)$-tuples:
\begin{align*}
    \vect S_i^{(k-1)}(\mathbf{z}, x, x') 
    &=\{z_1\dots z_{i-1}xz_{i+1}\dots z_{k-1},\, z_{k-1}\dots z_{i+1}x'z_{i-1}\dots z_1\}\\
    &=\{r_i(\mathbf{z},x), (r_i(\mathbf{z},x'))^{-1}\}
    \text{ and}\\
    \vect T_i^{(k-1)}(\mathbf{z},\mathbf{z}', x,x') &= \vect S_i^{(k-1)}(\mathbf{z}, x, x') \cup \vect S_i^{(k-1)}(\mathbf{z}', x', x) \\
    & = \{ r_i(\mathbf{z},x), r_i(\mathbf{z'},x'), (r_i(\mathbf{z},x'))^{-1} , (r_i(\mathbf{z'},x))^{-1}\}
    .
\end{align*}
For a~$k$-tuple~$\mathbf y$ (instead of a~$(k-1)$-tuple), it will be convenient to consider the set~$\vect S_i^{\smash{(k-1)}}(\mathbf y, x,x')$ with the same definition but omitting the last element of the tuple~$y_k$.
More precisely, given a~$k$-tuple~$\mathbf y$ we define \[ \vect S_i^{(k-1)}(\mathbf y, x,x') =\vect S_i^{(k-1)}(s_k(\mathbf y), x,x') = \{r_i(s_k(\mathbf y),x), r_i(s_k(\mathbf y),x')^{-1}\}.\]
An analogous definition holds for~$\vect T_i^{(k-1)}(\mathbf y, \mathbf y', x,x')$, where~$\mathbf y$ and~$\mathbf y'$ are~$k$-tuples.
Since~$k$ will be always clear from the context, we will omit it in the notation.

\newcommand{\firstgadget}{balancer}
\newcommand{\secondgadget}{swapper}

The two types of $C_{\ell}^{\smash{(k)}}$-decomposable $k$-graphs to be constructed will be $B_j(x,x')$ in Corollary~\ref{corjbasic} (which we call `\emph{\firstgadget{} gadgets}') and $T_j(\mathbf{y},\mathbf{y}',x,x')$ in Lemma~\ref{lem:jswitch} (which we call `\emph{\secondgadget{} gadgets}').
This last \secondgadget{} gadget has a tour-trail decomposition whose residual digraph is exactly~$\vect T_j(\mathbf y, \mathbf y', x,x')$. 
Essentially, the main properties of the gadgets are:

\begin{itemize}
    \item The rôle of \firstgadget{} gadgets $B_j(x,x')$ is to enable us to adjust $p_{j}(x) - p_{k-1-j}(x)$ without affecting other vertices in $V(H) \setminus \{x,x'\}$.
    Hence, by adding edge-disjoint copies of \firstgadget{} gadgets, the resulting tour-trail decomposition~$\mathcal{T}$ will be balanced (see Lemma~\ref{lem:balancetourdecom}).
    
    \item Suppose now $\mathcal{T}$ is balanced.
    Consider $x \in V(H)$ and $1 \leq j \leq k/2$, we can now pair the members of~$D(\mathcal{T})$ containing~$x$ into pairs $(\mathbf{y},\mathbf{y}')$ such that $x$ is the $j$th vertex in~$\mathbf{y}$ and $(k-1-j)$th vertex in~$\mathbf{y}'$. 
    This is possible, as $\mathcal{T}$ is balanced.
    The \secondgadget{} gadget $T_j(\mathbf y,\mathbf y', x,x')$ will enable us to `replace' $x$ with a new vertex $x'$ in both $\mathbf y$ and $\mathbf y'$.
    By repeated applications of this gadget, this will allow us to convert~$\mathcal{T}$ into a tour decomposition (see Lemmata~\ref{lem:balancetourdecom2} and~\ref{lem:balancetourdecom3}).
\end{itemize}

\subsection{Basic gadgets}

Our gadgets will be composed of union of edge-disjoint tight cycles~$C_{\ell}^{\smash{(k)}}$, so  they are $C_{\ell}^{\smash{(k)}}$-decomposable. 
In order to show that there exists a tour-trail decomposition~$\mathcal{T}$ with the desired~$D(\mathcal{T})$, we adopt the following convention. 
Given a tight cycle $C = v_1 \dotsb v_{\ell}$, we first consider a trail decomposition of~$C$ consisting of two trails $C- v_1\dotsb v_{k} = v_2 \dotsb v_{\ell } v_1 \dotsb v_{k-1}$ and $v_1\dotsb v_{k}$. 
Note that the latter is a single edge. 
Then we pick a permutation $\sigma$ of~$[k]$ and replace $v_1\dotsb v_{k}$ with $v_{\sigma(1)} \dotsb v_{\sigma(k)}$.
Let $\mathcal{T}$ be the resulting trail decomposition of~$C$, so 
\begin{align*}
    \mathcal{T} & = 
    \left\{
        \begin{array}{c}  
	       v_2 \dotsb v_{\ell } v_1 \dotsb v_{k-1}, \\
	v_{\sigma(1)} \dotsb v_{\sigma(k)}
\end{array}
\right\}
& \text{ and }
D(\mathcal{T}) & = 
\left\{
\begin{array}{cc}  
    v_{k} \dotsb v_2, & v_1 \dotsb v_{k-1} \\
    v_{\sigma(k-1)} \dotsb v_{\sigma(1)}, & v_{\sigma(2)} \dotsb v_{\sigma(k)}
\end{array}
\right\}.
\end{align*}
Since a gadget is a union of edge-disjoint tight cycles $C_1, \dots, C_s$, its tour-trail decomposition~$\mathcal{T}$ will be first given by $\mathcal{T}_1 \cup \dots \cup \mathcal{T}_{s}$, where each $\mathcal{T}_i$ is a trail decomposition of~$C_i$ in the form above. 
For ease of checking, we often write $D(\mathcal{T}) = D(\mathcal{T}_1) \cup \dots \cup D(\mathcal{T}_s)$ before merging some of the trails in~$\mathcal{T}$, that is, deleting pairs in $D(\mathcal{T})$ of form~$\{\mathbf y, \mathbf y^{-1}\}$.

The next lemma finds many trails of prescribed length which connect any given pair of ordered $(k-1)$-tuples.
It follows from~\cite[Lemma 2.3]{JoosKuhn2021}.
In particular, it shows that every edge is in a tight cycle. 

\begin{lemma} \label{lem:findcycle}
Let $1/n \ll \rho \ll \eps \ll 1/\ell, 1/k$ with $k \ge 3$ and $\ell \ge k^2 - k$.
Let $H$ be a $k$-graph on $n$ vertices with $\delta^{(2)}(H) \geq \eps n$.
Then, for every two $(k-1)$-tuples $\mathbf{x} = v_1 \dotsb v_{k-1}$ and~$\mathbf{y} = v_{\ell + 1} \dotsb v_{\ell + k - 1}$, $H$ contains at least $\rho n^{\ell - k + 1}$ trails $v_1 \dotsb v_{\ell + k - 1}$ from~$\mathbf{x}$ to $\mathbf{y}$ on $\ell$ edges, each of them with no repeated vertices, except possibly those already repeated in~$\mathbf{x}$ and~$\mathbf{y}$.
In particular, if $\mathbf{x}, \mathbf{y}$ are disjoint, then each of these trails is a tight path of length~$\ell$.
\end{lemma}

We now construct a $C_{\ell}^{\smash{(k)}}$-decomposable $k$-graph, which will be a basic building block of all of the next gadgets. 

\begin{lemma}\label{lem:jbasic}
Let $1/n \ll \eps \ll 1/\ell, 1/k $ with $k \ge 3$ and $\ell \ge k^2 - k + 1$.
Let $H$ be a~$k$-graph on $n$ vertices with $\delta^{(2)}(H) \geq \eps n$.
Let $j \in [k-1]$. 
Let $x,x'\in V(H)$ be distinct vertices and $\mathbf y$ be a~$k$-tuple of~$V(H)$ such that~$Y=\{y_i : i \in [k]\setminus \{ j \} \} \in N_H(x) \cap N_H(x')$. 
Then there exists a $C^{\smash{(k)}}_\ell$-decomposable $k$-graph $G = G_j(\mathbf y, x,x')$ in~$H$ with $|G| =2 \ell$ and a tour-trail decomposition~$\mathcal{T}_j$ of~$G$ satisfying
\begin{align*}
	D(\mathcal{T}_j) = 
	    \vect S_j(\mathbf y, x,x')\cup 
	    \vect S_1(\sigma_1(\mathbf y),x',x)\cup 
	    \vect S_{j-1}(\sigma_2(\mathbf y), x,x'),
\end{align*}
where $\sigma_1 = j12\dotsb(j-1)(j+1)\dotsb k$ and $\sigma_2 = 2\dotsb k 1$.
Moreover, $G[\{x,x'\} \cup Y\}] = \{x \cup Y, x' \cup Y\}$.
\end{lemma}

\begin{proof}
Orient $Y \cup x$ and $Y \cup x'$ into $y_{k} \dotsb y_{j+1} x y_{j-1} \dotsb y_1$ and $x'y_1 \dotsb y_{j-1} y_{j+1} \dotsb y_{k}$.
By Lemma~\ref{lem:findcycle}, there exist two tight cycles of length~$\ell$,
\begin{align*}
	C_1 & = y_{k} \dotsb y_{j+1} x y_{j-1} \dotsb y_1 v_{k+1} \dotsb v_{\ell}, \\
	C_2 & = x'y_1 \dotsb y_{j-1} y_{j+1} \dotsb y_{k}u_{k+1} \dotsb u_{\ell},
\end{align*}
where $v_{i}$ and~$u_i$ for $k+1\leq i\leq \ell$ are new distinct vertices. 
Let $G = C_1 \cup C_2$.
Define the tour-trail decomposition~$\mathcal{T}$ of $C_1 \cup C_2$ to be 
\begin{align*}
\mathcal{T}_j  
& = \left\{
\begin{array}{c}
x y_1 \dots y_{j-1} y_{j+1} \dots y_{k},\\
y_{k-1} \dots y_{j+1} x y_{j-1} \dots y_1 v_{k+1} \dots v_{\ell} y_k \dots y_{j+1} x y_{j-1} \dots y_2,\\
y_k \dots y_{j+1} x' y_{j-1} \dots y_1,\\
y_1 \dots y_{j-1} y_{j+1} \dots y_{k}u_{k+1} \dots u_{\ell}x' y_1 \dots y_{j-1} y_{j+1}\dots y_{k-1}
\end{array}
\right\}.
\end{align*}
In words, as explained before, the first trail in $\mathcal{T}_j$ consists of a single edge of $C_1$ in a different order to how it appears in $C_1$, and the second trail corresponds to $C_1$ without the previous edge; similar with $C_2$ and the third and fourth trails in $\mathcal{T}_j$.
Hence, $D(\mathcal{T}_j)$ consists of 
\begin{align*}
D(\mathcal{T}_j) 
& = \left\{
\begin{array}{cc}
y_{k-1} \dots y_{j+1} y_{j-1} \dots y_1 x,
&
\hcancel{y_1 \dots y_{j-1} y_{j+1} \dots y_{k}},
\\
y_1 \dots y_{j-1} x y_{j+1} \dots y_{k-1},
&
y_k \dots y_{j+1} x y_{j-1} \dots y_2,
\\
y_2 \dots y_{j-1} x' y_{j+1} \dots y_{k},
&
y_{k-1} \dots y_{j+1} x' y_{j-1} \dots y_1,
\\
\hcancel{y_{k} \dots y_{j+1} y_{j-1} \dots y_1},
&
x' y_1 \dots y_{j-1} y_{j+1} \dots y_{k-1}
\end{array}
\right\}\\
& = \vect S_j(\mathbf y, x,x')\cup \vect S_1(\sigma_1(\mathbf y),x',x)\cup \vect S_{j-1}(\sigma_2(\mathbf y), x,x'),
\end{align*}
where $\sigma_1 = j1\dotsb(j-1)(j+1)\dotsb k$ and $\sigma_2 = 2\dotsb k 1$.
See Figure~\ref{fig:basic}.
\end{proof}

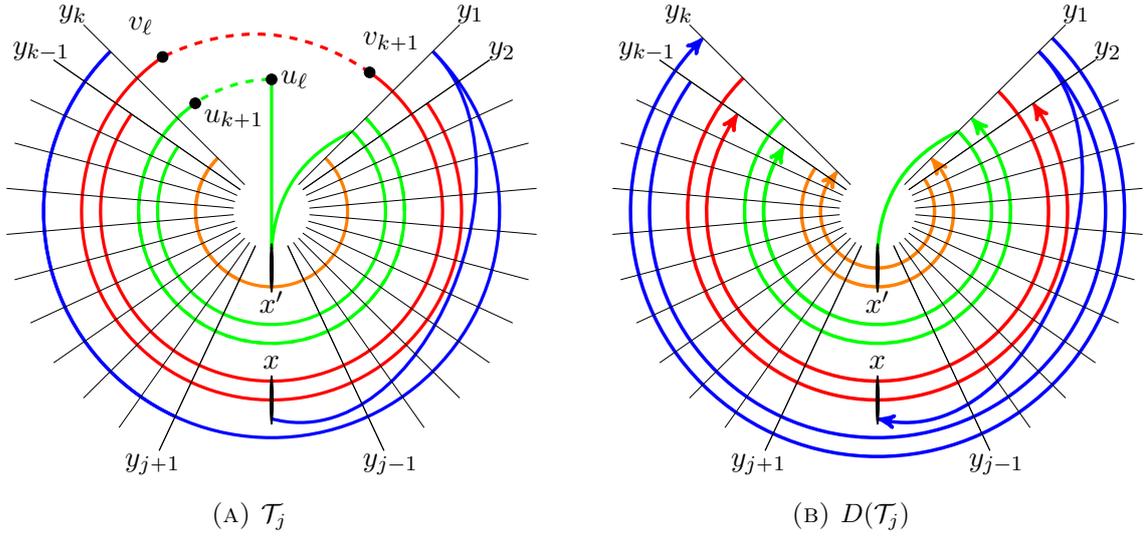
\begin{figure}[!htbp]
     \centering
     \begin{subfigure}[b]{0.45\textwidth}
         \begin{tikzpicture}[scale=1]

        \begin{scope}[line width=1.2pt]
			\draw[blue] (45:3) arc (45:-225:3);		
			\draw[blue] (-225:3) arc (-225:-115:3);		
			\draw[blue] (45:3) to [out =-45,in=-15] (0,-2.75);
			
			\draw[red] (55:2.25) arc (55:-215:2.25);		
			\draw[red] (35:2.5) arc (35:-235:2.5);		            
			\draw[red, dashed] (55:2.25) to [bend right] (125:2.5);
			
			\draw[orange] (45:1) arc (45:-225:1);		
			\draw[orange] (-225:1) arc (-225:-115:1);

		   	\draw[green, dashed] (45:1.75) arc (45:-270:1.75);
            \draw[green] (45:1.75) arc (45:-235:1.75);
			\draw[green] (90:1.75) -- (0,-.5) to [bend left] (45:1.5);
			\draw[green] (45:1.5) arc (45:-215:1.5);	

   
		\end{scope}

		\foreach \y/\x in {45/1,35/2,-65/{j-1},-115/{j+1},145/{k-1},135/{k}}{
				\draw[line width=0.1pt] (\y:3.5) -- (\y:0.5) ;
				\draw (\y:3.7) node {$y_{\x}$};
			}
				
		\foreach \y in {35,25,...,-65}{
				\draw[line width=0.1pt] (\y:3.5) -- (\y:0.5) ;
			}
		\foreach \y in {145,155,...,245}{
				\draw[line width=0.1pt] (\y:3.5) -- (\y:0.5) ;
			}

            \filldraw (0,-2.5) ellipse (0.2 mm and 3.2mm);
				\draw (-90:2) node {$x$};
				
			\filldraw (0,-0.75) ellipse (0.2 mm and 3.2mm);
				\draw (0,-1.2) node {$x'$};
					
			\filldraw (55:2.25) circle (2pt);
				\draw (55:2.75) node {$v_{k+1}$};
				
			\filldraw (125:2.5) circle (2pt);
				\draw (125:3) node {$v_{\ell}$};
			
			\filldraw (90:1.75) circle (2pt);
				\draw (0.3,1.75) node {$u_{\ell}$};
									
			\filldraw (125:1.75) circle (2pt);
				\draw (112:1.32) node {$u_{k+1}$};
        \end{tikzpicture}
         \caption{$\mathcal{T}_j$}
         \end{subfigure}
     \hfill
     \begin{subfigure}[b]{0.45\textwidth}
         \begin{tikzpicture}[scale=1]

			\begin{scope}[line width=1.3pt]
			\draw[blue,->,>=stealth'] (45:3.25) arc (45:-225:3.25);		
			
			\draw[blue] (45:3) arc (45:-215:3);			
			\draw[blue,->,>=stealth'] (45:3) to [out =-45,in=-15] (0,-2.75);		
			
			\draw[red,->,>=stealth'] (45:2.25) arc (45:-215:2.25);		
			\draw[red,<-,>=stealth'] (35:2.5) arc (35:-225:2.5);

			\draw[orange,->,>=stealth'] (35:0.75) arc (35:-225:0.75);		
    		\draw[orange,<-,>=stealth'] (45:1) arc (45:-215:1);		
			
	       	\draw[green,<-,>=stealth'] (45:1.75) arc (45:-225:1.75);		
			
			\draw[green] (0,-.5) to [bend left] (45:1.5);
			\draw[green] (45:1.5) arc (45:-215:1.5);		
			\draw[green,<-,>=stealth'] (-215:1.5) arc (-215:-115:1.5);		
			
			\end{scope}

		\foreach \y/\x in {45/1,35/2,-65/{j-1},-115/{j+1},145/{k-1},135/{k}}{
				\draw[line width=0.1pt] (\y:3.5) -- (\y:0.5) ;
				\draw (\y:3.7) node {$y_{\x}$};
			}
				
		\foreach \y in {35,25,...,-65}{
				\draw[line width=0.1pt] (\y:3.5) -- (\y:0.5) ;
			}
		\foreach \y in {145,155,...,245}{
				\draw[line width=0.1pt] (\y:3.5) -- (\y:0.5) ;
			}
				\filldraw (0,-2.5) ellipse (0.2 mm and 3.2mm);
				\draw (-90:2) node {$x$};
				
				\filldraw (0,-0.75) ellipse (0.2 mm and 3.2mm);
				\draw (0,-1.2) node {$x'$};
    				
        \end{tikzpicture}
         \caption{$D(\mathcal{T}_j)$}
         \end{subfigure}
    \caption{The tour-trail decomposition $\mathcal{T}_j$ of the basic gadget $G = G_j(\mathbf y, x,x')$ and its residual di-$(k-1)$-graph~$D(\mathcal{T}_j)$. Dotted lines represent tight paths using new vertices.}
     \label{fig:basic}
\end{figure}

In the following subsections we introduce further gadgets based on~$G_j(\mathbf {y}, x,x')$.
For those, diagrams similar to the one in Figure~\ref{fig:basic} would get more convoluted. 
Thus, we refrain from presenting such diagrams and check the main properties of the gadgets based solely on text and on a explicit list of elements in the residual di-$(k-1)$-graphs~$\mathcal D(\mathcal T)$ of the respective tour-trail decompositions~$\mathcal T$.

\subsection{Balancer gadgets}

Next, we will use Lemma~\ref{lem:jbasic} to construct a \firstgadget{} gadget~$B_j = B_j(x, x')$.
As mentioned before, the main property of $B_j(x,x')$ is to enable us to increase $p_{\mathcal{T},i}(x) - p_{\mathcal{T},k-1-i}(x)$ (and decrease $p_{\mathcal{T},1}(x) - p_{\mathcal{T},k-1}(x)$) without affecting other vertices in $V \setminus \{x,x'\}$.

\begin{corollary}[Balancer gadgets] \label{corjbasic}
Let $1/n \ll \eps \ll 1/\ell, 1/k $ with $k \ge 3$ and $\ell \ge k^2-k+1$.
Let $H$ be a $k$-graph on $n$ vertices with $\delta^{(2)}(H) \geq \eps n$.
Let $j \in [k-1] \setminus \{1\}$ and let~$x,x' \in V(H)$ be distinct. 
Then there exists a $C^{\smash{(k)}}_\ell$-decomposable $k$-graph $B_j = B_j(x, x')$  with $|B_j| = 2(j-1) \ell$ and a tour-trail decomposition~$\mathcal{T}'_j$ of~$B_j$ such that for all $i \in [k-1]$, $p_{\mathcal{T}'_j,i}(v) - p_{\mathcal{T}'_j,k-i}(v) = 0$ for all $v \in V(H) \setminus \{x,x'\}$ and 
\begin{align*}
	p_{\mathcal{T}'_j,i}(x) - p_{\mathcal{T}'_j,k-i}(x) 
	& = 
	p_{\mathcal{T}'_j,k-i}(x')  - p_{\mathcal{T}'_j,i}(x') 
	= 
	\mathbbm{1}_{i=j} - \mathbbm{1}_{i=k-j} - j (\mathbbm{1}_{i=1} - \mathbbm{1}_{i=k-1}).
		\end{align*}
Moreover, when $k$ is even and $j = k/2$, $p_{\mathcal{T}'_{k/2},k/2} (v) \equiv \mathbbm{1}_{v \in \{x,x'\}}  \bmod{2}$.
\end{corollary}

\begin{proof}
We will proceed by induction on~$j$.
Let~$\mathbf y$ be a~$k$-tuple such that~$Y = \{y_i : i \in [k] \setminus \{ j \} \} \in N(x) \cap N(x')$.
By Lemma~\ref{lem:jbasic}, there is a $C^{\smash{(k)}}_\ell$-decomposable $k$-graph~$G_j\!=\! G_j(\mathbf y, x, x')$  such that 
\begin{enumerate}[label={\rm(\roman*)}]
	\item $G_j[\{x,x'\} \cup Y] = \{x \cup Y, x' \cup Y\}$,
	\item $|G_j| = 2\ell$, and
	\item there exists a tour-trail decomposition~$\mathcal{T}_j$ of~$G_j$ such that 
\begin{align*}
	D(\mathcal{T}_j)  
		= \vect S_j(\mathbf y, x,x')\cup
		\vect S_1(\sigma_1(\mathbf y),x',x)\cup
		\vect S_{j-1}(\sigma_2(\mathbf y), x,x'),
\end{align*}
where $\sigma_1 = j12\dotsb(j-1)(j+1)\dotsb k$ and $\sigma_2 = 2\dotsb k 1$.
\end{enumerate}
Note that, for all  $i \in [k-1]$ and $v \in V(H) \setminus \{x,x'\}$, we have $
p_{\mathcal{T}_j,i}(v) - p_{\mathcal{T}_j,k-i}(v) = 0
$
as each of $\vect S_j(\mathbf y,x,x')$, $\vect S_1(\sigma_1(\mathbf y),x',x)$ and $\vect S_j(\sigma_2(\mathbf y),x,x')$ contributes zero.
Moreover,
\begin{align}
p_{\mathcal{T}_j,i}(x) - p_{\mathcal{T}_j,k-i}(x) 
& = p_{\mathcal{T}_j,k-i}(x') - p_{\mathcal{T}_j,i}(x') \nonumber \\
&= \left( \mathbbm{1}_{i=j} - \mathbbm{1}_{i=k-j} \right) -  \left( \mathbbm{1}_{i=1} - \mathbbm{1}_{i=k-1} \right) - \left( \mathbbm{1}_{i=j-1} - \mathbbm{1}_{i=k-j+1} \right).
\label{eqn:B_j1}
\end{align}

For $j = 2$, we set $B_2 = G_2$ and we are done. 
For $j >2$, there exists $B_{j-1}(x,x')$ edge-disjoint from $G_j$, by our induction hypothesis. 
Let $\mathcal{T}_{j-1}'$ be the corresponding tour-trail decomposition. 
Set $B_j = G_j \cup B_{j-1}(x,x')$. 
Clearly $|B_j| = |G_j| + |B_{j-1}| =   2(j-1) \ell$. 
Note that $B_j$ is $C^{\smash{(k)}}_\ell$-decomposable and has a tour-trail decomposition~$\mathcal{T}'_j = \mathcal{T}_j \cup \mathcal{T}_{j-1}'$.
Together with~\eqref{eqn:B_j1}, we deduce that $\mathcal{T}'_j$ satisfies the desired properties. 
The moreover statement can be verified similarly. 
\end{proof}

\subsection{Swapper gadgets}
The construction of \secondgadget{} gadgets requires more steps.
We start with the following proposition. 

\begin{proposition} \label{prop:k-1}
Let $1/n \ll \eps \ll 1/\ell, 1/k $ with $k \ge 3$ and $\ell \ge k^2 - k + 1$.
Let $H$ be a~$k$-graph on $n$ vertices with $\tmindeg(H) \geq \eps n$.
Let $x,x'\in V(H)$ be distinct vertices and~$\mathbf y$ be a $(k-1)$-tuple of~$V(H)$ such that~$\{x,x'\}\cup \{y_i\colon 2\leq i\leq k-1\}$ is of size~$k$.
Then, there exists a vertex $y_k \in V(H)$, a $C^{\smash{(k)}}_{\ell}$-decomposable $k$-graph~$F_1 = F_1(\mathbf y,x,x')$ in~$H$  with $|F_1| = 3 \ell$ and a tour-trail decomposition~$\mathcal{T}_1$ such that 
\begin{align*}
	D(\mathcal{T}_1) = \begin{cases}
	\vect S_1(\mathbf y, x,x') \cup \{	x x' , x x' 	\} & \text{if $k = 3$,}\\
	\vect S_1( \mathbf y, x,x') \cup \{	x x' y_{k} \dots y_4,  	y_4 \dots y_{k} x x' 	\} & \text{if $k \ge 4$.}
	\end{cases}
\end{align*}
Moreover, $F_1[ \{x,x',y_2 , \dots,y_{k-1} \}] = \emptyset$, that is, $xx'y_2\dotsb y_{k-1}$ is not an edge of $F_1$.
\end{proposition}

\begin{proof}
Let $y_k \in N(x y_2 \dotsb y_{k-1} ) \cap N( x' y_2 \dotsb y_{k-1} ) \cap N(x x' y_3 \dotsb  y_{k-1})$, which exists by our assumption (here for $k =3$ we consider $y_3 \dotsb y_{2}$ to be empty).
By Lemma~\ref{lem:findcycle}, there exist three tight cycles of length~$\ell$
\begin{align*}
C_{1} & = y_2 \dotsb y_{k} x' u_{k+1} \dotsb u_{\ell}, \\
C_{2} & = x y_2 \dotsb y_{k} v_{k+1} \dotsb v_{\ell} \text{, and} \\
C_{3} & = y_3 \dotsb y_{k} x'  x w_{k+1} \dotsb w_{\ell},
\end{align*}
where $u_i, v_{i}, w_i$ are all distinct new vertices. 
Let $T_1 = C_1 \cup C_2 \cup C_3$.
Consider the trail decomposition~$\mathcal{T}_1$ of $T_1$ such that 
\begin{align*}
\mathcal{T}_1 = 
\left\{
\begin{array}{c}  
	y_3 \dotsb y_{k} x' u_{k+1} \dotsb u_{\ell} y_2 \dotsb y_{k}, \\
	x' y_2 \dotsb y_{k},
	\\
	y_2  \dotsb y_{k} v_{k+1} \dotsb v_{\ell} x y_2 \dotsb y_{k-1},\\
	y_2 \dotsb y_{k}x, 
	\\
	y_4 \dotsb y_{k} x' x w_{k+1} \dotsb w_{\ell} y_3 \dotsb y_{k} x',\\
	y_3 \dotsb y_{k} x x'
\end{array}
\right\}
\end{align*}
 and so
\begin{align*}
D(\mathcal{T}_1) = 
\left\{
\begin{array}{cc}  
	\hcancel{x' y_k \dotsb y_{3}}, &
	\hcancel{y_2 \dotsb y_{k}}, \\
	y_{k-1} \dotsb y_2 x', & 
	\hcancel{y_2 \dotsb y_{k}}, 
	\\
	\hcancel{y_k  \dotsb y_{2}} &
	x y_2 \dotsb y_{k-1},\\
	\hcancel{y_k \dotsb y_{2}}, &
	\hcancel{y_3 \dotsb y_{k}x},
	\\
	x x' y_k \dotsb y_{4}, &
	\hcancel{y_3 \dotsb y_{k} x'},\\
	\hcancel{x y_k \dotsb y_{3}}, & 
	y_4 \dotsb y_{k} x x'
\end{array}
\right\}
\end{align*}
as required.
\end{proof}


We construct a \secondgadget{} gadget of the form $T_1(\mathbf y,\mathbf y', x,x')$ in the next proposition.

\begin{proposition}[Swapper gadget -- case~$j=1$]\label{prop:1switch}
Let $1/n \ll \eps \ll 1/\ell, 1/k $ such that $k \ge 3$ and $\ell \ge k^2 - k + 1$.
Let $H$ be a $k$-graph on $n$ vertices with $\tmindeg(H) \geq \eps n$.
Consider vertices $x, x' \in V(H)$ and $(k-1)$-tuples $\mathbf y$,~$\mathbf y'$ of~$V(H)$ such that both~$\{x,x'\}\cup \{y_i\colon 2\leq i\leq k-1\}$ and $\{x,x'\}\cup \{y_i'\colon 2\leq i\leq k-1\}$ are of size~$k$.
Then there exists a $C^{\smash{(k)}}_\ell$-decomposable $k$-graph~$T_1 = T_1(\mathbf y, \mathbf y', x, x')$ in~$H$ and a tour-trail decomposition~$\mathcal{T}_1$ of~$T_1$ such that
\begin{enumerate}[label={\rm(\roman*)}]
	\item $T_1[\{x,x'\}\cup \{y_1,\dots y_{k-1}\}\cup \{y_1',\dots,y'_{k-1}\}] = \emptyset$ and $|T_1| = 2 \ell k$ and
	\item $D(\mathcal{T}_1) = \vect T_1(\mathbf y, \mathbf y', x,x')$.
\end{enumerate}
\end{proposition}

\begin{proof}
We apply Proposition~\ref{prop:k-1} twice, the first time for $x$, $x'$ and~$\mathbf y$ as input and the second time with $x'$, $x$ and $\mathbf y'$ as input (we exchange the r\^oles of~$x$ and~$x'$).
This yields distinct vertices $y_{k},y'_{k} \in V(H)$ and two $C_{\ell}^{\smash{(k)}}$-decomposable $k$-graphs $F =F_1(\mathbf y, x,x')$ and $F' = F_1(\mathbf y',x',x)$ such that 
\begin{enumerate}[label={\rm(a$_{\arabic*}$)}]
	\item $V(F) \cap V(F') \subseteq \{x,x'\} \cup \{y_i, y_i'\colon i\in [k-1]\} $ and $|F|= |F'| = 3 \ell$ and
	\item there exists a tour-trail decomposition~$\mathcal{T}$ of $F \cup F'$ such that
	\[ D(\mathcal{T}) = T_1(\mathbf y,\mathbf y', x',x)  \cup 
	\{	
	x x' y_{k} \dotsb y_4,  \ 
	y_4 \dotsb y_{k} x x' , \ 
	x' x y'_{k} \dotsb y'_4, \ 
	y'_4 \dotsb y'_{k} x' x	\}, 
	\] where, for $k=3$, we interpret the strings $y_k \dotsb y_4$ and $y'_k \dotsb y'_4$ to be empty.
\end{enumerate}
\smallskip
If $k=3$, then $D(\mathcal{T}) = T_1(\mathbf y,\mathbf y', x',x) \cup
	\{	
	x' x , 
    x' x , 
	x x' ,  
	 x x'	\} = T_1(\mathbf y,\mathbf y', x',x)$, thus we are done.
	 So we may assume that $k \geq 4$.
Note that if we had~$y_i = y'_i$ for all $i \in \{4, \dotsc, k\}$, then $D(\mathcal{T})$ is as desired.
Thus, our aim is to `replace'~$y_i, y'_i$ with a new vertex $z_i$, for each $i = \{4, \dots, k\}$.
We do this in turns, as follows. 
For each $i \in \{4, \dotsc, k\}$, let 
\begin{align*}
	z_i \in N(x x' z_4 \dotsb z_{i-1}   y_i \dotsb y_{k}) \cap N( x x' z_4 \dotsb z_{i-1} y'_i \dotsb y'_{k})
\end{align*}
be a new vertex (here we consider~$z_4 \dotsb z_3$ to be empty).
Consider the two ordered edges $z_i \dotsb z_4  x x'  y_{k} \dotsb y_i$ and $z_i \dots z_4 x' x  y'_{k} \dotsb y'_i$ and apply Lemma~\ref{lem:findcycle} to obtain two tight cycles of length~$\ell$, 
\begin{align*}
	C^i = z_i \dotsb z_4  x x'  y_{k} \dotsb y_i  v^i_{k+1} \dotsb v^{i}_{\ell}
	\text{ and }
	D^i = z_i \dotsb z_4 x' x  y'_{k} \dotsb y'_i  w^i_{k+1} \dotsb w^{i}_{\ell},
\end{align*}
such that $v^i_{j},w^i_j$ are new vertices. 
Define a tour-trail decomposition~$\mathcal{T}^i$ of $C^i \cup D^i$ such that 
\begin{align*}
\mathcal{T}^i = 
\left\{
\begin{array}{cc}  
z_{i-1} \dotsb z_4  x x'  y_{k} \dotsb y_i  v^i_{k+1} \dotsb v^{i}_{\ell} z_i \dotsb z_4  x x' y_{k} \dotsb y_{i+1}, 
\\ 
z_i \dotsb z_4  x' x  y_{k} \dotsb y_i,
\\
z_{i-1} \dotsb z_4  x' x  y'_{k} \dotsb y'_i  w^i_{k+1} \dotsb w^{i}_{\ell} z_i \dotsb z_4  x' x  y'_{k} \dotsb y'_{i+1} 
\\
z_i \dotsb z_4  x x'  y'_{k} \dotsb y'_i
\end{array}
\right\}.
\end{align*}
Note that 
\begin{align*}
	D(\mathcal{T}^i) 
	& = 
\left\{
\begin{array}{cc}
	y_i \cdots y_{k} x' x z_4 \cdots z_{i-1}, & 
	z_i \dots z_4  x x'  y_{k} \cdots y_{i+1},
	\\
	y_{i+1} \cdots y_{k} x x' z_4 \cdots z_{i}, &
	z_{i-1} \cdots z_4  x' x  y_{k} \cdots y_i,
	\\
	y'_i \cdots y'_{k} x x' z_4 \cdots z_{i-1},&
	z_i \cdots z_4  x' x  y'_{k} \cdots y'_{i+1},
	\\
	y'_{i+1} \cdots y'_{k} x' x z_4 \cdots z_{i}, &
	z_{i-1} \cdots z_4  x x'  y'_{k} \cdots y'_i
\end{array}
\right\}
\\
& = 
\left\{
\begin{array}{cc}
	y_i \cdots y_{k} x' x z_4 \cdots z_{i-1}, \\ 
	y'_i \cdots y'_{k} x x' z_4 \cdots z_{i-1},\\
	z_{i-1} \cdots z_4  x' x  y_{k} \cdots y_i,\\
	z_{i-1} \cdots z_4  x x'  y'_{k} \cdots y'_i
\end{array}
\right\}
\cup 
\left\{
\begin{array}{cc}
	z_i \cdots z_4  x x'  y_{k} \cdots y_{i+1},\\
	z_i \cdots z_4  x' x  y'_{k} \cdots y'_{i+1}
	\\
	y_{i+1} \cdots y_{k} x x' z_4 \cdots z_{i},\\
	y'_{i+1} \cdots y'_{k} x' x z_4 \cdots z_{i}	
\end{array}
\right\}
.
\end{align*}
When $i=k$, then the second set can be simplified to an empty set. 
Note that $D( \mathcal{T}^4), \dotsb,D( \mathcal{T}^k)$ forms a `telescoping' set of residual di-$(k-1)$-graphs, so we deduce that 
\begin{align*}
  D\left( \bigcup_{ 4 \leq i \leq k} \mathcal{T}^i\right)
= \{	
	x x' y'_{k} \dotsb y'_4,  \ 
	y'_4 \dotsb y'_{k} x x' , \ 
	x' x y_{k} \dotsb y_4, \ 
	y_4 \dotsb y_{k} x' x	\}.
\end{align*}
We are done by setting $T_1 = F \cup F' \cup  \bigcup_{4 \leq i \leq k} (C^i \cup D^i)$ 
and $\mathcal{T}_1 = \mathcal{T} \cup  \bigcup_{4 \leq i \leq k} \mathcal{T}^i$. 
\end{proof}

We can now describe the general version of the \secondgadget{} gadget $T_j$, for all $j \in [k-1]$.

\begin{lemma}[Swapper gadget -- general case]\label{lem:jswitch}
Let $1/n \ll \eps \ll 1/\ell, 1/k $ such that $k \ge 3$ and $\ell \ge k^2 - k + 1$.
Let $H$ be a $k$-graph on $n$ vertices with $\tmindeg(H) \geq \eps n$.
Let $j \in [k-1]$ and consider distinct vertices $x, x' \in V(H)$ and $(k-1)$-tuples $\mathbf y$,~$\mathbf y'$ of~$V(H)$ such that both~$\{x,x'\}\cup \{y_i\colon i\in [k-1]\setminus \{j\}\}$ and $\{x,x'\}\cup \{y_i'\colon i\in [k-1]\setminus \{j\}\}$ are of size~$k$.
Then there exists a $C^{\smash{(k)}}_\ell$-decomposable $k$-graph~$T_j = T_j(\mathbf y, \mathbf y',x,x')$ and a tour-trail decomposition~$\mathcal{T}_j$ of~$T_j$ such that
\begin{enumerate}[label={\rm(\roman*)}]
	\item $T_j[\{x,x'\}\cup \{y_1,\dots y_{k-1}\}\cup \{y_1',\dots,y'_{k-1}\}] = \emptyset$ and $|T_j| \le 3^j \ell k$ and
	\item $	D(\mathcal{T}_j) = \vect T_j(\mathbf y,\mathbf y', x, x')$.
\end{enumerate}
\end{lemma}

\begin{proof}
We proceed by induction on $j$.
Note that Proposition~\ref{prop:1switch} implies the case when $j =1$, so we may assume that $j \ge 2$. 
Let 
$$y_k  \in N(xy_1\dots y_{j-1}y_{j+1}\dots y_{k-1})  \cap N(x'y_1\dots y_{j-1}y_{j+1}\dots y_{k-1} )$$
be a new vertex. 
By Lemma~\ref{lem:jbasic}, there exists a $C^{\smash{(k)}}_\ell$-decomposable $k$-graph $G_j = G_j(\mathbf y, x, x')$ with $|G_j| = 2 \ell$ and a tour-trail decomposition~${\mathcal{G}}_j$ of~$G_j$ satisfying
\begin{align*}
	D(\mathcal{G}_j) = 
		\vect S_j(\mathbf y, x,x') \cup 
		\vect S_{1}(\sigma_1(\mathbf y),x',x) \cup  
		\vect S_{j-1}(\sigma_2(\mathbf y),x,x'),
\end{align*}
where $\sigma_1 = j12\dots(j-1)(j+1)\dots k$ and $\sigma_2 = 2\dots k 1$.
Analogously, there is a~$C^{\smash{(k)}}_\ell$-decomposable $k$-graph $G'_j = G'_j(\mathbf y' , x', x)$  with $|G'_j| = 2 \ell$ and a tour-trail decomposition~$\mathcal{G}'_j$ of~$G_j$ satisfying
\begin{align*}
	D(\mathcal{G}'_j) = 
		\vect S_j(\mathbf y', x',x) \cup 
		\vect S_{1}(\sigma_1(\mathbf y'),x,x') \cup  
		\vect S_{j-1}(\sigma_2(\mathbf y'),x',x).
\end{align*}
Note that 
\begin{align}
 |G_j \cup G'_j|= 4 \ell \label{eqn:SYj}
\end{align}
and 
\begin{multline}
    D(\mathcal{G}_j) \cup 	D(\mathcal{G}'_j) 
    =
    \vect T_j(\mathbf y,\mathbf y', x,x') \cup
	\vect T_1(\sigma_1(\mathbf y),\sigma_1(\mathbf y'), x',x) \cup
	\vect T_{j-1}(\sigma_2(\mathbf y),\sigma_2(\mathbf y'), x,x') .
	\label{eqn:SYj2}
\end{multline}
Due to the induction hypothesis, there are $C^{\smash{(k)}}_\ell$-decomposable $k$-graphs 
\begin{align*}
T_1=T_1(\sigma_1(\mathbf y), \sigma_1(\mathbf y'),x,x') 
\quad \text{and} \quad
T_{j-1}=T_{j-1}(\sigma_2(\mathbf y), \sigma_2(\mathbf y'),x',x)
\end{align*}
and a tour-trail decomposition~$\mathcal{T}^*_j$ of~$T_1 \cup T_{j-1}$
such that their union $T^*_j = T_1\cup T_{j-1}$ satisfies
\begin{enumerate}[label={\rm(\roman*$'$)}]
	\item $T^*_j[\{x,x'\} \cup \{y_1,\dots,y_{k-1}\}\cup\{y_1,\dots,y_{k-1}\}] = \emptyset$ and $|T^*_j| \le 2 \cdot 3^{j-1}\ell k $,\label{itm:S*j1}
	\item $D(\mathcal{T}^*_j) =
	\vect T_1(\sigma_1(\mathbf y),\sigma_1(\mathbf y'), x,x')  \cup  
	\vect T_{j-1}(\sigma_2(\mathbf y),\sigma_2(\mathbf y'), x',x)$.
	\label{itm:S*j2}
\end{enumerate}
We set $T_j= G_j \cup G'_j \cup {T}^*_j$
and $\mathcal{T}_j = \mathcal{G}_j \cup \mathcal{G}'_j \cup \mathcal{T}^*_j$. 
By \eqref{eqn:SYj} and \ref{itm:S*j1}, we deduce that 
\begin{align*}
|T_j| \le 4 \ell + 2 \cdot 3^{j-1}\ell k \le 3^j \ell k.
\end{align*}
Moreover \eqref{eqn:SYj2} and \ref{itm:S*j2} imply that $	D(\mathcal{T}_j) = \vect T_j(\mathbf y,\mathbf y', x,x')$ as required.
\end{proof}


\section{Transformers II: Tour-trail decompositions} \label{section:tourtrail}

Here, we use the gadgets constructed in the previous section to prove that any $C^{\smash{(k)}}_\ell$-divisible $k$-graph can be augmented to a new, not-too-large, subgraph which has a tour decomposition.
That is the content of the next crucial lemma, whose proof will be given at the end of this section.
Note that we only require $\deg_G(v)$ is divisible by~$k$ for all vertices~$v \in V(G)$ instead of $C^{\smash{(k)}}_\ell$-divisible. 

\begin{lemma} \label{lem:tourdecom}
Let $1/n \ll \eps \ll \rho, 1/\ell, 1/k $ with $k \ge 3$ and $\ell \ge k^2 - k + 1$.
Let $H$ be a~$k$-graph on $n$ vertices with $\tmindeg(H) \geq \eps n$.
Let $G$ be a $k$-graph with $V(G) \subseteq V(H)$ and~$m = |V(G)| \le \eps n^{1/k(k+1)}$ such that $\deg_G(v)$ is divisible by $k$ for all $v \in V(G)$.
Then $H-G$ contains a $C_{\ell}^{\smash{(k)}}$-decomposable subgraph~$J$ such that $G \cup J$ has a tour decomposition, $J[V(G)] = \emptyset$ and  $|G \cup J| \le 3^{k+2} k^4 \ell^2 m^{k+1}$.

Moreover, if $G'$ has an edge-bijective homomorphism to~$G$ with $V(G') \subseteq V(H) \setminus V(G)$, then we have $H-G-G'$ contains a subgraph~$J'$ such that $G' \cup J'$ is edge-bijective homomorphic to $G \cup J$ and $J'[V(G')] = \emptyset$.
\end{lemma}

As described before, the lemma will be proven by starting with any tour-trail decomposition of $G$, and adding gadgets to it repeatedly.
We will first use
\firstgadget{} gadgets to make sure we have a balanced tour-trail decomposition, and then we will use \secondgadget{} gadgets to eliminate any remaining trails one by one.

\subsection{Basic properties}
We begin by stating basic properties of any tour-trail decomposition $\mathcal T$.
Recall that $p_{\mathcal{T} ,i}(v)$ is the number of (directed) edges of $\mathcal{T}$ where $v$ is the $i$th vertex and that the definition of balanced tour-trail decomposition is given in~\eqref{dfn:balanced}. 

\begin{proposition} \label{prop:sumpi}
Let $H$ be a $k$-graph such that $\deg(v)$ is divisible by $k$ for every vertex~$v \in V(H)$. 
Let $\mathcal{T}$ be a tour-trail decomposition of~$H$.
Then, for each $v \in V(H)$, 
\begin{align*}
 \sum_{i \in [k-1]} i p_{\mathcal{T},i}(v) \equiv 0 \bmod k.
\end{align*}
Moreover, if $k$ is even and $\mathcal{T}$ is balanced, then $p_{\mathcal{T},k/2}(v)$ is even for all $v \in V(H)$. 
\end{proposition}

\begin{proof}
Note that only trails in~$\mathcal{T}$ contribute to $ \sum_{i \in [k-1]} i p_{\mathcal{T},i}(x)$.
Moreover, for any tour~$C$, we have $\deg_C(v) \equiv 0 \bmod k$ for all $v \in V(H)$. 
Hence by deleting all tours in~$\mathcal{T}$ and their corresponding edges in~$H$, we may assume that $\mathcal{T}$ consists of trails only.

Fix $v \in V(H)$. 
Let $T=v_1\dotsb v_t$ be a trail in~$\mathcal T$ and consider~$I\subseteq [t]$ such that~$v_i=v$ for each~$i\in I$. 
Let
\begin{align*}
\phi_T(v) =\sum_{i\in[t-k+1]} \left( \mathbbm{1}_{v=v_i}+\mathbbm{1}_{v=v_{i+1}}+\dots+\mathbbm{1}_{v=v_{i+k-1}} \right).
\end{align*}
Observe that every time the trail `passes through~$v$' it increases~$\phi_T(v)$ by~$k$ except if it is at the beginning or the end of~$T$. 
More precisely, it is not hard to check that
\begin{align*}
\phi_T(v) 
&=
k\vert I\cap [k, t-k]\vert + \sum_{i\in [k-1]\cup [t-k+1,t]} \left( \mathbbm{1}_{v=v_i}+\dots+\mathbbm{1}_{v=v_{i+k-1}} \right) \, 
\\ & \equiv \sum_{i\in [k-1]}ip_{i,T}(v)\, \bmod k.
\end{align*}
On the other hand, it is easy to see that~$\phi_T(v) = \sum_{e\in T}\mathbbm{1}_{v\in e}$.
Thus, summing over all trails in~$\mathcal T$, we get 
\begin{align*}
    0 \equiv \deg_H(v) = \sum_{T\in \mathcal T} \phi_T(v) \equiv \sum_{i\in [k-1]}ip_{i,\mathcal T}(v)\, \bmod k.
\end{align*}
Furthermore, suppose that $k$ is even and $\mathcal{T}$ is balanced, then
\begin{align*}
   \sum_{i \in [k-1]} i p_{\mathcal{T},i}(v) = \sum_{i \in [k/2-1]} k p_{\mathcal{T},i}(v) + (k/2) p_{\mathcal{T},k/2}(v)\,.
\end{align*}
Since this is equivalent to~$0 \bmod{k}$, we get that $p_{\mathcal{T},k/2}(v)$ is even.
\end{proof}

\subsection{Balancing}

Recall that any $k$-graph admits a trail decomposition, by orienting edges arbitrarily.
We begin by using \firstgadget{} gadgets (as given by Corollary~\ref{corjbasic}) repeatedly to obtain a balanced tour-trail decomposition of~$G$.

\begin{lemma} \label{lem:balancetourdecom}
Let $1/n \ll \eps \ll \rho, 1/\ell, 1/k $ with $k \ge 3$ and $\ell \ge k^2 - k + 1$.
Let $H$ be a $k$-graph on $n$ vertices and $\tmindeg(H) \geq \eps n$.
Let $G$ be a $k$-graph with $V(G) \subseteq V(H)$, $m =  |V(G)|  \le \eps n^{1/k^2}$ and such that $\deg_G(v)$ is divisible by $k$ for all $v \in V(G)$.
Then~$H-G$ contains a $C_{\ell}^{\smash{(k)}}$-decomposable~$J$ such that $G \cup J$ has a balanced tour-trail decomposition~$\mathcal{T}^*$, $J [V(G)] = \emptyset$ and $|G \cup J| \le \ell m^{k+1}$.
\end{lemma}

\begin{proof}
Let $\mathcal{T}_0$ be an arbitrary tour-trail decomposition of~$G$. 
The following claim forms the basis of our proof and allows us to adjust the values of $p_{\mathcal{T}_0, k/2}(v)$.

\begin{claim} \label{clm:balnace_1}
Suppose $k$ is even. Then $H\!-\!G$ contains a $C_{\ell}^{\smash{(k)}}$-divisible subgraph~$J_{{k/2}}$ with~$|J_{k/2}|\!\le\! k \ell  m/2$ and $J_{k/2}[V(G)] = \emptyset $ such that there exists a tour-trail decomposition~$\mathcal{T}_{{k/2}}$ of~$J_{k/2}$ satisfying, for all $v \in V(H)$,
\begin{align*}
	p_{\mathcal{T}_{{k/2}}, k/2 }(v) & \equiv  \mathbbm{1}_{\text{$p_{\mathcal{T}_{0}, k/2}(v)$ is odd}} \bmod{2},\\
	|p_{\mathcal{T}_{{k/2}}, 1}(v) -  p_{\mathcal{T}_{{k/2}}, k-1}(v)| & = (k/2) \mathbbm{1}_{\text{$p_{\mathcal{T}_{0}, k/2}(v)$ is odd}},\\
	p_{\mathcal{T}_{{k/2}}, i}(v) -  p_{\mathcal{T}_{{k/2}}, k-i}(v) & = 0 \text{ if $i \in [2,k-2]$.}
\end{align*}
\end{claim}

\begin{proofclaim}
Note that $\sum_{v \in V(H)} p_{\mathcal{T}_{0}, k/2}(v) = |D(\mathcal{T}_0)|$ is twice the number of trails in $\mathcal{T}_0$.
Thus, there is an even number of vertices $v$ such that $p_{\mathcal{T}_{0}, k/2}(v)$ is odd.
Suppose that $v_1, \dots, v_{2s} \in V(H)$ are precisely those vertices, so $p_{\mathcal{T}_{0}, k/2}(v_j)$ is odd for each~$j\in[2s]$. 
For each $j \in [s]$, we apply Corollary~\ref{corjbasic} and obtain edge-disjoint \firstgadget{} gadgets~$B_{k/2} ( v_{2j-1} ,  v_{2j} )$ in~$H-G$. 
Let $J_{k/2}$ be the union of these \firstgadget{} gadgets. 
Clearly, $|J_{k/2}| \le s k \ell/2  \le k \ell  m/2$.
Let $\mathcal{T}_{k/2}$ be the tour-trail decomposition of~$J_{k/2}$, which is the union of the corresponding tour-trail decompositions of each~$B_{k/2} ( v_{2j-1} ,  v_{2j} )$.
For all $v \in V(H)$, we have $	p_{\mathcal{T}_{{k/2}}, k/2 }(v) \equiv 1 \bmod 2$ if and only if~$v \in \{v_1, \dotsc, v_{2s}\}$, which proves the first required property. 
We can deduce the other two properties using the properties of the corresponding tour-trail decomposition of each~$B_{k/2} (    v_{2j-1} ,  v_{2j} )$.
This proves the claim.
\end{proofclaim}

If $k$ is even, then let $J_{k/2}$ be given by Claim~\ref{clm:balnace_1}, otherwise just set $J_{k/2} = \emptyset$.
The next claim allows us to adjust the values of $p_{\mathcal{T}_0, i}(v)$ for $2\le  i < k/2$.

\begin{claim}  \label{clm:balnace_2}
For each $2 \leq i < k/2$, $H-G-J_{k/2}$ contains a $C_{\ell}^{\smash{(k)}}$-divisible subgraph~$J_i$ such that there exists a tour-trail decomposition~$\mathcal{T}_i$ of~$J_i$ satisfying, for all $v \in V(H)$ and $i' \in [k-1]$,
\begin{align*}
	p_{\mathcal{T}_i, i' } ( v ) - p_{\mathcal{T}_i, k-i'} ( v ) = 
	\begin{cases}
			p_{\mathcal{T}_0, k-i' } ( v ) - p_{\mathcal{T}_0, i'} ( v ) & \text{if $i' \in \{ i, k-i \}$,}\\
			- i (p_{\mathcal{T}_0, k-i' } ( v ) - p_{\mathcal{T}_0, i'} ( v )) & \text{if $i' \in \{1,k-1\}$,}\\
			0 & \text{otherwise.}
	\end{cases}	
\end{align*}
Moreover, $|J_i| \le 4 i \ell  \binom{m}k$, $J_i[V(G)] = \emptyset$ and the $J_i$'s are edge-disjoint. 
\end{claim}

\begin{proofclaim}
Let $2 \leq i < k/2$.
Suppose that we have already constructed subgraphs $J_2, \dots, J_{i-1}$.
We now construct $J_i$ as follows. 
Let $H' = H-G - J_{k/2} - \bigcup_{i' \in [2,i-1]} J_{i'}$.

For all $i' \in [k-1]$, note that $ \sum_{v \in V(H)} 
 p_{\mathcal{T}_0, i' } ( v )  = | D ( \mathcal{T}_0 ) |$ and so 
 if we define~$w_i(v):=p_{\mathcal{T}_0, i } ( v ) - p_{\mathcal{T}_0, k-i} ( v )$ then
\begin{align}\label{eq:zerosum}
	 \sum_{v \in V(H)} w_i(v)
  = 0.
\end{align}
Define a ($2$-uniform) multidigraph~$H_i$ on~$V(H)$ such that, for all $v \in V(H)$,
\begin{align}\label{eq:deg-weights}
    d^-_{H_i}(v)  = 	\max \{ w_i(v), 0 \} 
	   \text{ and }
   d^+_{H_i}(v)  = 	\max \{ -w_i(v), 0 \} \,.
\end{align}
Note that $H_i$ can be constructed greedily.%
    \footnote{
        Indeed, let $V^+ = \{v \in V(H) : w(v) >0\}$ and $V^- = \{v \in V(H) : w(v) < 0\}$.
        Note that  \eqref{eq:zerosum} implies that $\sum_{v \in V^+} w(v) = - \sum_{v \in V^-} w(v)$.
    Thus $H_i$ can be obtained by adding appropriate edges from~$V^+$ to~$V^-$.
}
Note that $|H_i| = |D(\mathcal{T}_0)|$ is twice the number of trails in~$\mathcal{T}_0$, so $|H_i| \le 2 \binom{m}k$.
For each directed edge $xy \in H_i$, we apply Corollary~\ref{corjbasic} and obtain edge-disjoint \firstgadget{} gadgets~$B_i (  x ,  y )$ in~$H'$. 
Let $J_i$ be the union of these \firstgadget{} gadgets. 
Clearly, $|J_i| \le 2 i \ell  |H_i| \le 4 i \ell \binom{m}k$.
Let $\mathcal{T}_i$ be the tour-trail decomposition of~$J_i$, which is the union of the corresponding tour-trail decompositions of each~$B_i (  x ,  y )$.
It is straightforward to check that $\mathcal{T}_i$ has the desired properties, which proves the claim. 
\end{proofclaim}

For each $2 \leq i < k/2$, let $J_i$ be given by Claim~\ref{clm:balnace_2}.
Together with $J_{k/2}$, we have then edge-disjoint $J_2, \dotsc, J_{\lfloor k/2 \rfloor}$ for any $k$.
Let $H^* = H - G - \bigcup_{2 \leq i \leq \lfloor k/2 \rfloor} J_i$ and set~$\mathcal{T}' = \mathcal{T}_0 \cup \bigcup_{2 \leq i \leq \lfloor k/2 \rfloor} \mathcal{T}_i$. 
Recall that $\sum_{i \in [k-1]} p_{\mathcal{T}_0, i } ( v )$ is the number of $(k-1)$-tuples in $D(\mathcal{T}_0)$ containing~$v$, so $\sum_{i \in [k-1]} p_{\mathcal{T}_0, i } ( v ) \le 2 \binom{m}k$.
For $i \in \{2, \dotsc, k-2\}$ and $v \in V(H)$, we have 
\begin{align}
 p_{\mathcal{T}',i}(v)  & = p_{\mathcal{T}',k-i}(v) , \label{eqn:balance1}\\
	p_{\mathcal{T}',k/2}(v)  & \equiv 0 \bmod{2} \text{ if $k$ is even, and}  \label{eqn:balance2}
\\
| p_{\mathcal{T}',1}(v) - p_{\mathcal{T}',k-1}(v) | & \le \sum_{i \in [k-1]} i p_{\mathcal{T}_0, i } ( v ) \le 2 (k-1) \binom{m}k. \label{eqn:balance3}
\end{align}
Moreover, $p_{\mathcal{T}',1}(v) = p_{\mathcal{T}',k-1}(v)$ for all $v \in V(H) \setminus V(G)$. 

We now balance $p_{\mathcal{T}',1}(v)$ and $p_{\mathcal{T}',k-1}(v)$ as follows.
Note that $\lceil k/2 \rceil - 1$ is the largest integer which is strictly less than $k/2$.
We have 
\begin{align*}
 p_{\mathcal{T}',1}(v)  - p_{\mathcal{T}',k-1}(v)  
& \overset{\eqref{eqn:balance1}}{=}  p_{\mathcal{T}',1}(v)  - p_{\mathcal{T}',k-1}(v) + \sum_{2 \le i \le \lceil k/2 \rceil - 1} i \left( p_{\mathcal{T}',i}(v) - p_{\mathcal{T}',k-i}(v) \right)\\
& \overset{\eqref{eqn:balance2}}{\equiv}  \sum_{1 \leq i \leq k-1} i p_{\mathcal{T}',i}(v) 
\overset{\text{Prop.~\ref{prop:sumpi}}}{\equiv} 0 \pmod{k}.
\end{align*}
For each $v \in V(G)$, let 
\begin{align*}
b(v) = (p_{\mathcal{T}',1}(v) -  p_{\mathcal{T}',k-1}(v)  ) / k ,
\end{align*}
so $b(v) \in  \mathbb{Z}$. 
Define (greedily) a multi-digraph $H_1$ on $V(H)$ such that, for all $v \in V(H)$,
\begin{align*}
	d^-_{H_1}(v)  = 	\max \{ b(v) , 0 \}
	\text{ and }
	d^+_{H_1}(v)  = 	\max \{ - b(v) , 0 \}.
\end{align*}
By~\eqref{eqn:balance3}, $\Delta^{\pm}(H_1) \le 2 \binom{m}k$ and $|H_1| \le m \binom{m}{k}$.
For each directed edge $xy \in H_1$, we apply Corollary~\ref{corjbasic} to obtain edge-disjoint \firstgadget{} gadgets~$B_{k-1} (  x ,  y )$ in~$H^*$. 
We call~$J_1$ the union of these \firstgadget{} gadgets.
Clearly, $|J_1| \le 2 (k-1) \ell  |H_1| \le 2 (k-1) \ell m \binom{m}k$.

Let $J = \bigcup_{1 \leq i \leq \lfloor k/2 \rfloor} J_i$.
Note that 
\begin{align*}
|G \cup J| \le \binom{m}k+ 2 (k-1) \ell m \binom{m}k + \sum_{2 \leq i \leq \lfloor k/2 \rfloor} 4 i \ell \binom{m}k \le 2 k \ell m \binom{m}k \le \ell m^{k+1}.
\end{align*}
Let $\mathcal{T}_1$ be the tour-trail decomposition of~$J_1$, which is the union of the corresponding tour-trail decompositions of each~$B_{k-1} (  x ,  y )$.
Note that given a $B_{k-1} (  x ,  y )$ and its tour-trail decomposition~$\mathcal{T}=\mathcal T(x,y)$,
we have, for all $v \in V(H)$ and $2 \leq i \leq k - 2$, 
\begin{align*}
	p_{\mathcal{T},1}(v)  - p_{\mathcal{T},k-1}(v)  = - k (\mathbbm{1}_{v = x} - \mathbbm{1}_{v = y})\, 
	\text{ and }
	p_{\mathcal{T},i}(v)  - p_{\mathcal{T},k-i}(v)  = 0\,.
\end{align*}
Hence $\mathcal{T}^* = \mathcal{T}_1 \cup \mathcal{T}'$ is a balanced tour-trail decomposition of $J \cup G$, as required.
\end{proof}

\subsection{Focusing}

The following lemma shows that all the residual $D(\mathcal{T}^*)$ can be moved onto a fixed set of $k-1$ vertices.

\begin{lemma} \label{lem:balancetourdecom2}
Let $1/n \ll \eps \ll 1/\ell, 1/k $ with $k \ge 3$ and $\ell \ge k^2 - k + 1$.
Let $H$ be a~$k$-graph on $n$ vertices and $\tmindeg(H) \geq \eps n$.
Let $G$ be a $k$-graph with $V(G) \subseteq V(H)$ and $m = |V(G)|  \le \eps n^{1/k^2}/2$ such that $\deg_G(v)$ is divisible by $k$ for all vertices $v \in V(G)$.
Suppose that $m$ is prime and $|G| < m/k$. 
Suppose that $G$ has a balanced tour-trail decomposition~$\mathcal{T}$. 
Let $z_1, \dots, z_{k-1} \in V(H) \setminus V(G)$ be distinct vertices. 
Then~$H-G$ contains a~$C_{\ell}^{\smash{(k)}}$-decomposable~$J^*$ such that~$|J^{*}|\leq 3^k k \ell m$, $J^{*}[V(G)]=\emptyset$ and~$G\cup J^{*}$ has a balanced tour-trail decomposition~$\mathcal {T}^{*}$ with~$|D(\mathcal{T}^*)| \leq 3m$ and satisfying
$$p_{\mathcal {T}^{*},i}(v) = 0 \text{ for all~$v\notin \{z_i, z_k-i\}$} 
\qquad \text{and} \qquad
p_{\mathcal{T}^*,i}(z_i) = p_{\mathcal{T}^*,i}(z_{k-i})\,,
$$
for all~$i\in[k-1]$.
\end{lemma}

We first outline the proof of Lemma~\ref{lem:balancetourdecom2}.
Take $z_1, \dots, z_{k-1}\in V$ distinct vertices. 
Recall the definition of~$r_j$ at the beginning of Section~\ref{section:gadgets}.
For $1 \leq i \leq \lfloor k/2 \rfloor$, define the functions $\zeta_i, \bar \zeta_i : V^{k-1} \rightarrow V^{k-1} $ be such that, for $\mathbf a  = a_1\dotsb a_{k-1}\in  V^{k-1}$,
\begin{align*}
	\zeta_i (\mathbf a) & =  r_{k-i}(  r_{i}(\mathbf a, z_{i}) , z_{k-i} ) 
	=  a_1 \dotsb a_{i-1} z_i a_{i+1} \dotsb a_{k-i-1} z_{k-i} a_{k-i+1} \dotsb  a_{k-1} \text{ and }\\
	\bar\zeta_i( \mathbf a ) & =   r_{k-i}(  r_{i}(\mathbf a, z_{k-i} ) , z_{i} ) = a_1 \dotsb a_{i-1}  z_{k-i}  a_{i+1}  \dotsb a_{k-i-1}  z_{i}  a_{k-i+1} \dotsb   a_{k-1}.
\end{align*}
That is, $\zeta_i$ replaces the $i$th and $(k-i)$th vertices with the vertices $z_{i}$ and~$z_{k-i}$, respectively, whereas $\bar\zeta_i$ replaces the $i$th and $(k-i)$th vertices with vertices $z_{k-i}$ and~$z_{i}$, respectively.
If $k$ is even and $i = k/2$, then both $\zeta_{k/2}$ and $\bar\zeta_{k/2}$ replace the $(k/2)$th vertex with $z_{k/2}$, this is~$\zeta(\mathbf a) = \bar\zeta(\mathbf a) = r_{k/2}(\mathbf a, z_{k/2})$.

We say that a tour-trail decomposition~$\mathcal{T}'$ is an \emph{$i$-convert} of~$\mathcal{T}$ if $D(\mathcal{T})$ can be partitioned into $D_1$ and~$D_2$ of equal size such that $D(\mathcal{T}') = \zeta_i (D_1) \cup \bar\zeta_i ( D_2 )$.
Note that the definitions of $\zeta_i, \bar\zeta_i$ and $i$-convert depend on a choice of $z_1, \dotsc, z_{k-1}$. 
However, such choice will be always clear from the context, so we omit it in the notation.

Let $\mathcal{T}_0$ be a balanced tour-trail decomposition of~$G$.
Our aim is to construct tour-trail decompositions $\mathcal{T}_1,\dots, \mathcal{T}_{\lfloor k/2 \rfloor}$ such that $\mathcal{T}_i$ is an $i$-convert of~$\mathcal{T}_{i-1}$.
Notice that $\mathcal{T}_{\lfloor k/2 \rfloor}$ will be the desired tour-trail decomposition.
Indeed, observe first that for each edge $\mathbf{a} = a_1\cdots a_{k-1} \in D(\mathcal T_0)$ all of its vertices are replaced eventually by a vertex in~$\{z_1,\dots,z_{k-1}\}$ in~$\mathcal{T}_{\lfloor k/2 \rfloor}$. 
And second, for every~$1\leq i<k/2$, each directed edge increases the value of both $p_{\mathcal{T}_i,i}(z_i)$ and $p_{\mathcal{T}_i,i}(z_{k-i})$ by exactly one with respect to their value in the previous tour-trail decomposition~$\mathcal {T}_{i-1}$.
In particular, $p_{\mathcal{T}_i,i}(z_i)=p_{\mathcal{T}_i,i}(z_{k-i})$ for every~$1\leq i< k/2$. 

We will also need the following notation.
Let $\mathcal{T}$ be a tour-trail decomposition and~$1 \leq i < k/2$. 
Define $A_i(\mathcal{T})$ to be the multidigraph on~$V(H)$ such that every ordered tuple $v_1 \dotsb v_{k-1}$ in~$D(\mathcal{T})$ corresponds to a distinct directed edge $v_{i}v_{k-i}$ in~$A_i(\mathcal{T})$.

The following is immediate from our definition of $i$-convert and~$A_{j}(\mathcal{T})$.

\begin{proposition}\label{prop:convert}
Let $k \ge 3$ and $1 \leq i < k/2$.
Let $V$ be a set of vertices, and let~$z_1, \dots, z_{k-1} \in V$ be distinct vertices. 
Let $\mathcal{T}$ and $\mathcal{T}'$ be tour-trail decompositions of two (not necessarily of the same) subgraphs in $V$.
Suppose $\mathcal{T}'$ is an \emph{$i$-convert} of~$\mathcal{T}$.
Then~$A_{j}(\mathcal{T}') = A_{j}(\mathcal{T})$ for all $1 \leq j < k/2$ such that $j \neq i$.
\end{proposition}

The next lemma shows that we can always get a tour-trail decomposition such that~$A_i(\mathcal{T})$ is strongly connected for all $1 \leq i < k/2$ and spans~$V(G)$.
The proof is simple and follows by greedily adding new arcs to~$A_i(\mathcal T)$.

\begin{lemma} \label{lemma:convert}
Let $1/n \ll \eps \ll 1/\ell, 1/k $ with $k \ge 3$ and $\ell \ge k^2 - k + 1$.
Let $H$ be a~$k$-graph on $n$ vertices with $\tmindeg{(H)} \geq \eps n$.
Let $U \subseteq V(H)$ with $|U| = m \le \eps n$ and~$m$ is a prime number. 
Then there exists a $C_{\ell}^{\smash{(k)}}$-decomposable~$J_0$ such that $|J_0| = \ell m$, $J_0[U] = \emptyset$, $J_0$ has a balanced tour-trail decomposition~$\mathcal{T}_0$ satisfying $V(D(\mathcal{T}_0)) \subseteq U$ and, for all $1 \leq i < k/2$, $A_i(\mathcal{T}_0)$ is a strongly connected multidigraph which spans~$U$ and $|A_i(\mathcal{T}_0)| = 2m$.
\end{lemma}

\begin{proof}
Let $u_1, \dots, u_m$ be an enumeration of~$U$. 
Consider $j \in [m]$.
We apply Lemma~\ref{lem:findcycle} to obtain a copy $C_j$ of~$C_{\ell}^{\smash{(k)}}$ with $V(C_j) = u_{j+1} \dotsb u_{j+k-1} w_{j,k} \dotsb w_{j,\ell}$, where $w_{j,j'}$ are new vertices. 
For its trail decomposition~$\mathcal{T}_j$ we consider~$C_j$ to be a trail
\begin{align*}
u_{j+1} \dotsb u_{j+k-1} w_{j,k} \dotsb w_{j,\ell} u_{j+1} \dotsb u_{j+k-1}.
\end{align*}
Then $u_{j+i}u_{j+k-i} ,u_{j+k-i}u_{j+i} \in  A_{i}( \mathcal{T}_j )$ for all $1 \le i < k/2$.
Let $J_0 = \bigcup_{j \in [m]} C_j$ and $\mathcal{T}_0 = \bigcup_{j \in [m]} \mathcal{T}_j$ (without simplification). 
Note that $|J_0| = \ell m$ as each $C_j$ has $\ell $ edges.
Note that 
\begin{align*}
|A_i(\mathcal{T}_0)| = \sum_{j \in [m]} |A_i(\mathcal{T}_j)| = \sum_{j \in [m]} |D(\mathcal{T}_j)| = 2m.
\end{align*}
Moreover, since~$A_i(\mathcal T_0)$ consists on arcs~$u_{j_1}u_{j_2}$ and $u_{j_1}u_{j_2}$ where~$j_2-j_1 = k-2i$, and thus it spans~$U$.
Moreover, recall~$m$ is prime and in particular~$k-2i$ does not divide~$m$. 
Therefore,~$A_i$ never closes a cycle of length smaller than~$m$ and thus~$A_i(\mathcal{T}_0)$ is strongly connected for all $1 \leq i < k/2$.
\end{proof}

We are now ready to prove Lemma~\ref{lem:balancetourdecom2}.

\begin{proof}[Proof of Lemma~\ref{lem:balancetourdecom2}]
Apply Lemma~\ref{lemma:convert} with $U = V(G)$ and $H = H \setminus \{z_1, \dots, z_{k-1}\}$, we obtain a $C^{\smash{(k)}}_{\ell}$-decomposable graph~$J_0 \subseteq H -G$ such that 
\begin{align}
	|J_0| = \ell m,
	\label{eqn:convert1}
\end{align}
and $J_0$ has a balanced tour-trail decomposition~$\mathcal{T}_0$ such that, for all $1 \leq i < k/2$, $A_i(\mathcal{T}_0)$ is a connected multidigraph which spans~$V(G)$. 

Let $G_0^* = G \cup J_0$ and $\mathcal{T}_0^* = \mathcal{T} \cup \mathcal{T}_0$, considering all trails, without doing any further simplification even if it is possible to do so.
For all $1 \leq i < k/2$, $A_i = A_i(\mathcal{T}_0^*)$.
Thus~$A_i$ is a connected multidigraph spanning~$V(G)$.
Observe that since $\mathcal{T}^*_0$ is balanced, then~$A_i$ is Eulerian.
Let $s$ be such that $|D(\mathcal{T}^*_0)| = 2s$.
Note that 
\begin{align}
	2s = |A_i| = |D(\mathcal{T}^*_0)| =  2|G| + |A_i(\mathcal{T}_0)| \le 2 m /k + 2 m \le  3 m.
	\label{eqn:convert2}
\end{align}

\begin{claim} \label{clm:convert}
There exist edge-disjoint $k$-graphs $J_0, \dots, J_{\lfloor k/2 \rfloor}$ in~$H - G$ such that 
\begin{enumerate}[label={\rm(\roman*)}]
	\item each $J_{i}$ is $C_{\ell}^{\smash{(k)}}$-decomposable and $|J_{i}| \le 2 \cdot 3^{i}  k \ell s$, \label{itm:convert1}	
	\item $G \cup J_0 \cup \bigcup_{j \in [i]}J_j$ has a balanced tour-trail decomposition~$\mathcal{T}_{i}^*$, and \label{itm:convert2}
	\item $\mathcal{T}_{i}^*$ is an $i$-convert of~$\mathcal{T}_{i-1}^*$. \label{item:convert3}
\end{enumerate}
\end{claim}

We first show that the claim implies the lemma. 
Set $J^* =J_0 \cup \bigcup_{j \in [\lfloor k/2 \rfloor]}J_j$ and set~$\mathcal{T}^* = \mathcal{T}_{\lfloor k/2 \rfloor}^*$.
Clearly, $J^*$ is $C^{\smash{(k)}}_{\ell}$-decomposable since each $J_j$ is.
Note that 
\begin{align*}
|J^*| \le |J_0| + \sum_{i \in [\lfloor k/2 \rfloor]}  |J_i| 
\overset{\eqref{eqn:convert1}, \ref{itm:convert1}}{\le}
 \ell m +  2 k \ell  s\sum_{i \in [\lfloor k/2 \rfloor]} 3^{i} 
\overset{{\eqref{eqn:convert2}}}{\le}
 3^k k \ell  m.
\end{align*}
Since $\mathcal{T}^*$ is balanced by~\ref{itm:convert2}, \ref{item:convert3} implies that each $v_1 \dotsb v_{k-1} \in D( \mathcal{T}^*)$ satisfies
\begin{align*}
	v_1 \dotsb v_{k-1} \in \{z_1,z_{k-1}\} \times\{z_2,z_{k-2}\} \times \dots \times \{ z_{k-1},z_{1}\}.
\end{align*}
Hence, for all $i \in [k-1]$ and $v \in V(H)$, $p_{\mathcal{T}^*,i}(v) = 0$ unless $v \in \{z_i,z_{k-i}\}$ as desired. 
Also $|D(\mathcal{T}^*)| = |D(\mathcal{T})| = 2s \le 3m$ by \ref{item:convert3} and~\eqref{eqn:convert2}.
Therefore to complete the proof, it remains to prove Claim~\ref{clm:convert}.

\begin{proofclaim}
Suppose that we have already found $J_0, \dots, J_{i-1}$ and we now construct~$J_i$ as follows. 

\medskip
\noindent \emph{Case 1: $i = k/2$}. 
We first prove the case $i = k/2$ as it is simpler and illustrates some of the key ideas. 
If we are in this case, then $k$ is even and, by Proposition~\ref{prop:sumpi}, we have~$p_{\mathcal{T}_{k/2-1},k/2}(v) \equiv 0 \pmod{2}$ for all~$v \in V(H)$. 
Also $| D(\mathcal{T}^*_{k/2-1}) | = | \mathcal{T}^*_{0} | = 2s$.
Take an arbitrary enumeration of~$D(\mathcal{T}^*_{k/2-1})$ into $\mathbf {b}_1, \dots, \mathbf {b}_{2s}$ such that $b_{2j-1,k/2} = b_{2j,k/2}$ for all $j \in [s]$, where $\mathbf {b}_j =  b_{j,1} \dotsb b_{j,k-1}$ (here we use the fact that $p_{\mathcal{T}_{k/2-1},k/2}(v)$ is even).
For every $j \in [s]$, let
\begin{align*}
	b^*_j & = b_{2j-1,k/2} = b_{2j,k/2}.
\end{align*}
Apply Lemma~\ref{lem:jswitch} to obtain a \secondgadget{} gadget~$T^j_{k/2} = T_{k/2}(\mathbf {b}_{2j-1}, \mathbf {b}_{2j}^{-1},  z_{k/2}, b^*_j )$ with a tour-trail decomposition~$\mathcal{T}^j$ such that 
\begin{align*}
D(\mathcal{T}^j) 
= \vect T_{k/2}(\mathbf {b}_{2j-1}, \mathbf {b}_{2j}^{-1},z_{k/2}, b^*_j ) =
\{ \mathbf {b}^{-1}_{2j-1}, \mathbf {b}^{-1}_{2j},  \zeta_{k/2}(\mathbf {b}_{2j-1}), \zeta_{k/2}(\mathbf {b}_{2j})\}. 
\end{align*}
We may further assume that these $T^j$ are edge-disjoint. 

Let $J_{k/2}$ be the union of these \secondgadget{} gadgets and $\mathcal{T}_{k/2}$ be the union of their tour-trail decompositions, together with~$\mathcal{T}_{k/2-1}$. 
Note that $|J_{k/2} | \le 3^{k/2} \ell k s$, and since
\begin{align*}
	D(\mathcal{T}^*_{k/2}) &= D(\mathcal{T}^*_{k/2-1}) \cup \bigcup_{j \in [s]} \{ \mathbf {b}^{-1}_{2j-1}, \mathbf {b}^{-1}_{2j},  \zeta_{k/2}(\mathbf {b}_{2j-1}), \zeta_{k/2}(\mathbf {b}_{2j})\} 
	\\ &= \{ \mathbf {b}_{j}, \mathbf {b}^{-1}_{j}, \zeta_{k/2}(\mathbf {b}_{j}) : j \in [2s]\}
	=  \{ \zeta_{k/2}(\mathbf {b}_{j}) : j \in [2s]\},
\end{align*}
we deduce $\mathcal{T}_{k/2}$ is a $(k/2)$-convert of~$\mathcal{T}_{k/2-1}$, as required.

\medskip
\noindent \emph{Case 2: $i < k/2$}. 
By \ref{item:convert3} and Proposition~\ref{prop:convert}, we deduce that $A_i(\mathcal{T}_{i-1}) = A_i$ and so, it is an Eulerian multidigraph. 
Hence, there exists an enumeration of~$D(\mathcal{T}_i)$ as~$\mathbf {b}_1, \dots, \mathbf {b}_{2s}$, which corresponds to an Eulerian tour in~$A_i$. 
Let $\mathbf {b}_j = b_{j,1} \dotsb b_{j,k-1}$, so we have~$b_{j,k-i} = b_{j+1,i}$ for all $j \in [2s]$. 

We now replace the $(k-i)$th vertex in each~$\mathbf{b}_{2j-1}$ with~$z_{i}$ and the $i$th vertex in each~$\mathbf{b}_{2j}$ with~$z_{i}$, as follows. 
For every~$j \in [s]$ let
\begin{align*}
	b^*_j & = b_{2j-1,k-i} = b_{2j,i}.
\end{align*}
Apply Lemma~\ref{lem:jswitch} to obtain a \secondgadget{} gadget $T^j = T_{i}(\mathbf b_{2j},\mathbf b_{2j-1}^{-1}, z_{i}, b^*_j ,)$ with a tour-trail~$\mathcal{T}^j$ such that 
\begin{align*}
D(\mathcal{T}^j) 
= \vect T_{i}(\mathbf b_{2j}, \mathbf b_{2j-1}^{-1}, z_{i}, b^*_j ) = 
\{ \mathbf {b}^{-1}_{2j-1}, \mathbf {b}^{-1}_{2j},  r_{k-i}(\mathbf {b}_{2j-1}, z_i), r_{i}(\mathbf {b}_{2j}, z_i)\}. 
\end{align*}
Recall that~$r_{i}(\mathbf{b}_{2j}, z_{i})$ and $r_{k-i}(\mathbf{b}_{2j-1},z_{i})$ correspond to the tuples $\mathbf {b}_{2j}$ and $\mathbf {b}_{2j-1}$ replacing $b^*_j$ with~$z_{i}$ (see definition at the beginning of Section~\ref{section:gadgets}). 
Note that 
\begin{align*}
	D\big(\mathcal{T}_i \cup \bigcup_{j \in [s]} \mathcal{T}^j\big) 
	&= \bigcup_{j \in [2s]} \{ \mathbf {b}_{2j-1}, \mathbf {b}_{2j}, \mathbf {b}^{-1}_{2j-1}, \mathbf {b}^{-1}_{2j}, r_{k-i}(\mathbf {b}_{2j-1},z_i), r_{i}(\mathbf {b}_{2j} ,z_i)\} \\ 
	&= \bigcup_{j\in [s]}\{r_{k-i}(\mathbf {b}_{2j-1},z_i), r_{i}(\mathbf {b}_{2j}, z_i)\}.
\end{align*}
Equivalently, $D(\mathcal{T}_i \cup \bigcup_{j \in [2s]} \mathcal{T}^j)$ is obtained from $D(\mathcal{T}_i)$ by replacing the $(k-i)$th vertex in $\mathbf {b}_{2j-1}$ and $i$th vertex in $\mathbf {b}_{2j}$ with $z_{i}$.

By considering the pairs of tuples $r_i(\mathbf {b}_{2j},z_i), r_{k-i}(\mathbf {b}_{2j+1},z_i)$, 
a similar argument implies that we can replace the $i$th vertex in $r_{k-i}(\mathbf {b}_{2j+1},z_i)$ and $(k-i)$th vertex in~$r_i(\mathbf {b}_{2j},z_i)$ with~$z_{k-i}$.

Let $J_{i}$ be the union of these \secondgadget{} gadgets and $\mathcal{T}_i$ be the union of $\mathcal{T}_{i-1}$ and the corresponding tour-trail decomposition. 
Notice that $\mathcal{T}_i$ is an $i$-convert of~$\mathcal{T}_{i-1}$ and that~$| J_i | \le 3^{j} \ell k \cdot 2 s $.
This finishes the proof of Claim~\ref{clm:convert}.
\end{proofclaim}

As discussed, this finishes the proof of the lemma.
\end{proof}

\subsection{Untangling the last arcs}
Observe that after Lemma~\ref{lem:balancetourdecom2} in the previous subsection we have found a tour-trail decomposition in which the arcs of its residual digraph lie in a small set of~$k-1$ vertices~$z_1,\dots, z_{k-1}$. 
Here we show how to `untangle' those arcs in such a way that all `cancel' each other.
After this cancellation, the trails from the tour-trail decomposition are removed and we obtain a tour decomposition. 

\begin{lemma} \label{lem:balancetourdecom3}
Let $1/n \ll \eps \ll 1/\ell, 1/k $ with $k \ge 3$ and $\ell \ge k^2 - k + 1$.
Let $H$ be a $k$-graph on $n$ vertices with $\tmindeg(H) \geq \eps n$.
Let $G$ be a $k$-graph with $V(G) \subseteq V(H)$ and $|V(G)| \le \eps n$. 
Let $z_1, \dots, z_{k-1} \in  V(G)$ be distinct vertices. 
Suppose that $G$ has a balanced tour-trail decomposition~$\mathcal{T}_1$ such that $ | D(\mathcal{T}_1)| \le 5m $ and
$$
p_{\mathcal{T}_1,i}(v) = 0 \text{ for all~$v\notin \{z_i, z_k-i\}$} 
\qquad \text{and} \qquad
p_{\mathcal{T}_1,i}(z_i) = p_{\mathcal{T}_1,i}(z_{k-i})\,,
$$
for all~$i\in[k-1]$. 
Then $H-G$ contains a  $C^{\smash{(k)}}_{\ell}$-decomposable subgraph~$J$ such that $|J| \le k^3  \ell |D(\mathcal{T}_1)|$, $J[V(G)] = \emptyset$ and $G \cup J$ has a tour decomposition.
\end{lemma}

\begin{proof}
We simplify $\mathcal{T}_1$ as much as possible. 
Let $|D(\mathcal{T}_1)| = 2s$ and then we have, for all~$i \in [k-1]$,
\begin{align}
	p_{\mathcal{T}_1,i}(z_i) = p_{\mathcal{T}_1,i}(z_{k-i})
 = 
		\begin{cases}
			s & \text{if $i \ne k/2$,}\\
			2s & \text{if $i = k/2$.}
		\end{cases}
		\label{eqn:bad1}
\end{align}
We now colour $\mathbf {b} \in D(\mathcal{T}_1)$ red if $\mathbf {b}$ starts at~$z_1$ (i.e. the first vertex of~$\mathbf b$ is~$z_1$), and blue otherwise.  
So there are $s$ red $(k-1)$-tuples and $s$ blue $(k-1)$-tuples in~$D(\mathcal{T}_1)$. 
Ideally, we would like to transform all red~$(k-1)$-tuples to $z_1 \dotsb z_{k-1}$ and all blue~$(k-1)$-tuples to $z_{k-1} \dotsb z_{1}$ (so that they would cancel out).
A $(k-1)$-tuple is~\emph{$i$-bad} if~$z_i$ is at the ``wrong place''.
More precisely, an $i$-bad red $(k-1)$-tuple (and an $i$-bad blue $(k-1)$-tuple) will be of form $v_1 \dotsb v_{i-1} z_{k-i} v_{i+1} \dotsb v_{k-i-1} z_{i} v_{k-i+1} \dotsb v_{k-1}$ (and $v_1 \dotsb v_{i-1} z_{i} v_{i+1} \dotsb v_{k-i-1} z_{k-i} v_{k-i+1} \dotsb v_{k-1}$, respectively), where $\{v_j, v_{k-j}\} = \{z_j, z_{k-j}\}$ for all $j \in [k-1] \setminus \{ i, k-i \}$. 
Note that, by definition, all red tuples start with $z_1$ and therefore are not~$1$-bad.
Similarly, blue tuples cannot by~$1$-bad and thus, there are no~$1$-bad~$(k-1)$-tuples.
Analogously, there are no $(k-1)$-tuple which is~$(k-1)$-bad or $k/2$-bad.
If a $(k-1)$-tuple is $i$-bad, then it is also $(k-i)$-bad.
By~\eqref{eqn:bad1}, the number of red $i$-bad tuples is equal to the number of blue $i$-bad tuples.

We claim that there exist edge-disjoint $k$-graphs $J_2, \dots, J_{\lceil k/2 \rceil-1}$ in~$G - H$ such that, for $2 \leq i < k/2$,
\begin{enumerate}[label={\rm(\roman*)}]
	\item each $J_{i}$ is $C^{\smash{(k)}}_{\ell}$-decomposable and $|J_{i}| \le 8 i \ell k s$,
	\item $G  \cup \bigcup_{j \in [2,i]}J_j$ has a balanced tour-trail decomposition~$\mathcal{T}_{i}$,
	\item for all $j \in [k-1]$ and $v \in V(H)$, $p_{\mathcal{T}_i,j}(v) = 0$ unless $v \in \{z_{j},z_{k-j}\}$,
	\item	 $|D(\mathcal{T}_i)| = |D(\mathcal{T}_{i-1})|$, and
	\item $D(\mathcal{T}_i)$ contains no $j$-bad $(k-1)$-tuple for all $j \in [i]$.
\end{enumerate}

Suppose that, for some $2 \leq i < k/2 - 1$,  we have constructed  $J_2, \dots, J_{i-1}$.
We describe the construction of $J_i$ as follows. 

Let $\mathbf {a}_1, \dots, \mathbf {a}_t$ be the $i$-bad red $(k-1)$-tuples in $D(\mathcal{T}_{i-1})$ and $\mathbf {b}_1, \dots, \mathbf {b}_t$ be the $i$-bad blue $(k-1)$-tuples in $D(\mathcal{T}_{i-1})$.
Consider any $j \in [t]$.
Let 
\begin{align*}
 \mathbf {a}_j & = a_{j,1} \dotsb a_{j,i-1} z_{k-i} a_{j,i+1} \dotsb a_{j,k-i-1} z_{i} a_{j, k-i+1} \dotsb a_{j,k-1}\\
 \mathbf {b}_j & = b_{j,1} \dotsb b_{j,i-1} z_{i} b_{j,i+1} \dotsb b_{j,k-i-1} z_{k-i} b_{j,k-i+1} \dotsb b_{j,k-1}.
\end{align*}
We now mimic the argument in the proof of Lemma~\ref{lem:balancetourdecom2} to swap $z_i$ and~$z_{k-i}$ in $\mathbf {a}_j$'s and~$\mathbf {b}_j$'s.
However, we are unable to construct \secondgadget{} gadgets $T_i(\mathbf a_j, \mathbf b_j^{-1}, z_i,z_{k-i})$ as both $\mathbf {a}_j$ and~$\mathbf {b}_j$ contain both $z_i$ and~$z_{k-i}$.
To overcome this problem, we first replace~$z_i$ with a new vertex~$w$ (so $\mathbf {a}_j$ and~$\mathbf {b}_j$ are now free of~$z_i$).
After this is done, then we can replace~$z_{k-i}$ with~$z_i$, and finally we replace~$w$ with~$z_{k-i}$.
We now formalise the proof as follows. 

Let $w \in V(H) \setminus \{z_{i'} : i' \in [k-1] \}$ be a new vertex.
Apply Lemma~\ref{lem:jswitch} to obtain three edge-disjoint \secondgadget{} gadgets
\begin{multline*}
    T_i(\mathbf a_j,\mathbf b_j^{-1}, w, z_{k-i}), \quad
    T_i(r_{i}(\mathbf a_j, w)^{-1}, r_{k-i}(\mathbf b_j, w), z_{i}, z_{k-i})
    \quad\text{and } \\
    T_i(r_{k-i}(\mathbf a_j,z_{k-i}), r_{i}(\mathbf b_j, z_{k-i})^{-1}, z_i, w) 
\end{multline*}
Let~$T^j$ be its union and~$\mathcal T^j$ be the union of their tour-trail decompositions. 
It is not hard to check that after cancellation we obtain
$$D(\mathcal{T}^j)=\{ 
\mathbf a_j^{-1},\,
\mathbf b_j^{-1},\,
\zeta_i(\mathbf a_j),\,
\bar\zeta_i(\mathbf b_j)\}\,.$$
Note that $\zeta_i(\mathbf a_j)$ and 
$\bar\zeta_i(\mathbf b_j)$ are not~$i$-bad.

Let $J_i = \bigcup_{j \in [t]} T^j$ and let $\mathcal{T}_{i} = \mathcal{T}_{i-1} \cup \bigcup_{j \in [t]} \mathcal{T}^j$ be the corresponding tour-trail decomposition of~$G  \cup \bigcup_{2 \leq j \leq i}J_j$.
This finishes the construction of $J_i$.

Now we have constructed $J_i$ for all $2 \leq i < k/2$.
We set $J = \bigcup_{2 \leq i < k/2} J_i$, so~$|J| \le  2 k^3  \ell s = k^3  \ell |D(\mathcal{T}_1)|$.
Note that $\mathcal{T}_{\lceil k/2 \rceil -1}$ is a balanced tour-trail decomposition of~$G \cup J$ without any bad $(k-1)$-tuple. 
Therefore, after cancellation, $D(\mathcal{T}_{\lceil k/2 \rceil -1})$ is empty, implying that $\mathcal{T}_{\lceil k/2 \rceil -1}$ is a tour decomposition. 
\end{proof}

\subsection{Proof of Lemma~\ref{lem:tourdecom}}
We now put the pieces together to prove Lemma~\ref{lem:tourdecom}.

\begin{proof}[Proof of Lemma~\ref{lem:tourdecom}]
Apply Lemma~\ref{lem:balancetourdecom} to obtain a $C^{\smash{(k)}}_{\ell}$-decomposable~$J_1$ in $H-G$ such that $|G \cup J_1| \le \ell m^{k+1}$
and $G \cup J_1$ has a balanced tour-trail decomposition~$\mathcal{T}_1$. 

Let $m_1$ be a prime between $k \ell m^{k+1}$ and $2k \ell m^{k+1}$ (this exists by Bertrand's postulate).
Add isolated vertices to $G \cup J_1$ to obtain a subgraph $G_1$ of~$H$ such that 
\begin{align*}
	|V(G_1)| = m_1 \text{ and } |G_1| \le m_1/k.
\end{align*}

Let $z_1,\dots, z_{k-1} \in V(H) \setminus V(G_1)$. 
Apply Lemma~\ref{lem:balancetourdecom2} (with $G_1, \mathcal{T}_1$ playing the rôles of~$G ,\mathcal{T}$ ) to obtain a $C^{\smash{(k)}}_{\ell}$-decomposable~$J_2$ in $H-G_1$ such that
\begin{align*}
	|J_2| \le 3^k k \ell m_1 
\end{align*}
and $G_2 = G_1 \cup J_2$ has a balanced tour-trail decomposition~$\mathcal{T}_2$ satisfying, for all $i \in [k-1]$ and $v \in V(H)$, $p_{\mathcal{T}_2,i}(v) = 0$ unless $v \in \{z_i,z_{k-i}\}$ and $|D(\mathcal{T}_2)| \le 3 m_1 $.

Apply Lemma~\ref{lem:balancetourdecom3} (with $G_2, \mathcal{T}_2$ playing the rôles of $G, \mathcal{T}_1$) to acquire a $C^{\smash{(k)}}_{\ell}$-decomposable subgraph~$J_3$ in~$H-G_2$ such that 
\begin{align*}
|J_3| \le k^3  \ell |D(\mathcal{T}_2)| \le 3 k^3 \ell m_1
\end{align*}
and $G_2 \cup J_3$ has a tour decomposition.
We set $J = J_1 \cup J_2 \cup J_3$ and so $G \cup J$ has a tour decomposition and 
\begin{align*}
	|G \cup J| \le m_1 +  3^k k \ell m_1 + 3 k^3 \ell m_1 \le 3^{k+2} k^4 \ell^2 m^{k+1}.
\end{align*}

To prove the `moreover' statement, we simply repeat the identical construction of~$J$ to obtain~$J'$ in $H'=H-G-J - G'$ as $\delta^{(3)} (H') \ge \eps n /2$. 
For example, $J_1$ consists of edge-disjoint balancer gadgets. 
Then $J'_1$ will consists of the same number of edge-disjoint balancer gadgets in the same configurations. 
Thus we can then assume that $G \cup J_1 $ is homomorphic to~$G'\cup J'_1$.
Therefore, we obtain~$J'$ homomorphic to~$J$ and such that~$G\cup J$ is homomorphic to~$G'\cup J'$.
\end{proof}

\section{Transformers III: Proof of Lemma~\ref{lem:transformer}} \label{section:transformer}

Finally, in this section we use the machinery of transformers and tour-trail decompositions to find cycle absorbers and prove that every~$H$ with~$\delta^{(3)}(H) \geq 2\eps n$ contains an absorber for every given leftover~$G\subseteq H$ with~$|V(G)|$ sufficiently small. 
More precisely, we prove Lemma~\ref{lem:transformer}. 

\begin{proof}[Proof of Lemma~\ref{lem:transformer}] Given~$n$, $\eps$, $\ell$, $k$, $m$, and~$H$ as in the statement of the lemma.
Let $\eta: \mathbb{N} \rightarrow \mathbb{N}$ such that~$\eta (x) = 3^{k+3} k^5 \ell^3 x^{k+1}$ and let $W \subseteq V(H)$ be of size~$m \le m_0$. 
By Lemma~\ref{lem:findcycle}, given any two ordered $k$-tuples in $V(H) \setminus W$, $v_1 \dotsb v_k$ and $v_{\ell - k+1} \dotsb v_{\ell}$, the~$k$-graph~$H \setminus W$ contains a tight walk~$v_1 \dotsb v_{\ell}$ on $\ell$ vertices and no repeated vertices outside of $v_1 \dotsb v_k$ and~$v_{\ell - k + 1} \dotsb v_\ell$.
In particular, there is an $\ell$-cycle in $H \setminus W$ covering any arbitrary $k$-tuple $v_1, \dotsc, v_k$ in $V(H) \setminus W$.
Hence, by Lemma~\ref{lem:absorbing}, it suffices to show that $H$ is $(C^{\smash{(k)}}_{\ell}, m_0, m_0 ,\eta)$-transformable for some increasing function $\eta: \mathbb{N} \rightarrow \mathbb{N}$ which satisfies~$\eta(x) \geq x$.
Indeed, this will imply that $H$ is $(C^{\smash{(k)}}_\ell, m_0, m_0, \eta')$-absorbable for some increasing function~$\eta': \mathbb{N} \rightarrow \mathbb{N}$ such that $\eta'(x) \geq x$, as desired.
Observe that here, since~$\eta$ is independent of~$\eps$ and~$n$ so is~$\eta'$.

We will show that $H$ is $(C^{\smash{(k)}}_{\ell}, m_0, m_0 ,\eta)$-transformable.
To do so, let $G_1$ and $G_2$ be two vertex-disjoint $C^{\smash{(k)}}_{\ell}$-divisible $k$-graphs with $V(G_1), V(G_2) \subseteq V(H)$; suppose that $|V(G_1)|,|V(G_2)| \le m_0$ and that there is an edge-bijective homomorphism from $G_1$ to~$G_2$.
Let $W \subseteq V(H) \setminus V(G_1 \cup G_2)$ with~$|W| \leq m_0$.
Let $m = \max \{|V(G_1)|,|V(G_2)|\}$.
Let $H' = H \setminus W$. 
It is enough to show that $H'$ contains a $(G_1,G_2;C^{\smash{(k)}}_{\ell})$-transformer of order at most~$\eta(m)$.
This will be our task from now on.

Apply Lemma~\ref{lem:tourdecom} with~$G_1$ and~$G_2$ playing the r\^oles of~$G$ and~$G'$ to obtain edge-disjoint subgraphs~$J_1$ and~$J_2$ of~$H'-G_1-G_2$ such that 
\begin{enumerate}[label={\rm(a$_{\arabic*}$)}]
	\item  $G_1 \cup J_1$ and $G_2 \cup J_2$ have tour decompositions,
	\item  $J_1$ and $J_2$ are $C^{\smash{(k)}}_{\ell}$-decomposable, \label{itm:trans1}
	\item  $J_1[V(G_1 \cup G_2)]$ and $J_2[V(G_1 \cup G_2)]$ are empty, \label{itm:trans2}
	\item  $|G_1 \cup J_1|,|G_2 \cup J_2| \le 3^{k+2} k^4 \ell^2 m^{k+1}$, and\label{itm:trans4}
	\item  there is an edge-bijective homomorphism~$\phi$ from~$G_1 \cup J_1$ to~$G_2 \cup J_2$.
\end{enumerate}

Let $G'_j= G_j \cup J_j$ for~$j \in [2]$.
We now claim that there exists a $(G_1',G_2';C^{\smash{(k)}}_{\ell})$-transformer $T^*$ with $|T^*| = (\ell-1) |G'_1|$.
Indeed, let $\{A_i : i \in [s] \}$ be a tour decomposition of~$G'_1$, and recall that $\phi$ is an edge-bijective homomorphism from $G'_1$ to $G'_2$.
Therefore,~$\{ \phi(A_i):  i \in [s] \}$ is a tour decomposition of~$G'_2$.
Now, suppose that for some~$ i \in [s]$, we have already constructed edge-disjoint $T_1, \dots, T_{i-1}$ in $H' - G_1' - G_2'$ such that, for $i' \in [i-1]$
\begin{enumerate}[label={\rm(b$_{\arabic*}$)}]
    \item  $T_{i'}$ is an $(A_{i'},\phi(A_{i'});C^{(k)}_{\ell})$-transformer,
	\item  $|T_{i'}| = (\ell-1) |A_{i'}|$, and
	\item  $T_{i'} [V(G_1')] \cup  T_{i'} [V(G_2')]= \emptyset $.
\end{enumerate}
We construct $T_i$ as follows. 
Let $H^* = H' - G_1' - G_2' - \bigcup_{i' \in [i-1]} T_{i'}$.
Let $A_i = v_1\dotsb v_{t}$.
Let~$w_j = \phi(v_j)$ for all $j \in [t]$, so $\phi(A_i) = w_1 \dotsb w_t$.
Note that $\ell - k^2 + k - 1 \geq k^2 - k$ by the choice of $\ell$.
Therefore, by Lemma~\ref{lem:findcycle}, $H^*$ contains tight paths $P_1,\dots, P_s, Q_1, \dots, Q_s$ such that, for each $j \in [t]$,
\begin{enumerate}[label={\rm(c$_{\arabic*}$)}]
	\item  $P_j$ is a tight path of length~$k^2 - k$ from $v_{j+1}v_{j+2} \dotsb v_{j+k-1}$ to~$w_{j} w_{j+1} \dotsb w_{j+k-2}$,
    \item  $Q_j$ is a tight path of length~$\ell - k^2 + k - 1$ starting from~$w_{j} w_{j+1} \dotsb w_{j+k-2}$ and ending in~$v_j v_{j+1} \dotsb v_{j + k-2}$, and
	\item $V(P_j) \setminus V(D(P_j))$ and $V(Q_j) \setminus V(D(Q_j))$ are new vertices.
\end{enumerate}
Each of $ v_{j} v_{j+1} \dotsb v_{j+k-1}  \cup P_j \cup Q_j$ and $w_{j} w_{j+1} \dotsb w_{j+k-1}  \cup P_j \cup Q_{j+1}$ forms a~$C^{\smash{(k)}}_{\ell}$.
Hence, we are done by letting $T_i = \bigcup_{j \in [s]} P_j \cup Q_j$.
This finishes the construction of~$T_i$.
Thus $T^* = \bigcup_{i \in [s]} T_i$ is the desired $(G_1',G_2';C^{\smash{(k)}}_{\ell})$-transformer.

We finish by defining $T = J_1 \cup J_2 \cup T^*$, so 
\begin{align*}
|V(T)| \le k |T| \le k (\ell+1) |G'_1| \overset{\mathclap{\text{\ref{itm:trans4}}}}{\le} 3^{k+3} k^5 \ell^3 m^{k+1} = \eta(m).
\end{align*}
Together with~\ref{itm:trans1} and~\ref{itm:trans2}, we deduce that $T[V(G_1)]$ is empty and
\begin{align*}
 G_1 \cup T = \left( (G_1 \cup J_1) \cup T^* \right) \cup J_2  = (G'_1 \cup T^*) \cup J_2
\end{align*}
is $C^{\smash{(k)}}_{\ell}$-decomposable, and similarly $G_2 \cup T$ is $C^{\smash{(k)}}_{\ell}$-decomposable.
Therefore, $T$ is a~$(G_1,G_2;\smash{C^{(k)}_{\ell}})$-transformer.
\end{proof}

\section{Cover-down lemma} \label{section:coverdown}

In this section we prove Lemma~\ref{lem:coverdownlemma} which is the main step in the iterative part of iterative absorption.
We prove this lemma by induction on $k$, and when dealing with~$k$-uniform hypergraphs we will require results on path decompositions for $(k-1)$-uniform hypergraphs.
To organise our arguments, we define the following two statements for each $k \geq 3$.
The first statement corresponds precisely to the Cover-down lemma for $k$-graphs, while the second one concerns decompositions of $k$-graphs into paths.

\medskip
\begin{enumerate}[label={\rm ($\circledcirc_k$)}]
    \item \label{statement:coverdown-k} 
    For every~$\alpha>0$, 
    there is an $\ell_0 \in \mathbb{N}$ such that for every~$\mu>0$ and every~$n,\ell\in\mathbb{N}$ with $\ell \ge \ell_0$ and $1/n \ll \mu, \alpha$ the following holds.
	Let $H$ be a $k$-graph on $n$ vertices and $U \subseteq V(H)$ with $|U| = \lfloor \alpha n \rfloor$ such that
	\smallskip
\begin{adjustwidth}{20pt}{0pt}
	\begin{enumerate}[label={\rm(CD$_{\arabic*}$)}]
        \item $\delta^{(2)}(H) \geq 2\alpha n$,
	    \item $\delta^{(2)}(H,U) \geq \alpha \vert U\vert$ and \label{it:CDdegree-synth}
	    \item $\deg_H(x)$ is divisible by $k$ for each $x \in V(H) \setminus U$.\label{it:CDdiv-synth}
	\end{enumerate}
\end{adjustwidth}
    \smallskip
	Then $H$ contains  a $C_{\ell}^{\smash{(k)}}$-decomposable subgraph~$F\subseteq H$ such that $H - H[U] \subseteq F$
	and~$\Delta_{k-1}(F[U]) \leq \mu n$.
\end{enumerate}
\bigskip
\begin{enumerate}[label={\rm ($\curvearrowright_k$)}]
    \item \label{statement:pathdec-k} For every~$\ell\geq k$ and for every~$\alpha>0$ there is an $n_0$ such that the following holds.
    Every $k$-graph~$H$ on~$n\geq n_0$ vertices with~$\delta^{(2)}(H) \geq \alpha n$ and~$|H| \equiv 0 \bmod{\ell}$ contains a $P_\ell^{(k)}$-decomposition.
\end{enumerate}
\medskip

Thus, Lemma~\ref{lem:coverdownlemma} can be synthetically stated as follows.

\begin{lemma*}[Cover-down lemma (reprise)] 
   \ref{statement:coverdown-k} holds for every~$k\geq 3$.
\end{lemma*}

We show Lemma~\ref{lem:coverdownlemma} through an induction on~$k$, in which~\ref{statement:pathdec-k} is helpful to enable the induction step.

\begin{lemma}\label{lem:pathtocover}
    For each~$k\geq 3$, if \hyperref[statement:pathdec-k]{\upshape{($\curvearrowright_{k-1}$)}} holds, then \ref{statement:coverdown-k} holds.
\end{lemma}

\begin{lemma}\label{lem:covertopath}
    For each~$k\geq 3$, if \ref{statement:coverdown-k} holds, then \ref{statement:pathdec-k} holds.
\end{lemma}

Assuming the validity of these two last lemmata, Lemma~\ref{lem:coverdownlemma} follows easily, if we are provided with a base case.
For this, we use the following result by Botler, Mota, Oshiro and Wakabayashi~\cite{BMOW2017} on path decompositions in graphs.
A graph is~\emph{$k_\ell$-edge-connected} if after deleting fewer than~$k_\ell$ edges it remains connected.
It is not hard to check that every graph~$G$ is~$\delta^{(2)}(G)$-edge-connected, so the following result immediately yields~\hyperref[statement:pathdec-k]{\upshape{($\curvearrowright_{2}$)}}.

\begin{theorem}[\cite{BMOW2017}] \label{theorem:graphpathdecomposition}
    For each $\ell \geq 1$, there exists $k_{\ell}$ such that each $k_{\ell}$-edge-connected graph whose number of edges is divisible by $\ell$ has a $P^{\smash{(2)}}_\ell$-decomposition.
\end{theorem}

The proof of Lemma~\ref{lem:covertopath} is given in the next subsection.
The proof of Lemma~\ref{lem:pathtocover} will require more effort and is given in Subsection~\ref{subsec:coverdownproof}, after some previous necessary results.

\subsection{Path decompositions: proof of~Lemma~\ref{lem:covertopath} and Theorem~\ref{theorem:path}}\label{subsec:pathdecom}
To see that the bound $\delta_{\smash{P_{\ell}^{(k)}}} \ge 1/2$ holds, consider the following example. 
Take the union of $K^{\smash{(k)}}_{\lfloor n/2 \rfloor }$ and $K^{\smash{(k)}}_{\lceil n/2 \rceil}$ on vertex sets $A$ and $B$, respectively.
Delete a few edges if necessary so that the resulting $k$-graph~$H$ satisfies $|H| \equiv 0 \bmod \ell$ but $|H[A]| \not\equiv 0 \bmod \ell$.
Then $H$ is not $P_{\ell}^{\smash{(k)}}$-decomposable and $\delta(H) \ge (1/2-o(1))n$.
On the other hand, note that the upper bound of Theorem~\ref{theorem:path} can be obtained from Lemmata~\ref{lem:covertopath} and~\ref{lem:coverdownlemma}.

The proof of~Lemma~\ref{lem:covertopath} follows essentially the same strategy we use to prove Theorem~\ref{theorem:main} in Section~\ref{section:mainproof}. 
The Vortex lemma is the same and for the Cover-down lemma we may use~\ref{statement:coverdown-k}, which is assumed to hold as a hypothesis.
To see this, it is enough to notice that for every sufficiently large~$\ell'$ divisible by~$\ell$, a~$C_{\ell'}^{\smash{(k)}}$-decomposable subgraph is~$P_{\ell}^{\smash{(k)}}$-decomposable as well. 
Hence, the only new ingredient needed is the following Absorber lemma for paths.

\begin{lemma}[Absorber lemma for paths] \label{lem:transformer-paths}
Let $1/n \ll \eps \ll 1/\ell, 1/k, 1/m_0 $ with $k \ge 3$.
Let $H$ be a $k$-graph on $n$ vertices with $\delta^{(2)}(H) \geq \eps n$.
Then $H$ is $(P^{\smash{(k)}}_{\ell}, m_0, m_0 ,\eta')$-absorbing for some increasing function $\eta': \mathbb{N} \rightarrow \mathbb{N}$ satisfying $\eta'(x) \geq x$ and independent of $\eps$ and~$n$.
\end{lemma}

\begin{proof}
    Pick an arbitrary edge~$e\in P^{\smash{(k)}}_{\ell}$ and let~$\mathbf{a}$ be any~$k$-tuple in~$V(H)$.
    Since the inequality~$\delta^{(2)}(H)\geq \eps n$ holds, it is easy to see for every~$W\subseteq V(H)\setminus \{v_1,\dots,v_k\}$ of size at most~$m_0$ there is a copy of~$P_\ell^{\smash{(k)}}$ in~$(H\cup \{\mathbf{a}\})\setminus W$ in which~$\mathbf a$ plays the r\^ole of~$e$.
    We do this by simply extending a path (maybe in both directions) starting with~$\mathbf a$, which we can do simply because of $\delta(H) \geq \delta^{(2)}(H) \geq \eps n$.
    This enables us to use Lemma~\ref{lem:absorbing}, and hence, it is enough to prove that~$H$ is~$(P^{\smash{(k)}}_{\ell}, m_0, m_0, \eta)$-transformable for some increasing function~$\eta\colon \mathbb N \to \mathbb N$. 
    
    Let~$\ell_0$ be the smallest number larger than~$k^2-k$ which is divisible by~$\ell$ and define~$\eta(x)=\ell_0 x^2$.
    Let~~$G, G'\subseteq H$ be vertex-disjoint $P_{\ell}^{\smash{(k)}}$-divisible subgraphs such that there is an edge-bijective homomorphism~$\phi$ from~$G$ to~$G'$.
    Also, let  $W \subseteq V(H) \setminus V(G \cup G')$ and suppose $|V(G)|,|V(G')| \le m_0$ and $|W| \le m_0 - \eta(|V(G)|)$. 
    For every edge~$e\in G$ apply Lemma~\ref{lem:findcycle} to find a path~$P_e\subseteq H\setminus W$ between~$e$ and~$\phi(e)$ with precisely~$\ell_0+1$ edges. 
    Since~$\ell_0$ is divisible by~$\ell$,~$T=\bigcup_{e\in G} P_e$ is a~$(G,G'; \smash{P_\ell^{(k)}})$-transformer of size at most~$\ell_0 e(G) \leq \eta(\max\{|V(G)|, |V(G')|\})$.
\end{proof}

We omit further details of proof of Lemma~\ref{lem:covertopath} and reference the reader to the proof of Theorem~\ref{theorem:main}.

\subsection{Well-behaved approximate cycle decompositions} \label{subsection:wellbehaved}

Given a $k$-graph~$H$ such that~$\delta^{(2)}(H)\geq \alpha n$, we find a~$C_\ell^{\smash{(k)}}$-packing~$\mathcal C$ that covers almost all edges of $H$ and such that the leftover is not too concentrated in any~$(k-1)$-tuple. 
Here, a $C_\ell^{\smash{(k)}}$-packing is a set of edge-disjoint copies of~$C_\ell^{\smash{(k)}}$.
More precisely, we have the following lemma.

\begin{lemma}[Well-behaved cycle decompositions] \label{lem:wellbehaved}
    Given~$k\in \mathbb{N}$ and~$\alpha\geq 0$ there is an~$\ell_0\in\mathbb{N}$ such that for every~$\gamma>0$ and $\ell,n\in \mathbb{N}$ with~$\ell\geq \ell_0$ and~$1/n \ll \gamma, \alpha, 1/\ell$ the following holds.
    Let $H$ be a $k$-graph on $n$ vertices with $\delta^{(2)}(H) \ge \alpha n$.
	Then $H$ has a~$C_{\ell}^{\smash{(k)}}$-packing $\mathcal{C}$ such that $\Delta_{k-1}(H - \bigcup \mathcal{C}) \leq \gamma n$.
\end{lemma}

The case~$k=3$ is proven by the last two authors in \cite{PigaSanhuezaMatamala2021} and here we follow the same lines.
Given a $k$-graph~$H$ and an edge $e \in H$, recall that $\mathcal{C}_{\ell}(H)$ and $\mathcal C_\ell(H,e)$ are the family of all $\ell$-cycles in~$H$ and those containing~$e$. 
The proof of Lemma~\ref{lem:wellbehaved} rests in a result by Joos and K\"uhn~\cite{JoosKuhn2021} about fractional $C^k_{\ell}$-decompositions (see the definition at the beginning of Section~\ref{section:lowerbounds}).

\begin{theorem}[Joos and Kühn~\cite{JoosKuhn2021}]\label{them:FracDecom}
    Given~$k\in \mathbb{N}$ and~$\alpha,\mu\geq 0$ there is an~$\ell_0\in\mathbb{N}$ such that for every~$\ell,n\in \mathbb{N}$ with~$\ell\geq \ell_0$ and~$1/n \ll \alpha, 1/\ell$ the following holds.
    Let~$H$ be a~$k$-graph on $n$ vertices with~$\delta^{(2)}(H)\geq \alpha n$.
    Then there is a fractional~$C_\ell^{(k)}$-decomposition~$\omega$ of~$H$ with
    $$(1-\mu)\frac{2|H|}{\Delta_{k-1}(H)^\ell}\leq \omega(C) \leq (1+\mu)\frac{2|H|}{\delta_{k-1}(H)^{\ell}}$$
    for all cycles~$C\in \mathcal C_\ell(H)$.
\end{theorem}

Additionally, we need the following nibble-type matching theorem.
The statement of the theorem is technical, but in our context the conditions are easy to check.
Consider the following parameter
$g(H)=\Delta_1(H)/\Delta_2(H)$ for every~$k$-graph~$H$.

\begin{theorem}[Alon and Yuster~\cite{AlonYuster2005}]\label{thm:alonyuster}
    For every~$\gamma>0$, there is a~$\xi>0$ such that for every sufficiently large~$n$ the following holds.
    Let~$H$ be a $k$-graph on $n$ vertices and let~$U_1,\dots, U_t\subseteq V(H)$ be subsets of vertices with~$\log t\leq {g(H)^{1/(3k-3)}}$ and such that~$\vert U_i\vert \geq 5 g(H)^{1/(3k-3)}\log (g(H)t)$ for every~$i \in [t]$. 
    Suppose that
    \begin{enumerate}[label={\rm(\alph*)}]
        \item $\delta_1(H) \geq (1-\xi)\Delta_1(H)$ and 
        \item $\Delta_1(H)\geq (\log n)^7 \Delta_2(H)$\,.
    \end{enumerate}
    Then $H$ contains a matching such that at most $\gamma \vert U_i\vert$ vertices are uncovered in each~$U_i$.
\end{theorem}

Lemma~\ref{lem:wellbehaved} follows by an straightforward application of Theorems~\ref{them:FracDecom} and~\ref{thm:alonyuster}. 

\begin{proof}[Proof of Lemma~\ref{lem:wellbehaved}]
    Given~$k\in \mathbb{N}$ and $\alpha >0$ fix any~$\mu, \xi<1/3$ and take~$\ell_0$ given by Theorem~\ref{them:FracDecom}.
    Let~$\ell\geq \ell_0$, $\gamma>0$ and let $n$ be sufficiently large for an application of Theorems~\ref{them:FracDecom} and~\ref{thm:alonyuster}.
    
    First, we apply Theorem~\ref{them:FracDecom} to obtain a fractional $C_\ell^{\smash{(k)}}$-decomposition~$\omega$ of~$H$ satisfying
    \begin{align}\label{eq:fracbound}
        \omega(C) 
        \leq
        (1+\mu)\frac{2|H|}{\delta_{k-1}(H)^{\ell}}
        \leq
        \frac{ 3 n^{k}} {\delta^{(2)}(H)^{\ell}}
        \leq 
        \frac{3}{\alpha^{\ell} n^{\ell-k}}\,,
    \end{align}
    for all cycles~$C\in \mathcal C_\ell(H)$.

    Then, consider the auxiliary $\ell$-graph~$F$ with vertex set~$E(H)$ and an edge in~$F$ for each cycle in $\mathcal{C}_{\ell}(H)$ corresponding to its set of $\ell$ edges in~$H$.
    Define a random subgraph~$F' \subseteq F$ by keeping each edge~$C$ with probability $p_C = n^{1/2} \omega(C) \le 1$ by~\eqref{eq:fracbound}.

    For every edge~$e \in H$ we have $\mathbb{E}[d_{F'}(e)] = n^{1/2} \sum_{C \in \mathcal{C}_\ell(H, e)} \omega(C) = n^{1/2}$.
    Moreover, since two distinct edges $e, f \in E(H)$ can participate together in at most~$O(n^{\ell - (k+1)})$ many $C_\ell^{\smash{(k)}}$ in $H$, \eqref{eq:fracbound} implies that the expected $2$-degree is bounded by~$\mathbb{E}[ d_{F'}(e, f)] = O(n^{-1/2})$.
    Using standard concentration inequalities we get that with high probability $d_{F'}(e) = (1 + o(1))n^{1/2}$ for each $e \in V(F')$ and that~$\Delta_2(F') = O(\log n )$.
    This means that
    \begin{align*}
    \delta_1(F') \ge (1 - o(1))\Delta_1(F')\text{, }g(F')=\Omega(n^{1/2}/\log n) \text{ and } g(F')=O(n^{1/2}).
    \end{align*}
    
    For each $(k-1)$-set~$S$ of vertices of $H$, let $U_{S} \subseteq V(F)$ correspond to the edges in $H$ containing $S$.
    There are at most $n^{k-1}$ such sets and each has size at least~$\eps n$.
    Thus, it is easy to check that the conditions for Theorem~\ref{thm:alonyuster} are satisfied. 
    Therefore, there is a matching $M$ in $F'$ such that at most~$\gamma n$ vertices in $V(F')$ are uncovered in each~$U_{S}$.
    The matching~$M$ in~$F' \subseteq F$ translates to a $C^{\smash{(k)}}_{\ell}$-packing~$\mathcal{C}$ in~$H$ and the latter condition implies $\Delta_{k-1}(H - \bigcup \mathcal{C} ) \leq \gamma n$, as desired.
\end{proof}

\subsection{Extending lemma}

For this section we will use the following result (see~\cite[Theorem 5.5]{PigaSanhuezaMatamala2021}).

\begin{theorem} \label{thm:jain}
	Let $X_1, \dotsc, X_t$ be Bernoulli random variables (not necessarily independent) such that, for each~$i \in [t]$, we have $\mathbb{P} [X_i = 1 | X_1, \dotsc, X_{i-1}] \leq p_i$.
	Let $Y_1, \dotsc, Y_t$ be independent Bernoulli random variables such that $\mathbb{P} [Y_i = 1] = p_i$ for all~$i \in [t]$.
	Let~$X = \sum_{i \in [t]} X_i$ and $Y = \sum_{i \in [t]} Y_i$.
	Then $\mathbb{P} [X \ge k] \leq \mathbb{P} [Y \ge k]$ for all~$k \in \{0, 1, \dotsc, t \}$.
\end{theorem}

Let $\mathcal{S}$ be a multiset of ordered $(k-1)$-tuples in an $n$-vertex set $V$, possibly with repetitions.
We say that $\mathcal{S}$ is \emph{$\gamma$-sparse} if the multi-$(k-1)$-graph $S$ formed by all the unoriented $(k-1)$-sets from $\mathcal{S}$, counting repetitions, has $\Delta_{j}(S) \leq \gamma n^{k-j}$ for each~$0 \le j \le k-1$.
For instance,
the $j = 1$ case says that no vertex is in more than $\gamma n^{k-1}$ tuples (counting repetitions).
Recall the definition of ends of a trail~$P$ and $D(P)$ in Section~\ref{sec:residualgraph}.

\begin{lemma}[Extending lemma] \label{lemma:extending}
	Let $1/n \ll \gamma \ll \mu \ll \eps, 1/\ell, 1/k$.
	Let $H$ be a $k$-graph on $n$ vertices.
    Let $\mathcal{S} = \{ \mathbf{a}_i, \mathbf{b}_i :  i \in [t] \} $ be a multiset of ordered $(k-1)$-tuples in~$V(H)$ such that
	\begin{enumerate}[label={\rm(\alph*)}]
		\item\label{item:extending-Fsparse} $\mathcal S$ is $\gamma$-sparse and
		\item \label{item:extending-manyextensions} for each $i \in [t]$, there are at least $\eps n^{\ell}$ trails~$P$ in~$H$ on $\ell+2(k-1)$ vertices such that $D(P) = \{ \mathbf{a}_i , \mathbf{b}_i \}$.
	\end{enumerate}
	Then, there exist edge-disjoint trails $P_1, \dotsc, P_t$ in $H$ such that, for each $i \in [t]$,
	\begin{enumerate}[label={\rm(\roman*)}]
		\item $P_i$ has $\ell+2(k-1)$ vertices and $D(P_i) = \{ \mathbf{a}_i , \mathbf{b}_i\}$, 
		\item the vertices of $P_i$ outside $\mathbf{a}_i$ and $\mathbf{b}_i$ are all distinct, and
		\item $\Delta_{k-1}( \bigcup_{ i \in [t] } P_i ) \leq \mu n$.
	\end{enumerate}
\end{lemma}

\begin{proof}
	The idea is to pick, sequentially, a trail~$P_i$ chosen uniformly at random among all the trails whose ends are $\mathbf{a}_i$ and $\mathbf{b}_i$.
	Since $\mathcal S$ is $\gamma$-sparse and there are plenty of choices for $P_i$ in each step, we expect that in each step the random choices do not affect the codegree of the graph formed by the yet unused edges in~$H$ by much.
	This will ensure that, even after removing the edges used by $P_1, \dotsc, P_{i-1}$, there are still many trails $P_i$ available for the $i$th step.
	If all goes well, then we can continue the process until the end, thus finding the required trails.
	
	We say that a trail~$P$ is \emph{$i$-good} if $P$ is on  $\ell + 2(k-1)$ vertices, $D(P) = \{ \mathbf{a}_i , \mathbf{b}_i \}$ and  the vertices of~$P$ outside $\mathbf{a}_i$ and $\mathbf{b}_i$ are all distinct.
	Let $\mathcal{P}_i(H)$ be the set of all~$i$-good trails in~$H$.
	We begin by noting that $\mathcal{P}_i(H)$ is large.
	Indeed, there are at most~$\ell^2 n^{\ell - 1} \leq \eps n^{\ell}/2$ trails~$P$ on $\ell + 2(k-1)$ vertices with $D(P) = \{ \mathbf{a}_i , \mathbf{b}_i \}$ that is not~$i$-good.
	By~\ref{item:extending-manyextensions}, $|\mathcal{P}_i(H)| \geq \eps n^{\ell}/2$.
	Since $\mu\ll \eps$, we have 
	\begin{align}
		\text{if $G$ is a $k$-graph with $\Delta_{k-1}(G) \leq \mu n$, then $|\mathcal{P}_i(H-G)| \ge \eps n^{\ell}/3$.} \label{claim:extending-sparseremoval}
	\end{align}
	
	We now describe the random process. 
	For each $i \in [t]$, assume we have already chosen edge-disjoint $P_1, P_2, \dotsc, P_{i-1} \subseteq H$, and we describe the choice of~$P_i$.
	Let~$G_{i-1} = \bigcup_{j \in [i-1]} P_j$ correspond to the edges of~$H$ used by the previous choices of~$P_j$, which we need to avoid when choosing~$P_i$
	(note that $G_0$ is empty).
	If $\Delta_{k-1}(G_{i-1}) \leq \mu n$, then \eqref{claim:extending-sparseremoval} implies that $|\mathcal{P}_i(H - G_{i-1})| \ge \eps n^{\ell}/3$ and we take~$P_i\in \mathcal{P}_i(H-G_{i-1})$ uniformly at random.
	Otherwise, if $\Delta_{k-1}(G_{i-1}) > \mu n$, then let $P_i = \emptyset$.
	
	In any case, the process outputs a collection of edge-disjoint subgraphs~$P_1, \dotsc, P_{t}$.
	Our task now is to show that with positive probability, there is a choice of $P_1, \dotsc, P_{t}$ such that $\Delta_{k-1}(G_{t}) \leq \mu n$.
	This will imply also that $P_i \in \mathcal{P}_i$, which is what we needed.
	Formally, for each $i \in [t]$, let $\mathcal S_i$ be the event that $\Delta_{k-1}(G_{i}) \leq \mu n$.
	Thus it is enough to show $\mathbb{P}[\mathcal S_t] > 0$.
	
	Fix $e \in \binom{V(H)}{k-1}$.
	For each $i \in [t]$, let $X_i(e)$ be the random variable that takes the value $1$ precisely if $e$ belongs to an edge of $P_i$, and $0$ otherwise.
	Equivalently, $X_i(e) = 1$ if and only if $\deg_{P_i}(e) \ge 1$.
	Since $\Delta_{k-1}(P_i) \leq 2$ for each $i \in [t]$, we have
	\begin{align}
		\deg_{G_i}(e) \leq 2 \sum_{ j \in [i] } X_j(e)\,. \label{equation:degreeui}
	\end{align}
	For each $i \in [t]$, define 
	    \begin{align*}
	        r_i(e) & = \max \{ |e \cap \mathbf{a}_i|, |e \cap \mathbf{b}_i| \} 
	        \quad \text{ and } \quad
	        p^\ast_i(e)  = \min \left\{ 1, \frac{6 \ell k }{\eps n^{(k-1)-r_i(e) }}\right\},
	    \end{align*}
	where here $\mathbf{a}_i$, $\mathbf{b}_i$ are taken as the underlying $(k-1)$-sets.
	
	\begin{claim}
		For each $e\! \in\! \binom{V(H)}{k-1}$ and $i \in [t]$,
	$$
		 \mathbb{P} [X_i(e)=1 | X_1(e), X_2(e), \dotsc, X_{i-1}(e) ] \leq p^\ast_i(e). 
	$$
	\end{claim}
	
	\begin{proofclaim}[Proof of the claim.]
	    Fix $e \in \binom{V(H)}{k-1}$ and $i \in [t]$.
		Using conditional probabilities, we separate our analysis depending on whether $\mathcal S_{i-1}$ holds or not.
		If $\mathcal S_{i-1}$ fails, then~$P_i = \emptyset$ and so $X_i(e) = 0$ implying that our claim holds. 
		
		Now assume that $\mathcal S_{i-1}$ holds, so $\Delta_{k-1}(G_{i-1}) \leq \mu n$.
		By~\eqref{claim:extending-sparseremoval}, $P_i$ will be chosen uniformly at random from $\mathcal{P}_i(H-G_{i-1})$, which has size at least $\eps n^{\ell}/3$ regardless of the values of~$X_1(e), \dotsc, X_{i-1}(e)$.
		
		If $r_i(e) = k-1$, then $p^\ast_i(e) = 1$ and we are done. 
		We may assume that $r = r_i(e) \in [k-2] \cup \{ 0\}$.
		We now estimate the number of $P \in \mathcal{P}_i(H - G_{i-1})$ with $\deg_{P}(e) \ge 1$.
		If we have~$P= v_1 v_2 \dotsb v_{\ell + 2(k-1)}$ and $\deg_{P}(e) \ge 1$, then $j_0 = \min \{j : v_j \in e\} \in [\ell + k-1]$ and~$|\{ j \in [k] : v_{j_0+j} \notin e\}| = 1$.		
		Recall that, for each $P \in \mathcal{P}_i(H - G_{i-1})$, it holds that~$|V(P) \setminus \{ \mathbf{a}_i , \mathbf{b}_i\}| = \ell$ and it also holds that $| e \setminus \mathbf{a}_i|,| e \setminus \mathbf{b}_i| \ge k-1- r_i(e)$.
		Hence, we deduce that the number of  $P \in \mathcal{P}_i(H - G_{i-1})$ with $\deg_{P}(e) \ge 1$ is certainly at most~$(\ell + k-1)k n^{\ell - (k-1-r_i(e))} \le 2 \ell k n^{\ell - (k-1-r_i(e))}$.
		Thus we have 
		\begin{align*}
    		\mathbb{P} [X_i(e) = 1 | X_1(e),  \dotsc, X_{i-1}(e), \mathcal S_{i-1} ] 
    		&\leq 
    		\frac{  2 \ell k n^{\ell - (k-1-r_i(e))} }{ |\mathcal{P}_i(H-G_{i-1})|}
    		\leq
		    \frac{6 \ell k}{\eps n^{k-1-r_i(e)}} = p^\ast_i(e)\,,
		\end{align*}
		as required.
		This finishes the proof of the claim.
	\end{proofclaim}

	Now, we use that $\mathcal S$ is $\gamma$-sparse to argue that $\sum_{ i \in [t] } p_i^\ast(e)$ is small for each~$e \in \binom{V(H)}{k-1}$.
	Indeed, for each $0 \le r \le k-1$, let $t_r$ be the number of $i \in [t]$ such that $r_i(e) =r$.
	Since $\mathcal S$ is $\gamma$-sparse, we have $t_r \leq 2 \binom{k-1}{r} \gamma n^{k-r}$ for each $0 \le r \le k-1$.
	Recall that we are assuming the hierarchy~$\gamma \ll \mu \ll \eps, 1/\ell, 1/k$.
	Therefore, we have
	\begin{align}
		\sum_{ i \in [t] } p^\ast_i(e)
		= t_{k-1} + \sum_{0 \le r \le k-2} t_r \cdot \frac{6 \ell k }{\eps n^{k-1-r}}
		\leq
		\frac{\mu}{30} n. \label{equation:pibound}
	\end{align}

	We now claim that 
	\begin{align}
		\mathbb{P} \left[ \sum_{ i \in [t] } X_i(e) \geq \frac{\mu}{3} n \right] \leq \exp\left( - \frac{\mu}{3} n \right).
		\label{equation:extending-concentration}
	\end{align}
	Indeed, \eqref{equation:pibound} implies that $7 \sum_{ i \in [t] } p^\ast_i(e) \leq \mu n / 3$,
	so the bound follows from Theorem~\ref{thm:jain} combined with a Chernoff-type bound~\cite[Corollary 2.4]{JansonLuczakRucinski2000}.
	
	For each $e \in \smash{\binom{V(H)}{2}}$, let $X_e := \sum_{\smash{i \in [t]}} X_i(e)$.
	Let $\mathcal{E}$ be the event that $\max_e X_e \leq \mu n / 3$.
	By using a union bound over all the (at most $n^{k-1}$) possible choices of $e$ and using~\eqref{equation:extending-concentration}, we deduce that $\mathcal{E}$ holds with probability at least $1 - o(1)$.
	
	Now we can show that $\mathcal S_t$ holds with positive probability.
	In fact, we shall prove that~$\mathbb{P} [\mathcal S_t | \mathcal{E}] = 1$, which then will imply $\mathbb{P} [\mathcal S_t] \ge \mathbb{P} [\mathcal S_t | \mathcal{E}] \mathbb{P} [\mathcal{E}] \ge 1 - o(1)$.
	So assume~$\mathcal{E}$ holds, that is, $\max_e X_e \leq \mu n / 3$.
	Note that $\mathcal S_0$ holds deterministically, and suppose that~$i \in [t]$ is the minimum such that $\mathcal S_i$ fails to hold.
	Since $\mathcal S_{i-1}$ holds, using~\eqref{equation:degreeui} we deduce
	\begin{align*}
		\Delta_{k-1}( G_{i} )
		& \leq 2 + \Delta_{k-1}(G_{i-1}) = 2 + \max_{e} \deg_{G_{i-1}}(e)
		\leq 2 \left ( 1 + \max_e \sum_{j \in [i-1]} X_i(e) \right) \\
		& \leq 2 \left ( 1 + \max_e X_e \right) \leq 2 \left(1 + \frac{\mu}{3}n \right) \leq \mu n,
	\end{align*}
	where in the penultimate inequality we used $\mathcal{E}$,
	and in the last inequality we used $1/n \ll \mu$.
	Thus $\mathcal S_i$ holds, a contradiction.
\end{proof}

The following corollary of Lemma~\ref{lemma:extending} allows us to find a sparse path-decomposable subgraph whose removal adjusts the degrees modulo $k$.
This was used in proving Corollary~\ref{cor:lowerbound}.

\begin{corollary} \label{corollary:degreeadjuster}
    Let $0< 1/n \ll \mu \ll 1/\ell,1/k, \eps$ with $\ell > k \ge 3$.     
    Let $H$ be a $k$-graph on~$n$ vertices such that $\delta^{(2)}(H) \geq \eps n$.
    Then there exists a $P^{\smash{(k)}}_{\ell}$-decomposable subgraph~$H'$ such that
    \begin{enumerate}[label=\upshape{(\roman*)}]
        \item $|H'| \leq \ell^2 k n$,
        \item $\Delta_{k-1}(H') \leq \mu n$, and
        \item for each $x \in V(H)$, we have $\deg_{H - H'}(x) \equiv 0 \bmod k$.
    \end{enumerate}
\end{corollary}

\begin{proof}
    We start by finding a ($2$-uniform) multidigraph $D$ on $V = V(H)$ such that $d^+_D(v) - d^-_D(v) + \deg_{H}(v) \equiv  0 \bmod k$ holds for each $v \in V$.
    This can be constructed greedily, starting from an empty digraph $D$.
    As long as there is a pair of vertices $u, v$ and $0 < i \leq j < k$ with $d^+_D(u) - d^-_D(u) + \deg_{H}(u) \equiv i \bmod k$ and $d^+_D(v) - d^-_D(v) + \deg_{H}(v) \equiv j \bmod k$, we pick them by minimising $i$ and maximising $j$,
    and then we add the directed edge~$u \rightarrow v$ to~$D$.
    Since $\sum_{x \in V(H)} \deg_{H}(x) = k |H| \equiv 0 \bmod k$, this process is guaranteed to end.
    By construction, we have $d^+_D(v), d^-_D(v) \leq k$ for each $v \in V$, and $D$ has at most $kn$ arcs.
    
    Let $\ell_0$ be the minimum integer divisible by~$\ell$ such that $\ell_0 \geq k^2 - k + 2$.
    We clearly have the inequality~$\ell_0 \leq \ell^2$.
    Given vertices $u, v \in V$, suppose that $T = T_{u,v} \subseteq H$ is a~$P^{\smash{(k)}}_{\ell}$-decomposable subgraph on $\ell_0$ edges such that $\deg_{T}(u) = k-1$, $\deg_{T}(v) = 1$ and $\deg_{T}(w) \equiv 0 \bmod k$ for all other vertices.
    Suppose we can find an edge-disjoint collection~$\mathcal{T}$ of such subgraphs $T_{u,v}$, one for each edge $u \rightarrow v$ in~$D$, with the additional condition that the union~$H'$ of those subgraphs has codegree at most~$\mu n$.
    Then $H'$ is easily seen to satisfy the required conditions.
    We now describe the construction of such a family.
    
    Each $T_{u,v}$ will be chosen as follows.
    Given $uv \in E(D)$,
    we pick a $(k-2)$-tuple of vertices $\mathbf{x}(u,v) = x_2 \dotsb x_{k-1} \in (V \setminus \{u,v\})^{k-2}$, uniformly at random.
    Then, we consider the $(k-1)$-tuples $\mathbf{v}_{u,v} = u x_{k-1} \dotsb x_{2}$ and $\mathbf{w}_{u,v} = x_2 \dotsb x_{k-1} v$. 
    Note that a trail with ends $\mathbf{v}_{u,v}$ and $\mathbf{w}_{u,v}$ using $\ell_0$ edges and no new repeating vertices forms a $T_{u,v}$ with the required characteristics.
    In particular, such a $T_{u,v}$ has a $P^{(k)}_\ell$-decomposition.
    
    Consider the multiset of ordered $(k-1)$-tuples $\mathcal{Q} = \bigcup_{\smash{uv \in E(D)}} \{ \mathbf{v}_{u,v}, \mathbf{w}_{u,v} \}$.
    Since the bounds $\Delta^+(D) , \Delta^-(D) \leq k$ hold and $\mathbf{x}(u,v)$ was chosen at random for each directed edge~$uv \in E(D)$, we can assume that $\mathcal{Q}$ is $\gamma$-sparse.
    Select a new constant $\rho$ which satisfies the hierarchy~$\mu \ll \rho \ll \eps$.
    By Lemma~\ref{lem:findcycle}, for each $uv \in E(D)$, there exist~$\rho n^{\ell_0 - k + 1}$ trails with $\ell_0$ edges and ends $\mathbf{v}_{u,v}$ and $\mathbf{w}_{u,v}$.
    Then, Lemma~\ref{lemma:extending} (with $\rho$ in place of $\eps$) provides us with an edge-disjoint collection of trails $\{ T_{uv} : uv \in E(D) \}$, one for each $uv \in E(D)$, such that $T_{uv}$ has ends $\mathbf{v}_{u,v}$ and $\mathbf{w}_{u,v}$, no repeated vertices save for those in the ends, and $H' = \bigcup_{uv \in E(D)} P_{uv}$ satisfies $\Delta_{k-1}(H') \leq \mu n$, which is all we needed.
\end{proof}

\subsection{Cover-down lemma: Proof of Lemma \ref{lem:pathtocover}}\label{subsec:coverdownproof}

For this section, we will require a few pieces of new notation. Given a $k$-graph $H$, a vertex set $U\subseteq V(H)$, a~$(k-1)$-tuple~$e\in \smash{\binom{V(H)}{k-1}}$, and a set of~$(k-1)$-tuples $G\subseteq \smash{\binom{V(H)}{k-1}}$, define $N_H(e, U; G) = N(e,U)\cap G$.
Moreover, define $$\delta^{(2)}(H; U, G) =\min\left\{N_H(e_1, U; G)\cap N_H(e_2, U; G) \colon e_1, e_2\in \binom{V(H)}{k-1}\right\}\,.$$

\begin{proof}[Proof of Lemma~\ref{lem:pathtocover}]
	Given~$k\in \mathbb N$ and~$\alpha>0$ take~$\ell_0\in \mathbb N$ larger than the one given by~\hyperref[statement:pathdec-k]{\upshape{($\curvearrowright_{k-1}$)}} and sufficiently large for an application of Lemma~\ref{lem:wellbehaved}. 
	Moreover, for~$\mu>0$ we take auxiliary variables~$\gamma$,~$p_i$ and $\mu_i$ for every~$i\in [k-1]$, under the following hierarchy
    $$0<\gamma \ll \mu_1\ll p_1 \ll \dots \ll \mu_{k-1} \ll p_{k-1} \ll \mu_k\ll \mu,\alpha\,.$$ 
    Finally take~$\gamma_i=\gamma+2p_i\ll \mu_{i+1}$ and~$\alpha_{i}=p_i\alpha^2/2-\sum_{0 \le j \le i-1} \mu_j \gg \mu_{i}$.  
    Let~$n\in\mathbb N$ be sufficiently large and let $H$ be as in the statement of the lemma. 

    \smallskip
 	\noindent \textit{Step 1: Setting the stages.} 
	For every~$0 \le i \le k-1$, let $H_i=\{e\in H\colon |e\cap U| = i\}$ and let~$R_i\subseteq H_i$ be defined by choosing edges independently at random from~$H_i$ with probability~$p_i$.  
	Moreover, let~$R_{\geq i}=\bigcup_{i \le j \le k-1}R_j$.
	Considering~\ref{it:CDdegree}, by standard concentration inequalities we have that with non-zero probability the following inequalities happen simultaneously: for every~$0 \le i \le k-1$,
	\begin{align}
	    \Delta_{k-1}(R_i) &\leq 2p_i n \, ,\label{eq:CDmaxdeg}\\
	    \delta^{(2)}(R_{\geq i}\cup H[U]; U, G_{i-1})&\geq \frac{p_{i}\alpha \vert U \vert}{2} \geq \frac{p_{i}\alpha^2n}{2}\,,\label{eq:CDmindeg}
	\end{align}
	where $G_{i}=\{ e\in \binom{V}{k-1} \colon \vert e\cap U\vert \geq i \}$ (we include the degenerate cases~$G_{-1}\!=\!G_{0}\!=\!\binom{V}{k-1})$.
	From now on for every~$0 \le i \le k-1$ we consider~$R_i$ to be a fixed graph with those properties.
 	
 	Define~$H^\star= H-H[U]-R_{\geq 0}$ and observe that~$\delta^{(2)}(H^\star)\geq \alpha n/2$. 
 	Hence we can apply Lemma~\ref{lem:wellbehaved} to find a $C^{\smash{(k)}}_\ell$-packing~$\mathcal C$ in~$H^{\star}$ such that~$\Delta_{k-1}(H^\star - \bigcup 	\mathcal C) \leq \gamma n$. 
 	We shall find a~$C^{\smash{(k)}}_\ell$-packing that covers the leftover $J = H^\star- \bigcup \mathcal C$ and the graph~$R_{\geq 0}$. 
 	We do this in stages, covering the edges $J_i=(J\cap H_i)\cup R_i$ (and some from~$R_{\geq i})$ in each stage.
 	
 	\smallskip
 	\noindent \textit{Step 2: The first $k-1$ stages.} 
 	To start, let~$\mathcal C_{-1}=J_{-1}=\emptyset$.
 	Let~$0\leq\!i\!<\!k-2$ and denote the edges which were covered in previous stages by~$J_{\leq i-1}=\bigcup_{0 \le j \le i-1} J_j$.
 	Suppose there is a~$C^{\smash{(k)}}_\ell$-packing~$\mathcal C_{i-1}$ such that
 	{\thinmuskip=0.8mu 
    \medmuskip=1mu
    \thickmuskip=2mu
 	\begin{align}\label{eq:CDpacking}
 	     \bigcup \mathcal C_{i-1}\cap H[U] =\emptyset,
 	     \quad\!
 	     J_{\leq i-1} \subseteq \bigcup \mathcal C_{i-1}\,, 
          \quad\!\text{and}\quad\! 
          \Delta_{k-1}\big(\bigcup \mathcal C_{i-1}- J_{\leq i-1}\big)\leq \hspace{-0.1cm}\sum_{0 \le j \le i} \mu_j n\,.
 	\end{align}}
 	Note that~\eqref{eq:CDpacking} holds vacuously for~$i=0$.
 	We shall prove the existence of a packing~$\mathcal C_{i}$ satisfying~\eqref{eq:CDpacking} for~$i$ instead of~$i-1$.
 	
 	Let~$\tilde R_{\geq i+1}$ and $\tilde J_{i}$ be the remaining edges from~$R_{\geq i+1}$ and~$J_{i}$ after deleting~$\bigcup \mathcal C_{i-1}$. 
    More precisely let~$\tilde R_{\geq i+1}=R_{\geq i+1}- \bigcup \mathcal C_{i-1}$ and~$\tilde J_{i}=J_{i} - \bigcup \mathcal C_{i-1}$. 
    Because of~\eqref{eq:CDmindeg} and~\eqref{eq:CDpacking} we have that 
    \begin{align}\label{eq:CDmindegi}
        \delta^{(2)}(\tilde R_{\geq i+1}; U, G_{i}) 
        \geq 
        \Big(\frac{p_{i+1}\alpha^2}{2} - \sum_{0 \le j \le i+1} \mu_j\Big)n = \alpha_{i+1} n 
        \gg 
        \mu_{i+1}n\,.
    \end{align}
 	Moreover, in view of \eqref{eq:CDmaxdeg}, we obtain
 	\begin{align}\label{eq:CDmaxdegi}
 	    \Delta_{k-1}(\tilde J_{i})
 	    \leq
 	    \Delta_{k-1}(J) + \Delta_{k-1}(R_{i}) \leq \gamma n + 2p_{i}\alpha n 
 	    =
 	    \gamma_{i} n\,.
 	\end{align}
 	Enumerate edges of $\tilde  J_i$ into $e_1,\dots, e_t$. 
 	For each $j \in [t]$, we oriented $e_j$ arbitrarily and let $\{\mathbf{a}_j,\mathbf{b}_j\}$ be such that $D(e_j) = \{\mathbf{a}_j^{-1},\mathbf{b}_j^{-1}\}$.
 	Note that $\mathcal{S} = \{\mathbf{a}_j,\mathbf{b}_j : j \in [t] \}$ is~$\gamma_{i}$-sparse. 
 	Moreover,~\eqref{eq:CDmindegi} and~Lemma~\ref{lem:findcycle} implies, for each $j \in [t]$, $\tilde R_{i+1}$ contains at least~$\alpha_{i+1} n^{\ell-k}$ trails~$P$ on $\ell+k-2$ vertices such that $D(P) = \{ \mathbf{a}_j , \mathbf{b}_j \}$.
    We apply Lemma~\ref{lemma:extending} with~$\alpha_{i+1}, \mu_{i+1}, \gamma_{i}, \ell-k, \tilde{R}_{i+1}$ in the rôles of $\alpha, \mu, \gamma, \ell,H$ to obtain edge-disjoint trails $P_1, \dots, P_t$ in $\tilde{R}_{i+1}$ such that, for each $j \in [t]$, 
    \begin{enumerate}[label={\rm(\roman*)}]
 		\item $P_j$ has $\ell+k-2$ vertices and $D(P_j) = \{ \mathbf{a}_j , \mathbf{b}_j\}$,
		\item the vertices of $P_j$ outside $\mathbf{a}_j$ and $\mathbf{b}_j$ are all distinct, and
		\item $\Delta_{k-1}( \bigcup_{ j \in [t] } P_j ) \leq \mu_{i+1} n$.
 	\end{enumerate}
 	Note that $e_i \cup P_i$ is $C_{\ell}^{(k)}$, so $ \tilde  J_i \cup \bigcup_{j \in [t]}P_j$ has a $C_{\ell}^{(k)}$-decomposition~$\mathcal{C}'_i$.
 	It is easy to see that by taking~$\mathcal C_{i}=\mathcal C_{i-1} \cup \mathcal{C}'_i$ we obtain a~$C_{\ell}^{(k)}$-packing satisfying~\eqref{eq:CDpacking} with~$i$ instead of~$i-1$. 
 	
    \smallskip
 	\noindent \textit{Step 3: The last stage.} 
 	For the last stage, a few changes are needed.
 	This is because in the previous stages we used edges from~$H_{i+1}$ to extend paths in~$H_i$, which is no longer possible at this stage.
 	Instead, we rely on the path decompositions ensured by~\hyperref[statement:pathdec-k]{\upshape{($\curvearrowright_{k-1}$)}}.
 	
 	As before, we define~$\tilde J_{k-1}=J_{k-1}-  \bigcup\mathcal C_{k-2}$. 
 	For every vertex~$v\in V(H)\setminus U$, we let~$F(v)=\{e\setminus \{v\} \in \binom{V}{k} \colon v\in e\in \tilde J_{k-1}\}$ be the link graph of~$v$ in the hypergraph~$\tilde J_{k-1}$.
 	Note that~$F(v)$ is completely contained in~$U$.
 	We shall apply \hyperref[statement:pathdec-k]{\upshape{($\curvearrowright_{k-1}$)}} to find a $P^{(k-1)}_k$-decomposition in~$F(v)$.
 	For this, we first prove that~$\vert F(v)\vert = \deg_{\tilde J_{k-1}}(v)$ is divisible by~$k$.
 	Indeed,~\ref{it:CDdiv} says that~$\deg_H (v)$ is divisible by~$k$, and since~$\tilde J_{k-1}=H-H[U]- \bigcup \mathcal C- \bigcup \mathcal C_{k-2}$ we have~$\deg_{\tilde J_{k-1}}(v)$ is divisible by~$k$ as well. 
 	Moreover, because of \eqref{eq:CDmindeg} and~\eqref{eq:CDmindegi} we have that  
 	$$\delta^{(2)}(F(v))\geq \frac{ p_{k-1}\alpha^2}{2}n - \!\!\sum_{0 \le j \le k-1} \mu_{j} n \geq \alpha_{k-1}n\,.$$
 	Hence, \hyperref[statement:pathdec-k]{\upshape{($\curvearrowright_{k-1}$)}} yields a $P_{k}^{\smash{(k-1)}}$-decomposition of~$F(v)$.
 	Notice that each path in this decomposition corresponds to a~$P^{\smash{(k)}}_{k+1}$ in $\tilde J_{k-1}$ when we include the vertex~$v$ in every edge. 
 	Call this~$P^{\smash{(k)}}_{k+1}$-packing~$\mathcal P_v$ and observe that paths from~$\mathcal P_v$ and $\mathcal P_u$ are edge-disjoint for every~$u\neq v$. 
 	This means~$\mathcal P=\bigcup_{\smash{v\in V(H)\setminus U}} \mathcal P_v$ is a~$P^{\smash{(k)}}_{k+1}$-decomposition of~$\tilde J_{k-1}$. 
 
    Now we continue as in the previous stages and observe that 
    \begin{align*}
    \Delta_{k-1}(\tilde J_{k-1}) \leq \Delta_{k-1}(J) + \Delta_{k-1}(R_{k-1}) \leq \gamma n + 2p_{k-1}\alpha n =\gamma_{k} n \,,
    \end{align*}    
    which implies that~$D(\mathcal P)$ (without simplification) is~$\gamma_k$-sparse. 
    Moreover, \eqref{eq:CDpacking} implies
    \begin{align*}
        \delta^{(2)}((H- \mathcal C_{k-2}) [U]) = \delta^{(2)}(H[U]) \geq \alpha \vert U \vert\,.
    \end{align*}
    By Lemma~\ref{lem:findcycle}, for any $P \in \mathcal{P}$ with $D(P) = \{\mathbf{a}^{-1}, \mathbf{b}^{-1}\}$, $(H- \mathcal C_{k-2}) [U]$ contains at least $\alpha n^{\ell-k-1}$ trails~$Q$ on $\ell+k-3$ vertices such that $D(Q) = \{ \mathbf{a} , \mathbf{b} \}$.
    Finally, we apply Lemma \ref{lemma:extending} as in the previous stages to obtain edge-disjoint trails $\{ Q_P : P \in \mathcal{P}\}$ in~$(H- \mathcal C_{k-2}) [U]$ such that, for each $P \in \mathcal{P}$, 
    \begin{enumerate}[label={\rm(\roman*)}]
 		\item $Q_P$ has $\ell+k-3$ vertices and $D(Q_P) = \{ \mathbf{a} , \mathbf{b} \}$ such that $D(P) = \{ \mathbf{a}^{-1} , \mathbf{b}^{-1} \}$,
		\item the vertices of $Q_P$ outside $D(Q_P)$ are all distinct, and
		\item $\Delta_{k-1}( \bigcup_{ P \in \mathcal{P}} Q_P ) \leq \mu_{k} n$.
 	\end{enumerate}
    Each $P \cup Q_P$ forms a $C^{\smash{(k)}}_{\ell}$, so $\tilde{J}_{k-1} \cup \bigcup_{P \in \mathcal{P}} Q_P = \bigcup_{P \in \mathcal{P}} (P \cup Q_P)  $ has a $C^{\smash{(k)}}_\ell$-decomposition~$\smash{\mathcal{C}_{k-1}'}$. 
     Thus, recalling~\eqref{eq:CDpacking}, it is easy to see that the~$C_\ell^{\smash{(k)}}$-packing $\mathcal C^\star=\mathcal C\cup \mathcal C_{k-2} \cup \mathcal{C}_{k-1}'$ satisfies the requirements of the lemma. 
\end{proof}

\section{Eulerian Tours} \label{section:euleriantours}

We first show that a lower bound of (essentially) $n/2$ on the codegree of $k$-graphs is necessary to ensure that every edge is in some tight cycle. 
The bound is asymptotically tight by Lemma~\ref{lem:findcycle} (which can be used to find cycles which contain any given edge).
This also provides the lower bound in Theorem~\ref{cor:main}.

\begin{proposition} \label{prop:cyclecover}
For all $k \ge 3$ and $m \geq 2$, there exists a $k$-graph~$H$ on~$n = 2mk$ vertices with $\delta(H) \ge n/2 - 2k+1 $ such that $\deg(v)$ is divisible by $k$ for all $v \in V(H)$ and there is an edge that is not contained in any tight cycle. 
In particular, we have the bounds $\delta^{(k)}_{\smash{\mathrm{cycle}}},\delta^{(k)}_{\smash{\mathrm{Euler}}} \ge 1/2$.
\end{proposition}

\begin{proof}
Let~$A$ and~$B$ be disjoint vertex-sets each of size~$mk$. 
Recall that, for~$0\leq i\leq k$, we defined $H_i=H^{(k)}_i(A,B)$ as the $k$-graph with vertex set $A \cup B$ such that $e \in H_i$ if and only if $|e \cap B| = i$.
Consider the $k$-graph
$$H^\star = \!\!\bigcup_{\substack{ i \in ( \{0\} \cup [k]) \setminus \{1,k-1\}}} \!\!H_i^{(k)}(A,B)\,,$$
and observe that~$\delta(H^\star)\geq n/2-k+1$.
Note that each vertex has the same vertex-degree.
By removing at most $k-1$ perfect matchings in each of $H^\star[A] = H_0(A,B)$ and $H^\star[B] = H_k(A,B)$,
we may assume that each vertex has vertex-degree divisible by~$k$.
Additionally, remove edges $a_1\dotsb a_k \in H^\star[A]$ and $b_1 \dotsb b_k \in H^\star[B]$ and add two edges $a_1 \dotsb a_{k-1} b_k$ and $b_1 \dotsb b_{k-1} a_k$.
Call the resulting graph~$H$. 
Note that the bound~$\delta(H) \ge n/2 - 2k+1$ holds,
and for every vertex $v$, $d_H(v)$ is divisible by $k$.

We now claim that the edge $a_1 \dotsb a_{k-1} b_k$ is not contained in any tight cycle. 
Indeed, for $k = 3$, note that $\deg_H(a_1 b_3) = \deg_H(a_2b_3) =1$, so $a_1a_2b_3$ can only be the end of any tight path implying that $a_1a_2b_3$ is not contained in any tight cycle.
Now assume that $k \ge 4$. 
Since $a_1 \dotsb a_{k-1} b_k$ is the only edge in $H \cap H^1(A,B)$ (i.e. with exactly $k-1$ vertices in~$A$) any tight path of length at least~$k+1$ containing $a_1 \dotsb a_{k-1} b_k$ as a second edge must travel from $H^0(A,B)$ to $\bigcup_{i \ge 2} H^i(A,B)$. 
However, there is no other edge in $H \cap H^1(A,B)$ to close such a tight path into a cycle.

Since we have ensured every degree in $H$ is divisible by $k$, this construction shows that $\delta^{(k)}_{\textrm{cycle}},\delta^{(k)}_{\textrm{Euler}} \ge 1/2$.
\end{proof}

We split the other inequalities in Theorem~\ref{cor:main} into several lemmata.

\begin{lemma} \label{lem:euler1}
For $k \ge 3$, $\delta^{(k)}_{\mathrm{Euler}} \leq \delta^{(k)}_{\mathrm{cycle}}$.
\end{lemma}

\begin{proof}
Let $\ell = k^2$ and $k \ge 3$. 
Let $0 < 1/n \ll \gamma \ll \mu \ll \eps $. 
Let $H$ be a $k$-graph on $n$ vertices with $\delta(H) \ge \left( \delta^{(k)}_{\smash{\mathrm{cycle}}} + \eps \right) n $ such that $\deg_H(v)$ is divisible by~$k$ for all~$v \in V(H)$.
Note that $\delta(H) \geq (1/2 + \eps)n$ by Proposition~\ref{prop:cyclecover}.
Let $\sigma_1, \dots, \sigma_{t}$ be an enumeration of all ordered $(k-1)$-tuples of~$V(H)$, so $t = n!/(n-k+1)!$. 
For each $i \in [t]$, let~$\mathbf{a}_i = \sigma_{i}$ and $\mathbf{b}_i = \sigma_{i+1}^{-1}$, with indices taken modulo~$t$.
Let $\mathcal{S} = \{ \mathbf{a}_i, \mathbf{b}_i : i \in [t]\} $ be the multisets. 
Note that $\mathcal{S}$ is $\gamma$-sparse. 
By Lemma~\ref{lem:findcycle}, for all $i \in [t]$, $H$ contains at least $\eps n^{\ell}$ trails~$P$ on $\ell+2(k-1)$ vertices such that $D(P) = \{ \mathbf{a}_i , \mathbf{b}_i \}$.
Apply Lemma~\ref{lemma:extending} to obtain edge-disjoint trails $\{ P_i : i \in [t] \}$ in~$H$ such that, for each $i \in [t]$, 
    \begin{enumerate}[label={\rm(\roman*)}]
 		\item $P_i$ has $\ell+2 (k-1)$ vertices and $D(P_i) = \{ \mathbf{a}_i , \mathbf{b}_i \}$;
		\item the vertices of $P_i$ outside $D(P_i)$ are all distinct and
		\item $\Delta_{k-1}( \bigcup_{ i \in [t] } P_i ) \leq \mu n$.
 	\end{enumerate}
 Let $\mathcal{P} = \bigcup_{i \in [t]} P_i$, and note that (after joining trails) we obtain a tour in~$H$. 
 Consider the~$k$-graph~$H' = H - \mathcal{P}$.
 Note that  $\deg_{H'}(v)$ is divisible by~$k$ for all~$v \in V(H')$ and~$ \delta(H') \ge \delta (H) - \mu n \ge (  \delta^{(k)}_{\smash{\mathrm{cycle}}} + \eps/2 ) n $.
 Thus there is a cycle-decomposition~$\mathcal{C}$ of~$H'$.
 By attaching each cycle to the tour~$\mathcal{P}$, we obtain an Eulerian tour in~$H$.
 Hence we obtain~$\delta^{(k)}_{\smash{\mathrm{Euler}}} \le \delta^{(k)}_{\smash{\mathrm{cycle}}}$, as desired.
\end{proof}

\begin{lemma}\label{lem:euler2}
For $k \ge 3$, $\delta^{(k)}_{\smash{\mathrm{cycle}}} \leq \delta^{(k)}_{\smash{\mathrm{Euler}}}$.
\end{lemma}

\begin{proof}[Proof (sketch)]
    Let $\delta =  \delta^{(k)}_{\smash{\mathrm{Euler}}}$, by Proposition~\ref{prop:cyclecover} we have $\delta \geq 1/2$.
    Given $\eps > 0$, let~$n$ be sufficiently large and let $H$ be a $k$-graph on $n$ vertices with $\delta(H) \geq (\delta + 2 \eps)n$ with all vertex-degree divisible by $k$.
    It is enough to show that $H$ is decomposable into cycles.
    
    The idea is to use the iterative absorption framework.
    Indeed, since $\delta \geq 1/2$, we have $\delta^{(2)}(H) \geq 4 \eps n$.
    Thus there exists $\ell$ large enough (depending on $\eps$ only) such that the Vortex lemma (Lemma~\ref{lemma:vortex}) and the Cover-down lemma (Lemma~\ref{lem:coverdownlemma}) work in this setting.
    Thus it is possible to find a vortex $U_0 \supseteq U_1 \supseteq \dotsb \supseteq U_t$ to find a $C^{\smash{(k)}}_\ell$-packing which cover all edges except but those located in $U_t$.
    In fact, we can assume that the leftover $F \subseteq H[U_t]$ satisfies $\delta(F) \geq (\delta + \eps)|U_t|$ (see the proof of Theorem~\ref{theorem:main} in Section~\ref{section:mainproof} for detailed calculations to make these two steps work).
    The only missing step is the construction of an absorber for such a constant-sized leftover.
    
    The key observation here is that since the leftover $F$ will satisfy $\delta(F) \geq (\delta + \eps)|U_t|$, we can assume that $F$ admits an Euler tour.
    Since an Euler tour admits an edge-bijective homomorphism from a cycle, we can easily build a cycle-decomposable transformer between such a leftover and a cycle, and this step requires only $\delta^{(2)}(H) \geq \eps n$ (this is exactly what is done in the proof of Lemma~\ref{lem:transformer}).
\end{proof}

\begin{lemma} \label{lemma:cycledecvsinf}
    For $k \geq 3$, $\delta^{(k)}_{\smash{\mathrm{cycle}}} \le \inf_{\ell > k} \{\delta_{\smash{C_\ell^{(k)}}}\}$.
\end{lemma}

\begin{proof}
    Let $\delta = \inf_{\ell > k} \{\delta_{\smash{C_\ell^{(k)}}}\}$.
    Note that $\delta \geq 1/2$ by Proposition~\ref{prop:cyclecover}.
    
    Let $1/n \ll \eps$ and let $H$ be a $k$-graph on $n$ vertices with $\delta(H) \geq (\delta + 3 \eps)n$ and every degree divisible by $k$.
    By the definition of infimum, there exists $\ell$ (depending on $\eps$ only) such that $\delta_{\smash{C_\ell^{(k)}}} \leq \delta + \eps$.
    Since $\delta(H) \geq (1/2 + 2 \eps)n$, we can use Lemma~\ref{lem:findcycle} to find a cycle $C$ whose removal leaves a number of edges divisible by $\ell$.
    Thus $\delta(H - C) \geq (\delta + 2 \eps)n \geq (\delta_{\smash{C_\ell^{(k)}}} + \eps)n$, and therefore $H - C$ admits a $C_\ell^{\smash{(k)}}$-decomposition.
    Together with $C$, this is a cycle decomposition of $H$.
\end{proof}

Theorem~\ref{cor:main} follows immediately from Lemma~\ref{lem:euler1}, Lemma~\ref{lemma:cycledecvsinf}, Proposition~\ref{prop:cyclecover} and Theorem~\ref{theorem:main}.

\section{Concluding remarks}\label{section:remarks}

Theorem~\ref{cor:main} and Theorem~\ref{theorem:main} show that, for all $k$ and sufficiently large $\ell$, the inequalities $ 1/2 \le \delta^{(k)}_{\smash{\textrm{Euler}}} = \delta^{(k)}_{\smash{\mathrm{cycle}}} \le \delta_{\smash{C^{(k)}_{\ell}}} \le 2/3$ are valid.
For $k =3$, the second and third authors~\cite{PigaSanhuezaMatamala2021} gave an example showing that $\delta^{(3)}_{\smash{\textrm{Euler}}} \ge 2/3$, and therefore, $\delta^{(3)}_{\smash{\textrm{Euler}}} = \delta^{(3)}_{\smash{\mathrm{cycle}}} = \delta_{\smash{C^{(3)}_{\ell}}} = 2/3$ for large~$\ell$.
However, we were unable to generalise the examples presented there for~$k \ge 4$.
Our best example (Proposition~\ref{prop:cyclecover}) gives us~$\delta_{\smash{\textrm{cycle}}}^{(k)} \geq 1/2$, so we suggest the following question.
\begin{question}
Does there exist $k \ge 4$ such that  $\delta^{(k)}_{\smash{\mathrm{cycle}}} >1/2$? 
\end{question}

We gave a new lower bound for the fractional $C_\ell^{\smash{(k)}}$-decomposition threshold $\delta^*_{\smash{C^{(k)}_{\ell}}}$ in Proposition~\ref{proposition:newlowerbound}.
Moreover, when $k / \gcd(\ell,k)$ is even or $\gcd(\ell,k) = 1$, we are able to calculate the value given by our bound in a explicit form (see Corollary~\ref{cor:newlowerbound}).
Is the construction given by Proposition~\ref{proposition:newlowerbound} best-possible? 
We would like to propose the following weaker question. 

\begin{question}
Given $k \ge 2$, does there exist $\ell_0$ such that, for all $\ell > \ell_0$ with $\ell \not \equiv 0 \mod{k}$, $\delta^*_{\smash{C^{(k)}_{\ell}}} \le \frac{1}{2} + \frac{1}{2(\ell - 1)}$?
\end{question}

When $k=2$, we believe that $\ell_0$ should be $1$, which also implies the Nash-Williams conjecture~\cite{Nash-Williams} on $\delta_{K_3}$ (c.f. \cite[Theorem~1.4]{BKLO2016}).

\subsection*{Acknowledgement}
We thank the referees for their detailed and helpful remarks.

\printbibliography

\end{document}